\documentclass[12pt]{amsart}

\usepackage[shortlabels]{enumitem}

\usepackage[colorlinks,linkcolor={blue},
citecolor={blue},urlcolor={red},]{hyperref}

\usepackage{}
\usepackage{amsfonts}
\usepackage{amsmath, amsrefs, amsthm, amssymb}
\usepackage[active]{srcltx}		
\usepackage{pdfsync}					
\usepackage{graphicx}
\usepackage{amsthm}
\usepackage{graphicx}
\usepackage{framed}
\usepackage{multicol,multienum}
\usepackage{graphicx} 
\usepackage{booktabs} 
\usepackage{mathrsfs}
\usepackage{tcolorbox}
\usepackage{bm}
\usepackage{subfigure}
\usepackage{epstopdf}

\newtheorem{theorem}{Theorem}[section]

\newtheorem{lemma}[theorem]{Lemma}

\theoremstyle{definition}
\newtheorem{definition}[theorem]{Definition}

\numberwithin{equation}{section}

\begin{document}

\title []{Random perturbations for the chemotaxis-fluid model with fractional dissipation: Global pathwise weak solutions}

\author{Lei Zhang}
\address{School of Mathematics and Statistics, Hubei Key Laboratory of Engineering Modeling  and Scientific Computing, Huazhong University of Science and Technology,  Wuhan 430074, Hubei, P.R. China.}
\email{lei\_zhang@hust.edu.cn (L. Zhang)}

\author{Bin Liu}
\address{School of Mathematics and Statistics, Hubei Key Laboratory of Engineering Modeling  and Scientific Computing, Huazhong University of Science and Technology,  Wuhan 430074, Hubei, P.R. China.}
\email{binliu@mail.hust.edu.cn (B. Liu)}


\date{\today}

\begin{abstract}
This paper considers a stochastically perturbed Keller-Segel-Navier-Stokes (KS-SNS) system arising from the biomathematics in two dimensions, where the diffusion of fluid is expressed by a fractional Laplacian with an exponent in $[1/2,1]$. Our main result demonstrates that, under appropriate assumptions, the Cauchy problem of the KS-SNS system has a unique global probabilistically  strong and analytically weak solution, which also confirms that the quadratic logistic source authentically contributes to the global existence of solutions. First, a three-layer approximate system is introduced, this system seems to be new in the studying of chemotaxis-fluid model and it enables one to construct approximate solutions in regular Hilbert spaces $H^s(\mathbb{R}^2)$. Second, to accomplish the convergence progressively, a series of crucial entropy-energy inequalities for the approximations are derived, whose derivation strongly depends on the fine structure of the system and requires a novel strategy to adapt  to the appearance of fractional dissipation and the unboundedness of domain. Third, based on these uniform bounds, we establish the existence of global martingale weak solutions by virtue of a stochastic compactness method. Finally, by applying the Littlewood-Paley theory, we show that the pathwise uniqueness holds for these martingale solution. As a by-product, a few interpolation inequalities in Besov spaces are obtained, which seems to be new and may have their own interest.
\end{abstract}

\maketitle
 \section{Introduction}

Chemotaxis refers to the directional movement of cells, such as bacteria, in response to chemical signals. A notable example is that bacteria often swim towards areas with higher concentrations of oxygen for survival.  One of the most renowned models in chemotaxis is the Keller-Segel (KS) model, which was developed by Keller and Segel \cite{keller1970,keller1971}. It has since become one of the most extensively studied models in mathematical biology. The interplay between cells and the surrounding fluid, where chemical substances are consumed, has been acknowledged in \cite{fujikawa1989fractal,dombrowski2004self,tuval2005}. These studies confirm that the density of bacteria and chemoattractants change with the motion of fluid. Consequently, the velocity field of fluid is influenced by both moving bacteria and external body forces.  To describe such a coupled biological phenomena, Tuval et al. \cite{tuval2005} introduced the following Keller-Segel-Navier-Stokes (KS-NS) system with rigorous precision:
\begin{equation}\label{dns}
\left\{
\begin{aligned}
&\partial_t n +u\cdot \nabla n   = \Delta n  - \textrm{div}(n \nabla c) ,\\
&\partial_t c+ u\cdot \nabla c    = \Delta c - nc ,\\
&\partial_t u+  (u\cdot \nabla) u  +\nabla P  = \Delta u + n\nabla \phi+f ,\\
& \textrm{div}  u  =0.
\end{aligned}
\right.
\end{equation}
In \eqref{dns}, the unknowns consist of $n=n(t, x): \mathbb{R}^{+} \times \mathbb{R}^d \rightarrow \mathbb{R}^{+}$, $c=c(t, x): \mathbb{R}^{+} \times \mathbb{R}^d \rightarrow \mathbb{R}^{+}$, $u(t, x): \mathbb{R}^{+} \times \mathbb{R}^d \rightarrow \mathbb{R}^d$ and $P=P(t, x): \mathbb{R}^{+} \times \mathbb{R}^d \rightarrow \mathbb{R}$, $d=2,3,$ which respectively denote the density of bacteria, concentration of substrate, velocity field of fluid and associated pressure.  The term  $n\nabla \phi$ in \eqref{dns}$_3$ describes the external force exerted by bacteria on fluid through a given gravitational potential $\phi=\phi(x)$. Additionally, $f = f(t,x)$ represents an external force that can be generated by various physical mechanisms such as gravity or centripetal forces.

Extensive research has been conducted on the qualitative characteristics of solutions to the KS-NS system \eqref{dns} in both bounded and unbounded domains, owing to its significant applications in biomathematics \cite{hillen2009user,arumugam2021keller}. To provide a comprehensive overview, we would like to highlight several works that are closely related to the study of \eqref{dns}.

In the context of bounded domains, Lorz \cite{lorz2010coupled} was the first to establish the local well-posedness of weak solutions to \eqref{dns} in both 2D and 3D cases. Since then, a series of significant results have been obtained, as seen in the works of Winkler \cite{winkler2012global,winkler2012global,winkler2016global}, Black and Winkler \cite{black2022global}, Ding and Lankeit \cite{ding2022generalized}, and the references cited therein. Recently, Winkler \cite{winkler2017far} proved that the weak solution constructed in \cite{winkler2016global} possesses further regularity properties and satisfies the concept of eventual energy solutions after some relaxation time. Furthermore, the same author demonstrated that the possibility of singularities occurring on small time-scales for weak energy solutions only arises on sets of measure zero \cite{winkler2022does}.

When studying bacteria and chemicals in large bodies of water, it is typically assumed that the spatial domain is unbounded, usually $\mathbb{R}^2$ or $\mathbb{R}^3$. Significant contributions have been made in this area by Duan, Lorz and Markowich \cite{duan2010global}, Liu  and  Lorz \cite{liu2011coupled}, Chae, Kang  and  Lee \cite{chae2014global}, Kang and  Kim \cite{kang2017existence}, Zhang and Zheng \cite{zhang2014global,zhang2021global}, Diebou Yomgne \cite{diebou2021well}, Lei et al. \cite{lei2022large}  and the works listed therein. Recently, Jeong and Kang \cite{jeong2022well} investigated local well-posedness and blow-up phenomena in Sobolev spaces for both partially and fully inviscid KS-NS system in $\mathbb{R}^d $ $(d\geq 2)$. Therein, a new weighted Gagliardo-Nirenberg-Sobolev type inequality is used as their main analytical tool.

Incorporating stochastic effects is crucial in creating mathematical models for complex phenomena in science that involve uncertainty. For instance, the evolution of viscous fluids is not only affected by the external force $n\nabla \phi$ caused by bacteria, but also by random sources from the environment. The presence of randomness can significantly impact the overall evolution of the viscous fluid \cite{flandoli2008introduction,duan2014effective,breit2018stochastically}. Consequently, numerous studies have been conducted on the stochastic Navier-Stokes equations, as evidenced in \cite{bensoussan1995stochastic,flandoli1995martingale,brzezniak20132d,dong2011ergodicity,breit2018local,
chen2019martingale,hofmanova2019non,hofmanova2021ill,chen2022sharp} and their cited references. Due to the widespread applications of random fluctuations in hydrodynamics, developing a stochastic theory for the KS-NS system coupled with perturbed momentum equations by random forces is essential. This paper will mainly focus on this subject.

To our knowledge, the first result in this direction was obtained by Zhai and Zhang \cite{zhai20202d}, in which the authors introduced and studied the Cauchy problem for the following stochastic Keller-Segel-Navier-Stokes (KS-SNS) system:
\begin{equation} \label{dns1}
\left\{
\begin{aligned}
&\mathrm{d} n +u\cdot \nabla n \textrm{d}t = \Delta n \textrm{d}t- \textrm{div}(n \nabla c)\textrm{d}t ,&\textrm{in}~\mathbb{R}_+\times D,\\
&\mathrm{d} c+ u\cdot \nabla c  \textrm{d}t = \Delta c\textrm{d}t-n c\textrm{d}t ,&\textrm{in}~\mathbb{R}_+\times D,\\
&\mathrm{d} u+  (u\cdot \nabla) u \textrm{d}t+\nabla P \textrm{d}t = \Delta u\textrm{d}t + n\nabla \phi\textrm{d}t+ f(t,u) \mathrm{d} W_t ,&\textrm{in}~\mathbb{R}_+\times D,\\
& \textrm{div}  u  =0, &\textrm{in}~\mathbb{R}_+\times D,\\
&n|_{t=0}=n_0,~c|_{t=0}=c_0,~ u|_{t=0}=u_0,&\textrm{in}~ D,
\end{aligned}
\right.
\end{equation}
where  {$\mathbb{R}_+=[0,\infty)$} and $D\subseteq \mathbb{R}^2$ is a bounded convex domain. More precisely, Zhai and Zhang \cite{zhai20202d} established the existence and uniqueness of global mild solutions and weak solutions in appropriate function spaces for the KS-SNS system \eqref{dns1} with proper conditions on the initial data. Later in \cite{zhang2022global}, Zhang and Liu proved the existence of global martingale weak solutions to the KS-SNS system \eqref{dns1} in three spacial dimensions, subject to a general L\'{e}vy-type process. Recently, Hausenblas et al. \cite{hausenblas2023existence} established the existence and uniqueness of probability strong solutions to the two-dimensional KS-SNS system \eqref{dns1} in a bounded domain, with an additional transport-type noise on the $c$-equation.

Compared to its deterministic counterpart \eqref{dns}, the literature on the stochastically perturbed KS-NS system (such as \eqref{dns1}) is still relatively scarce, with the exception of the three aforementioned works.  {The primary objective of this paper is to make a further effort to comprehend the stochastic system that describes the interactions between chemotaxis and fluid.} To be more precise, we are interested in the existence and uniqueness of global solutions to the following Cauchy problem of KS-SNS system:
\begin{equation}\label{KS-SNS}
\left\{
\begin{aligned}
&\mathrm{d} n +u\cdot \nabla n \textrm{d}t = \Delta n \textrm{d}t- \textrm{div}(n \nabla c)\textrm{d}t+(  n-  n^2) \textrm{d}t , &\textrm{in}~\mathbb{R}_+\times \mathbb{R}^2,\\
&\mathrm{d} c+ u\cdot \nabla c  \textrm{d}t = \Delta c\textrm{d}t-n c\textrm{d}t,& \textrm{in}~\mathbb{R}_+\times \mathbb{R}^2,\\
&\mathrm{d} u+  (u\cdot \nabla) u \textrm{d}t+\nabla P \textrm{d}t = - (-\Delta)^\alpha u\textrm{d}t + n\nabla \phi\textrm{d}t+ f(t,u) {\mathrm{d} W_t} ,&\textrm{in}~\mathbb{R}_+\times \mathbb{R}^2,\\
& \textrm{div}  u  =0, &\textrm{in}~\mathbb{R}_+\times \mathbb{R}^2,\\
&n|_{t=0}=n_0,~c|_{t=0}=c_0,~ u|_{t=0}=u_0,&\textrm{in}~ \mathbb{R}^2.
\end{aligned}
\right.
\end{equation}
Here $ (-\Delta)^\alpha $ denotes the fractional Laplacian defined by Fourier transform:
\begin{equation}\label{1.2}
\begin{split}
 \widehat{(-\Delta)^\alpha f}(\xi)=(2 \pi |\xi|)^{2 \alpha} \widehat{f }(\xi).
 \end{split}
\end{equation}
Throughout this work, we assume that
\begin{equation}\label{dissipation}
\begin{split}
 \frac{1}{2}\leq \alpha \leq 1.
 \end{split}
\end{equation}

Let us highlight the new ingredients involved in the structure of the system \eqref{KS-SNS}.
\begin{enumerate}
\item [1.] One of the most important contribution of this work is to consider the stochastic Navier-Stokes equations \eqref{KS-SNS} with a fractional dissipation  $(-\Delta)^\alpha u$, by extending the previous works \cite{zhai20202d,hausenblas2023existence,zhang2022global}. Indeed, the fractional dissipation appears naturally in the hydrodynamics and certain combustion models \cite{resnick1996dynamical,woyczynski2001levy}, and it is a generalization of the classical Navier-Stokes equations (i.e. $\alpha = 1$). This model is applicable for describing fluids with internal friction interaction \cite{mercado2013analysis}, where the parameter $0<\alpha\leq1$ represents  the order of the diffusion and measures the strength of viscous effects. Since Lion's pioneering work \cite{lions1959quelques}, stochastic Navier-Stokes equations with fractional dissipation have received much attention, such as \cite{rockner2014local,debbi2016well,li2018stochastic,yamazaki2022remarks,yin2022well,
yamazaki2023non,rehmeier2023nonuniqueness} and the references cited therein. 

\item [2.] A new addition to our model is the quadratic logistic source on $n$-equation, i.e., $l(n):=  n- n^2$. This treatment is motivated by two reasons. First, the logistic source has been extensively utilized in KS models for characterizing proliferation-death mechanism (cf. \cite{tuval2005,fujikawa1989fractal,hillen2009user}), making it reasonable to incorporate this mechanism into our model, and we refer to \cite{winkler2022reaction} for such an application. Second, the presence of the weak dissipation $(-\Delta)^\alpha u$  makes it challenging to establish uniform bounds for approximations similar to \cite{zhai20202d}, thereby rending the tightness of approximations uncertain. However, the additional logistic source $l(n)$ enhances the $L^2$-integrability of $n^\epsilon$, which facilitates obtaining a priori estimates by introducing new ideas, see Remark (\textbf{b}) below.

\item [3.] Another new ingredient in \eqref{KS-SNS} is that the domain is assumed to be unbounded. The deterministic KS-NS system \eqref{dns} in unbounded domain has received much attention during last several years. However, as far as we ware, few relevant results for the stochastic counterpart could be available in the literatures, and this inspires us to make a first attempt to consider \eqref{KS-SNS} in unbounded domain. It is worth pointing out that the investigation of Cauchy problem for the KS-SNS system in unbounded domain is more subtle than the deterministic KS-NS system \cite{zhang2014global,lorz2010coupled,duan2010global} and the KS-SNS system \eqref{dns1} in bounded domain \cite{zhai20202d}. Therefore, some new ideas has to be applied.
\end{enumerate}

To state the main result of this work, let us start with the basic stochastic setting. Assume that $(\Omega,\mathcal {F},(\mathcal {F}_t)_{t\geq0},\mathbb{P})$ is a given complete filtered probability space on which a cylindrical Wiener process $\{W(t)\}_{t\geq 0}$ is defined. Formally, the stochastic process $W(t)$ is expanded as
\begin{equation} \label{1.4}
\begin{split}
W(t,\omega)=\sum_{k\geq 1}W^k(t,\omega)e_k,\quad  {\textrm{for all}~(t, \omega) \in[0, T] \times \Omega,}
 \end{split}
\end{equation}
where $\{W^k(t)\}_{k\geq1}$ is a family of independent one-dimensional standard Wiener processes, and $\{e_k\}_{k\geq1}$ is the complete orthogonal basis in a separable Hilbert space $U$.
To make sense of the expansion \eqref{1.4}, we introduce an auxiliary space $U_0$ via
\begin{equation}
\begin{split}
U_0=\left\{v=\sum_{k\geq 1}\alpha_k e_k;~\sum_{k\geq 1} \frac{\alpha_k^2}{k^2}<\infty\right\}\supset U,
 \end{split}
\end{equation}
which is endowed with the norm $
\|v\|_{U_0}^2=\sum_{k\geq 1}  \alpha_k^2/k^2$, for any $ v=\sum_{k\geq 1}\alpha_k e_k \in U_0$.
Note that the embedding from $U$ into $U_0$ is Hilbert-Schmidt (cf. \cite[Appendix C]{da2014stochastic}), and the trajectory of $W(t)$ belongs to $C([0,T];U_0)$, $\mathbb{P}$-a.s.

The main assumptions in the work are presented as follows.

\begin{itemize}[leftmargin=0.95cm]
\item [\textbf{(H1)}] 1) $\phi\in W^{1,\infty}(\mathbb{R}^2;\mathbb{R})$;

2) The initial data $(n_0,c_0,u_0)$ satisfies:  {$n_0>0$, $c_0>0$} on $\mathbb{R}^2$, and
\begin{equation*}
\begin{split}
  &\sqrt{1+|x|^2}n_0 \in L^1(\mathbb{R}^2),~n_0\in L^2(\mathbb{R}^2);~~ c_0\in L^1(\mathbb{R}^2)\cap L^\infty(\mathbb{R}^2),\\
  &\nabla \sqrt{c_0}\in L^2(\mathbb{R}^2);~~u_0 \in H^1(\mathbb{R}^2)\cap \dot{W}^{1,\frac{4}{3}}(\mathbb{R}^2).
 \end{split}
\end{equation*}

\item [\textbf{(H2)}]  {For any $s \geq 0$,} there exists a constant $C >0$ such that
\begin{equation*}
\begin{split}
 \|f(t,u)\|_{L_2(U;H^{s})}^2\leq  C \left(1+\|u\|_{H^{s}}^2 \right),
 \end{split}
\end{equation*}
\begin{equation*}
\begin{split}
 \|f(t,u_1)-f(t,u_2)\|_{L_2(U;H^{s})} \leq  C\|u_1-u_2\|_{H^{s}},
 \end{split}
\end{equation*}
for any $t>0$, and $u ,u_1,u_2\in  H^s(\mathbb{R}^2)$.

\item [\textbf{(H3)}]  For any $s \geq 1$, there exists a constant $C>0$ such that
$$
 \|\nabla \wedge f(t,u)\|_{L_2(U;H^{s-1})}^2\leq  C \left(1+\| u\|_{H^{s}}^2\right) ,
$$
$$
  \|\nabla \wedge f(t,u_1)-\nabla \wedge f(t,u_2)\|_{L_2(U;H^{s-1})}^2\leq  C\|u_1- u_2\|_{H^{s}}^2,
$$
for any $t>0$, and $u_1,u_2,u\in H^{s}(\mathbb{R}^2)$. Moreover, we assume that
$$
\left\|\frac{\nabla \wedge f(t,u)}{|\nabla\wedge u|^{\frac{1}{3}}}\textbf{1}_{\{\nabla\wedge u\neq 0\}} \right\|_{L_2(U;L^2)}^2\leq  C \left(1+\|\nabla\wedge u\|_{L^{\frac{4}{3}}}^{\frac{4}{3}}\right),
$$
for any $t>0$, and $u\in \dot{W}^{1,\frac{4}{3}}(\mathbb{R}^2)$.
\end{itemize}

Here the homogeneous Sobolev space $\dot{W}^{1,4/3}(\mathbb{R}^2)$  is defined as the closure of $C_0^\infty (\mathbb{R}^2)$ (infinity smooth, compactly supported  functions) with respect to the norm $$\|u\|_{\dot{W}^{1,4/3}}:= \|\nabla u\|_{L^{4/3}}.$$ A basic example of $f$ satisfying the assumptions (H2)-(H3) is $f(t,u)= \lambda u$ for some intensity $\lambda>0$, which means that the noise in \eqref{KS-SNS} reduces to the linear multiplicative noise.

Based on the above assumptions, we shall prove that the Cauchy problem \eqref{KS-SNS} admits a unique global-in-time solution. The following theorem shows the  {main result} in this article.

\begin{theorem}\label{main}
Assume that $(\Omega,\mathcal {F},(\mathcal {F}_t)_{t\geq0},\mathbb{P})$ is a fixed stochastic basis with a complete right-continuous filtration, and $W(t)$ is a $\mathcal {F}_t$-adapted cylindrical Wiener process defined by \eqref{1.4}. Then under the hypothesises (H1)-(H3) and \eqref{dissipation}, there exists a unique global pathwise weak solution $(n,c,u)$ to the Cauchy problem \eqref{KS-SNS}, such that the following statements hold:
\begin{itemize}[leftmargin=0.7cm]
\item [(1)] For any $T>0$, the triple $(n,c,u)$ satisfies $\mathbb{P}$-a.s.
 \begin{equation*}
\begin{split}
& n \in L^\infty \left(0,T; L^1(\mathbb{R}^2)\cap L^2 (\mathbb{R}^2)\right)\cap L^2 \left(0,T; H^1(\mathbb{R}^2)\right)\cap L^3 \left(0,T; L^3(\mathbb{R}^2) \right),\\
&c \in L^\infty \left(0,T; L^1(\mathbb{R}^2)\cap L^\infty (\mathbb{R}^2)\cap H^1(\mathbb{R}^2) \right)\cap L^2 \left(0,T; H^2(\mathbb{R}^2)\right),\\
& u \in L^\infty \left(0,T; H^1(\mathbb{R}^2)\cap \dot{W}^{1,\frac{4}{3}} (\mathbb{R}^2)\right)\cap L^2 \left(0,T; H^{1+\alpha}(\mathbb{R}^2)\right).
\end{split}
\end{equation*}

\item [(2)] The following relationships hold $\mathbb{P}$-a.s.
\begin{equation*}
\begin{split}
 (n(t),\varphi_1)_{L^2}=&(n_0,\varphi_1)_{L^2}+\int_0^t(u n-\nabla n+ n \nabla c ,\nabla\varphi_1)_{L^2}\textrm{d}r +\int_0^t\left(n-  n^2 ,\varphi_1\right)_{L^2} \textrm{d}r,   \\
 (c(t),\varphi_2)_{L^2}=&(c_0,\varphi_2)_{L^2}+\int_0^t (u c-\nabla c, \nabla\varphi_2)_{L^2} \textrm{d}r -\int_0^t\left(n c,\varphi_2\right)_{L^2}\textrm{d}r,\\
 (u(t),\varphi_3)_{L^2} =&(u_0,\varphi_3)_{L^2} + \int_0^t  ( u\otimes u,\nabla\varphi_3)_{L^2} \textrm{d}r -\int_0^t\left ( (-\Delta)^\frac{\alpha}{2}u,(-\Delta)^\frac{\alpha}{2}\varphi_3 \right)_{L^2}\textrm{d}r \\
 & + \int_0^t(n\nabla \phi,\varphi_3)_{L^2}\textrm{d}r+ \sum_{k\geq 1}\int_0^t(f(r,u)e_k ,\varphi_3)_{L^2}\mathrm{d} W^k_r,
 \end{split}
\end{equation*}
for all $t\in [0,T]$, $\varphi_1,\varphi_2\in C^\infty_0(\mathbb{R}^2;\mathbb{R})$ and $\varphi_3\in C^\infty_{0 }(\mathbb{R}^2;\mathbb{R}^2)$ with $\textrm{div} \varphi_3=0$.
\end{itemize}
\end{theorem}

\noindent
We make the following remarks related to Theorem \ref{main}:

\begin{itemize}[leftmargin=0.69cm]
\item [\textbf{(a)}]   Novelties in present work:
\begin{itemize}[leftmargin=0.43cm]
\item [$\bullet$]   The first new ingredient in our work is the three-layer approximate system \eqref{3.3}. At first glance, the system differs from the one approximated by using the Galerkin method. However, it is exactly efficient for constructing approximate solutions to  system \eqref{KS-SNS}, and we believe that the framework of approximate system used in this paper can be applied to other stochastic chemotaxis-fluid models. Indeed, our approximate system is motivated by the works for deterministic systems and the theory of SDEs in infinite dimensions. In this work, in order to overcome the difficulty caused by the low-regularity terms, such as $\textrm{div} (n(\nabla c) )$ in the $n$-equation, we first regularized the nonlinear terms by using standard mollifiers, which is essential for establishing the well-posedness of the approximate solutions, especially for the uniqueness of approximations and a series of entropy-energy estimates. Then, to prove the existence of approximate solutions to this regularized system, we transform the system into a class of SDEs in infinite dimensions by introducing proper Fourier truncations and cut-off techniques in \eqref{3.1}.  Please see the Remark (b) below for detailed  explanations.

\item [$\bullet$]   The second novelty of this work lies in considering the stochastic chemotaxis-Navier-Stokes system with fractional dissipation $\Delta ^\alpha u$, $1/2\leq \alpha\leq1$, which involves the system considered in \cite{zhai20202d,hausenblas2023existence} in the case of $\alpha=1$. As it is mentioned before, the stochastic fluid model with fractional dissipation appears naturally in the hydrodynamics and certain combustion models, which has been studied extensively during past decades. Therefore from the mathematical standpoints, it is important to investigate the Cauchy problem of the stochastic chemotaxis-Navier-Stokes system with fractional Laplacian. Moreover, the appearance of the fractional diffusion with $1/2\leq \alpha<1$ brings new difficulties in deriving several key a priori estimates, which makes the estimations in existing works unapplicable to the current problem. To overcome this difficulty, we use a different strategy to derive the entropy and energy estimates, which depends on the structure of the system and the use of the microlocal analysis in Littlewood-Paley theory. As a  by-product, the Sobolev interpolation inequalities obtained in Lemma \ref{lem1.6} seems to be new and also have their own interest.

\item [$\bullet$]  The third novelty is related to the problem caused by the fractional dissipation. More precisely, since it is impossible to obtain the estimates $L^p(\Omega,L^2(0,T;H^1(\mathbb{R}^2)))$ for $u^\epsilon$ as $\alpha<1$, the methods in \cite{zhai20202d,hausenblas2023existence} can not be applied to control the cross terms, which makes the stochastic entropy-energy inequality in \cite{zhai20202d,hausenblas2023existence} invalid here. Due to the randomness of system, the classical methods for the deterministic systems are invalid here too.  Our strategy for overcoming this difficulty is inspired by the fact that the subquadratic death terms may be sufficient to suitably counteract possible cross-diffusive destabilization, see e.g. \cite{Viglialoro2016Very,lankeit2016long,Winkler2020The}. Indeed, it is found that the logistic source $l(n)=n-n^2$ is sufficient to ensure the global existence of solutions to the system. Notice that this is the first time to explore the influence of logistic term in the stochastic chemotaxis-fluid model, whether the sub-quadratic logistic term has the similar effect is still an open problem. Due to the appearance of logistic term, we are able to use the strategy described in Remark (b) to derive a series of entropy and energy estimates, which guarantees the tightness of the approximation $u^\epsilon$.
    \end{itemize}
 The above three novelties make our work different from the existing ones, and we believe that the framework and mathematical tools used in the paper have a great inspiration for studying other stochastic chemotaxis-fluid systems.

\item [\textbf{(b)}]   {Difficulties  and strategies in the proof:}
\begin{itemize}[leftmargin=0.43cm]
\item [$\bullet$]
  {(\textbf{Approximate system}) Our approximate system is inspired by the treatment for both deterministic and stochastic PDEs. First, in order to deal with the low-regularity terms such as $\textrm{div} (n(\nabla c) )$ in the $n$-equation, we introduce the approximate system \eqref{3.1} in the stochastic setting. As we shall see later, this regularization with parameter $\epsilon$ is essential in establishing the well-posedness of the approximate solutions $\textbf{u}^\epsilon$, especially for the uniqueness result and deriving a series of entropy-energy estimates. Indeed, because of the special structure in the CNS system (which is connected to Boltzmann equations and semiconductor equations), the authors in \cite{liu2011coupled,chae2014global,duan2010global} are motivated to use the powerful entropy-energy inequality to establish the global results. To obtain such a Lyapunov functional inequality, the authors regularized the system by mollifiers similar to \eqref{3.1}. However, different from the deterministic cases and the existed works for stochastic counterparts, it is not an easy work to show the existence of solutions to the approximate system \eqref{3.1}, at least it is not obvious. To treat this problem, our basic idea is to use the regularization method in \cite{majda2002vorticity}, i.e., we further regularize the system \eqref{3.1} by using Friedrichs projectors (with parameter $k$), resulting in a nonlinear SDE in Hilbert space $\textbf{H}^s(\mathbb{R}^2)$. At this stage, it seems that the known theories for SDEs in infinite dimensions  (cf. \cite{prevot2007concise,kallianpur1995stochastic,da2014stochastic}) can not be applied directly. As a matter of fact, when establishing the well-posedness for such SDEs, a new difficulty arises, i.e., the estimates such as
$$
\mathbb{E}\int_0^{t\wedge \textbf{t}_\eta} (\Lambda^s\textbf{u}^{k,\epsilon},\Lambda^s\mathcal {J}_k\textbf{B}(\mathcal {J}_k\textbf{u}^{k,\epsilon} ))
_{\textbf{L}^2} \textrm{d}r \leq C \mathbb{E}\int_0^{t\wedge \textbf{t}_\eta} (1+ \|\textbf{u}^{k,\epsilon} \|_{\textbf{H}^s}^2) \textrm{d}r
$$
holds only up to a sequence of stopping times $\textbf{t}_\eta :=\inf\{t>0;~ \|\textbf{u}^{k,\epsilon}(t)\|_{\textbf{W}^{1,\infty}}\geq \eta\}$.  How to bound $\textbf{t}_\eta$ from below, uniformly in $\eta$, is still an open problem. Therefore only using stopping times seems not enough to obtain the estimate uniformly in $k$. Motivated by the truncation technique in SPDEs (cf. \cite{rockner2014local,breit2018stochastically,du2020local}), we further add a smooth cut-off function (with the parameter $R$) depending on the size of $\|\textbf{u}^{k,\epsilon} \|_{\textbf{W}^{1,\infty}}$ in front of the nonlinear term to obtain \eqref{3.3}. This is one of the differences compared with the existing papers on stochastic chemotaxis-fluid models. For this three-layer approximate system, one can derive a priori uniform estimates for the approximations $\textbf{u}^{k,R,\epsilon} $, and the existence and uniqueness of approximations is ensured by the classical Leha-Ritter Theorem.}
%

\item [$\bullet$]   {(\textbf{Taking limits  $k\rightarrow\infty$ and $R\rightarrow\infty$}) The main difficulty for taking the limit $k\rightarrow\infty$ lies in the fact that one can not prove the following property
    $$
    \theta_R(\|\textbf{u}^{k,R,\epsilon}\|_{W^{1,\infty}})
\rightarrow\theta_R(\|\textbf{u}^{R,\epsilon}\|_{W^{1,\infty}})  ~~ \textrm{as}~~k\rightarrow\infty
$$
directly, since the value of $\theta_R(\|\textbf{u}^{k,R,\epsilon}\|_{W^{1,\infty}})$ depends on the information of $\textbf{u}^{k,R,\epsilon}$ over the whole space, and the uniform bound in Lemma \ref{lem3} only ensures the local strong convergence $\textbf{u}^{k,R,\epsilon}\rightarrow\textbf{u}^{ R,\epsilon}$ in $\textbf{H}^{s} (K)$, for compact subset $K\subseteq \mathbb{R}^2$. We overcome this difficulty by adopting some ideas from \cite{li2021stochastic}, where the key step is to demonstrate the convergence of
$\textbf{u}^{k,R,\epsilon}$ in $C([0, T], \textbf{H}^{s-3}(\mathbb{R}^2))$ through a careful analysis of the approximation scheme, and then elevate the spatial regularity of the solution in $\textbf{H}^{s}(\mathbb{R}^2)$ (cf. Lemmas \ref{lem3.5}-\ref{lem3.4} below). Note that in bounded domain, the aforementioned convergence is a direct consequence of Lemma \ref{lem3} and Sobolev embedding theorem. With the Lemmas \ref{lem3.5}-\ref{lem3..4} at hand, we are able to take the first limit $k\rightarrow \infty$ to construct a local pathwise solution $\textbf{u}^{R,\epsilon}$ for the system \eqref{3.1} with cut-off functions $\theta_R(\|\textbf{u}^{R,\epsilon}\|_{\textbf{W}^{1,\infty}})$. Based on a uniqueness result for system \eqref{3.1} (cf. Lemma \ref{lem3.7}) and the stopping techniques, one can  show that \eqref{3.1} has a unique local maximal pathwise solution in $C([0,\widetilde{\textbf{t}^\epsilon}); \textbf{H}^s(\mathbb{R}^2))$ by taking the limit $R\rightarrow\infty$. Finally, we prove that $\mathbb{P}\{\widetilde{\textbf{t}^\epsilon}=+\infty\}=1$, i.e., the local pathwise solution to system \eqref{3.1} exists globally. The difficulty lies in dealing with the fluid equation, that is, we need to introduce proper stopping times and apply the localization technique in Littlewood-Paley theory to estimate the vorticity $\nu^\epsilon = \textrm{curl}~ u^\epsilon$. }

\item [$\bullet$]   {(\textbf{Entropy-energy estimates \& taking limit $\epsilon\rightarrow 0$}) Having obtained the approximate solutions $(\textbf{u}^\epsilon)_{\epsilon>0}$ to the system \eqref{3.1}, it remains to take the limit $\epsilon\rightarrow 0$ to construct global martingale solutions to  system \eqref{KS-SNS}. The proof of existence is based on the stochastic compactness method, which typically relies heavily on a series of a priori estimates for $ \textbf{u}^\epsilon $. The entropy-energy inequality  is widely utilized in this regard, see for example \cite{winkler2012global,zhai20202d,nie2020global,hausenblas2023existence}. However, the existing entropy-energy inequalities appear to be unsuitable for addressing the current problem due to the lower diffusion $\Delta^\alpha$ ($1/2\leq\alpha<1$) and unboundedness of domain. Thus, some new ideas have to be introduced. Our strategy for building a priori estimates can be briefly described as follows. We start with establishing the $L^2$-integrability of $n^\epsilon$ by using the structure of the $n$-equation,  which plays an important role in  deriving high-order momentum estimates for $u^\epsilon$.  Based on this, we then investigate the uniform bound for the quantity $\sqrt{c^{\epsilon}}$, which in turn enables us to construct an entropy-type inequality for the $n^\epsilon$-component. Please refer to Lemmas \ref{4.5}-\ref{4.6} for details. With these findings in hand, one can introduce an event set $(\Omega^\epsilon_N)_{N\geq 1}$ associated to the quantities $\sqrt{c^{\epsilon}}$ and $u^{\epsilon}$, which has high probability as $N$ large enough. Over this event set, we are able to derive an estimate for the term $\sup_{t\in[0,T]}\|n^\epsilon\|_{L^2}^2$ (cf. \eqref{4.28}).  Finally, we develop higher regularity for $c^\epsilon$ in $H^1(\mathbb{R}^2)$ and the vorticity $v^\epsilon$ in $L^{4/3}(\mathbb{R}^2)$ by using the structure of the vorticity equation. Based on the aforementioned uniform bounds, it is possible to prove the tightness of $\{\textbf{u}^\epsilon\}_{\epsilon>0}$ and apply the Skorokhod Representation Theorem to prove the existence of global martingale solutions. Finally, to verify that the solutions is strong in probability, we prove a pathwise uniqueness result for the martingale solutions. As a  by-product, the Sobolev interpolation inequalities obtained in Lemma \ref{lem1.6} seems to be new and also have their own interest.}

\end{itemize}

\item [\textbf{(c)}]  {Some interesting problems:}
\begin{itemize}[leftmargin=0.43cm]

\item [$\bullet$]  Our research might be seen as a continuation and extension of Zhai and Zhang's work \cite{zhai20202d}, so we use the same random noise. It will be of great interest to explore the KS-SNS system \eqref{KS-SNS} with general random noises, such as $f(t,n,c,u)\textrm{d} W_t$ instead of $f(t,u)\textrm{d} W_t$. Using an argument similar to Section 3, it is possible to construct approximations $\textbf{u}^{k,R,\epsilon}$ with global Lipschitz conditions on $f$. However, due to the nonlinear interaction between the stochastic force and the $n$- and $c$-equations, it is quite difficult to derive proper entropy-energy estimates as obtained in Section 4, which are crucial for taking the limit $\epsilon\rightarrow0$ to prove the existence of martingale weak solutions (cf. Lemma \ref{lem4.7} below). Therefore, an alternative novel proof strategy needs to be explored.

\item [$\bullet$]  Noted that the KS model is a macroscopic representation derived from the limiting behaviour of its microscopic counterpart through balance laws and Fick's law of diffusion \cite{stevens2000derivation}. Taking into account molecular fluctuations, a stochastic version of the KS model can be obtained as demonstrated in \cite{mayorcas2021blow,huang2021microscopic,hausenblas2022one,hausenblas2022uniqueness,misiats2022global}. Inspired by this, it is reasonable to investigate the stochastic KS model in conjunction with stochastic Navier-Stokes equations in both bounded and unbounded domains. However, only one study by Hausenblas et al. \cite{hausenblas2023existence} is available, where random noises affect the first and third components of \eqref{dns} in a bounded domain $D\subseteq \mathbb{R}^2$. Due to the complexities involved in the system, investigating \eqref{dns} with random noises on three components poses a significant challenge, requiring extensive research efforts.

\end{itemize}
\end{itemize}

The paper is organized as follows.  In Section 2, we will review fundamental concepts of the Littlewood-Paley theory and establish several key bilinear estimates in Besov spaces through the utilization of the Bony paraproduct decomposition. Section 3 is dedicated to introducing the approximation system \eqref{3.3} and proving the existence and uniqueness of smooth solutions to \eqref{3.1}. In Section 4, a series of crucial entropy and energy estimates are derived, which enable us to construct a unique weak solution to \eqref{KS-SNS} via stochastic compactness method. Section 5 is devoted to proving several useful basic inequalities applied in the argument.

\section{The Littlewood-Paley theory}
This section intends to review the definition of specific function spaces that will be frequently utilized in the subsequent sections, and offer a concise overview of essential concepts associated with the Littlewood-Paley theory. We denote by $W^{s,p}(\mathbb{R}^2)$, $s\in \mathbb{R}$, $1\leq p \leq \infty$, the Bessel potential space with the norm
$
\|f\|_{W^{s,p}}= \|(1-\Delta)^{\frac{s}{2}}f\|_{L^p}.
$
The norm of the homogenous space $\dot{W}^{s,p}(\mathbb{R}^2)$ is given by
$
\|f\|_{\dot{W}^{s,p}}= \|\Delta^\frac{s}{2} f\|_{L^p}.
$
For each $n\geq 2$, we define
$$
\textbf{W}^{s,p}(\mathbb{R}^2):=    \underbrace{W^{s,p}(\mathbb{R}^2) \times \cdots \times  W^{s,p}(\mathbb{R}^2)}_{\textrm{n-terms}},
$$
which is endowed with the norm
$
\| (f_1,\cdots,f_n)\|_{\textbf{W}^{s,p}}=\sum_{i=1}^n\|f_i\|_{W^{s,p}} .
$
For any $s\in\mathbb{R} $, we introduce the divergence-free space
$$
 \textbf{H}^s(\mathbb{R}^2):=\{(n,c,u)\in H^s(\mathbb{R}^2)\times H^s(\mathbb{R}^2) \times(H^s(\mathbb{R}^2))^2;~ \textrm{div} u=0\},
$$
 with the norm
$$
\|(n,c,u)\|_{\textbf{H}^s}= \|n\|_{H^s}+\|c\|_{H^s}+\|u\|_{H^s}.
$$
When there is no confusion, for $p=2$, we also write $\textbf{W}^{s,2}(\mathbb{R}^2)$ as $\textbf{H}^s(\mathbb{R}^2)$ to avoid superfluous notations. In the sequel, $C(a,b,...)$ denotes the positive constants depending only on $a,b,...$, which may change from line to line.

We will now review some fundamental concepts related to the Littlewood-Paley theory. For more comprehensive details, please refer to \cite[Chapter 2]{bahouri2011fourier} and \cite[Chapter 1]{miao2012littlewood}.  Define
$
\mathcal{C}=\{\xi \in \mathbb{R}^2; \frac{3}{4} \leq| \xi \mid \leq \frac{8}{3}\}$, $B(0, \frac{4}{3})=\{\xi \in \mathbb{R}^2; | \xi |\leq \frac{4}{3}\}$.
Then there are two radial functions $\varphi\in C_0^\infty(\mathcal{C})$ and $\chi\in C_0^\infty(B(0, \frac{4}{3}))$ such that
$$\sum_{j \in \mathbb{Z}} \varphi\left(2^{-j} \xi\right)=1, ~~~~\textrm{for all} ~\xi \in \mathbb{R}^2 \backslash\{\textbf{0}\},
$$
$$
  \operatorname{supp} \chi (\cdot)\cap \operatorname{supp} \varphi (2^{-j} \cdot )=\emptyset, ~~~~ \textrm{for all}~j \geq 1,
$$
and
$$
\operatorname{supp} \varphi (2^{-j}\cdot ) \cap \operatorname{supp} \varphi (2^{-j^{\prime}} \cdot)=\emptyset, ~~~~\textrm{for all}~ |j-j^{\prime} |\geq 2.
$$
The homogeneous  Littlewood-Paley blocks $\dot{\triangle}_j$ and low-frequency operators $\dot{S}_j$ are defined by
$$
\begin{aligned}
\dot{\triangle}_j u=\varphi\left(2^{-j} D\right) u  ,~~~ \dot{S}_j u=\chi\left(2^{-j} D\right) u=\sum_{j'\leq j-1}\dot{\triangle}_{j'} u.
\end{aligned}
$$
Let $\mathcal {S}'_h(\mathbb{R}^2)$ be the space of distributions $u$ such that {$\lim_{\lambda \rightarrow\infty}\|\theta(\lambda D)u\|_{L^\infty}=0$}, for all $\theta\in C_0^\infty(\mathbb{R}^2)$. The homogeneous Littlewood-Paley decomposition of a distribution $u$ is given by
$$
u=\sum_{j \in \mathbb{Z}}\dot{\triangle}_j u, ~~~~\textrm{for any}~ u \in\mathcal {S}'_h(\mathbb{R}^2).
$$

\begin{definition} For any $s \in \mathbb{R}$,  $1 \leq p,r \leq\infty$, the homogeneous Besov space $\dot{B}_{p, r}^s(\mathbb{R}^2)$ consists of all tempered distributions $u$ such that
$$
\|u\|_{\dot{B}_{p,r}^s}=
\begin{cases}
\bigg(\sum\limits_{j \in \mathbb{Z}}2^{jrs }\|\dot{\triangle}_ju\|_{L^p}^r\bigg)^{1/r}<\infty,&\textrm{if}~r<\infty,\\
\sup\limits_{j \in \mathbb{Z}}2^{js }\|\dot{\triangle}_ju\|_{L^p}<\infty,& \textrm{if}~ r=\infty.
\end{cases}
$$
\end{definition}

In contrast to the nonhomogeneous case, the homogenous Besov space does not possess monotonic property in $s\in\mathbb{R}$. However, the following significant embeddings are applicable.

\begin{lemma} {\cite[Proposition 2.20]{bahouri2011fourier}}
If $s \in \mathbb{R}$, $1 \leq p \leq \infty$, $1 \leq r_1 \leq r_2 \leq \infty$, then
\begin{equation}\label{2.2}
\begin{split}
\dot{B}_{p, r_1}^s(\mathbb{R}^2) \hookrightarrow \dot{B}_{p, r_2}^s(\mathbb{R}^2).
\end{split}
\end{equation}
If $1 \leq q \leq \infty$, $-\infty<s_2 \leq s_1<\infty$, $1 \leq p_1 \leq p_2 \leq \infty$ and $s_1- \frac{2}{p_1}=s_2-\frac{2}{p_2}$, then
\begin{equation}\label{2.3}
\begin{split}
\dot{B}_{p_1, r}^{s_1}(\mathbb{R}^2) \hookrightarrow \dot{B}_{p_2, r}^{s_2}(\mathbb{R}^2).
\end{split}
\end{equation}
\end{lemma}

The following bilinear estimates in Besov spaces are crucial for the discussion of Section 4.

\begin{lemma}\label{lem1.6}
For any $\alpha \in (\frac{1}{2},1]$, $1\leq p,r\leq\infty$, we have
\begin{equation}\label{2.4}
\begin{split}
\|f\cdot\nabla g\|_{\dot{B}_{p,r}^{-\alpha}}\leq \|f\|_{\dot{B}_{p,r}^{1-2\alpha+\frac{2}{p}}}\|g\|_{\dot{B}_{p,r}^{\alpha}}.
\end{split}
\end{equation}
In the case of $\alpha= \frac{3}{4}$ and $\alpha= \frac{1}{4}$, there hold
\begin{align}
 &\|f\cdot\nabla g\|_{\dot{B}_{p,r}^{- \frac{3}{4}}}\leq \|f\|_{\dot{B}_{p,r}^{ -\frac{1}{4}+\frac{2}{p}}}\|g\|_{\dot{B}_{p,r}^{\frac{1}{2}}},\label{2.5}\\
 &\|f\cdot\nabla g\|_{\dot{B}_{p,r}^{- \frac{1}{4}}}\leq \|\nabla g\|_{L^\infty}\|f\|_{\dot{B}_{p,r}^{ \frac{3}{4}}}.\label{2.6}
\end{align}
\end{lemma}
\begin{proof}
Let us recall that for any $u, v \in \mathcal{S}_h^{\prime}(\mathbb{R}^2)$, {the Bony paraproduct  decomposition} (cf. \cite[Section 2.6]{bahouri2011fourier} or \cite[Section 10]{miao2012littlewood}) of $uv$ is given by
$$
u v=\dot{T}_u v+\dot{T}_v u+\dot{R}(u, v),
$$
where
$$
\dot{T}_u v=\sum_{j \in \mathbb{Z}} \dot{S}_{j-1} u \dot{\triangle}_j v ,~~  \dot{T}_v u=\sum_{j \in \mathbb{Z}} \dot{S}_{j-1} v \dot{\triangle}_j u~~ \textrm{and}~~
\dot{R}(u, v)=\sum_{j \in \mathbb{Z}} \dot{\triangle}_j u \tilde{\triangle}_j v =\sum_{j \in \mathbb{Z}} \sum_{|j-k| \leq 1}\dot{\triangle}_j u  \dot{\triangle}_k v.
$$

We only need to prove \eqref{2.4}, since the inequalities \eqref{2.5}-\eqref{2.6} can be treated in a similar manner. By applying the Bony paraproduct decomposition to $f^i  \partial_i g$, $i=1,2$, we get
\begin{equation}\label{2.7}
\begin{split}
f \cdot \nabla g= \sum_{i=1}^2f^i\partial_i g&= \sum_{i=1}^2 \dot{T}_{f^i}\partial_i g + \sum_{i=1}^2\dot{T}_{\partial_i g}f^i+\sum_{i=1}^2 \dot{R}(f^i,\partial_i g):= B_1+B_2+B_3.
\end{split}
\end{equation}
For $B_1$, since $1-2\alpha <0$, we get by applying the continuity of the homogeneous paraproduct operator $\dot{T}$ and the remainder operator $\dot{R}$ (cf. \cite[Theorem 2.47 and Theorem 2.52]{bahouri2011fourier}) that
\begin{equation*}
\begin{split}
 \|B_1\|_{\dot{B}_{p,r}^{-\alpha}}
 &\leq C\sum_{i} \bigg(\sum_{q \in \mathbb{Z}}2^{-qr\alpha}\|\dot{S}_{q-1} f^i \dot{\triangle}_q \partial_ig \|_{L^p}^r\bigg)^{1/r}\\
 &\leq C\sum_{i=1}^2 \bigg (\sum_{q \in \mathbb{Z}}2^{ qr(1-2\alpha)}\|\dot{S}_{q-1} f^i\|_{L^\infty}^r 2^{ qr( \alpha-1)}\|\dot{\triangle}_q \partial_ig \|_{L^p}^r\bigg)^{1/r}\\
 &\leq C\sum_{i=1}^2\bigg[ 2^{ 1-2\alpha } \sup_{q \in \mathbb{Z}}2^{ (q-1) (1-2\alpha)}\|\dot{S}_{q-1} f^i\|_{L^\infty} \bigg(\sum_{q \in \mathbb{Z}}2^{ qr( \alpha-1)}\|\dot{\triangle}_q \partial_ig \|_{L^p}^r\bigg)^{1/r} \bigg]\\
 & = C\sum_{i=1}^2\bigg(2^{ 1-2\alpha }\| f^i\|_{\dot{B}_{\infty,\infty}^{1-2\alpha}}     \| \partial_ig \|_{\dot{B}_{p,r}^{\alpha-1} }\bigg)\leq C \|f \|_{B_{p,r}^{1-2\alpha+\frac{2}{p}}}   \| g \|_{\dot{B}_{p,r}^{\alpha } }.
\end{split}
\end{equation*}
For $B_2$, we have
\begin{equation*}
\begin{split}
 \|B_2\|_{\dot{B}_{p,r}^{-\alpha}}
 &\leq C\sum_{i=1}^2 \bigg(\sum_{q \in \mathbb{Z}}2^{ qr( \alpha-1-\frac{2}{p})}\|\dot{S}_{q-1} \partial_ig\|_{L^\infty}^r 2^{ qr( 1-2\alpha+\frac{2}{p})}\|\dot{\triangle}_q f^i \|_{L^p}^r\bigg)^{1/r}\\
 &\leq  C\sum_{i=1}^2 \bigg[2^{ \alpha-1-\frac{2}{p} }\sup_{q \in \mathbb{Z}}2^{ (q-1) ( \alpha-1-\frac{2}{p})}\|\dot{S}_{q-1} \partial_ig\|_{L^\infty}  \bigg(\sum_{q \in \mathbb{Z}}  2^{ qr( 1-2\alpha+\frac{2}{p})}\|\dot{\triangle}_q f^i \|_{L^p}^r\bigg)^{1/r}\bigg]\\
 &\leq  C\sum_{i=1}^2 \bigg( \|\partial_ig\|_{\dot{B}_{\infty,\infty}^{\alpha-1-\frac{2}{p}}} \|f^i\|_{\dot{B}^{1-2\alpha+\frac{2}{p}}_{p,r}}\bigg) \\
&\leq C\sum_{i=1}^2  \bigg( \| g\|_{\dot{B}_{\infty,\infty}^{\alpha -\frac{2}{p}}} \|f^i\|_{\dot{B}^{\alpha-1-\frac{2}{p}}_{p,r}}\bigg)\leq C \|f \|_{\dot{B}^{\alpha-1-\frac{2}{p}}_{p,r}}\| g\|_{\dot{B}_{p,r}^{\alpha}}.
\end{split}
\end{equation*}
For $B_3$, there holds
\begin{equation*}
\begin{split}
 \|B_3\|_{\dot{B}_{p,r}^{-\alpha}}
 &\leq  C\sum_{i=1}^2 \bigg(\sum_{q\in \mathbb{Z}} \sum_{|\nu|\leq1}2^{-qr \alpha} \|\dot{\triangle}_q f^i \dot{\triangle}_{q-\nu} \partial_i g  \|_{L^p}^r \bigg)^{1/r}\\
 &\leq  C\sum_{i=1}^2\bigg(\sum_{q\in \mathbb{Z}} \sum_{|\nu|\leq1}2^{ qr (1-2\alpha)} \|\dot{\triangle}_q f^i\|_{L^\infty}^r 2^{ qr (\alpha-1)} \|\dot{\triangle}_{q-\nu} \partial_i g  \|_{L^p}^r \bigg)^{1/r}\\
 &\leq  C  \sum_{i=1}^2\sum_{|\nu|\leq1}\bigg[ 2^{\nu (\alpha-1)}\sup_{q \in \mathbb{Z}}2^{ q (1-2\alpha)} \|\dot{\triangle}_q f^i\|_{L^\infty} \bigg(\sum_{q\in \mathbb{Z}} 2^{(q-\nu)r (\alpha-1)} \|\dot{\triangle}_{q-\nu} \partial_i g  \|_{L^p}^r \bigg)^{1/r}\bigg]\\
 & \leq C \sum_{i=1}^2 \bigg(\|f^i\|_{\dot{B}^{1-2\alpha}_{\infty,\infty}}\|\partial_i g\|_{\dot{B}_{p,r}^{\alpha-1} }\bigg) \leq C \|f \|_{\dot{B}^{1-2\alpha+\frac{2}{p}}_{p,r}}\| g\|_{\dot{B}_{p,r}^{\alpha} }.
\end{split}
\end{equation*}
Plugging the above estimates for $B_i$, $i=1,2,3$, into \eqref{2.7}, we obtain  \eqref{2.4}.
\end{proof}

We have the following useful property for the fractional Laplacian  $(-\Delta)^{\alpha}$.

\begin{lemma} {\cite[Proposition 2.30]{bahouri2011fourier}}
 For any $s \in \mathbb{R}$, $1 \leq p, r \leq \infty$, there exist a constant $ C> 0$ such that
\begin{equation*}
\begin{split}
 {\frac{1}{C}\|f\|_{\dot{B}_{p, r}^{s+2\alpha}}}\leq\|(-\Delta)^{\alpha} f\|_{\dot{B}_{p, r}^s}\leq C\|f\|_{\dot{B}_{p, r}^{s+2\alpha}}.
\end{split}
\end{equation*}
\end{lemma}

 \section{Construction of approximate solutions}

\subsection{Approximate system}

Let us introduce the approximation procedure as follows.

\textbf{(I)} For any $\epsilon \in (0,1)$, being inspired by the works \cite{liu2011coupled,chae2014global}, the first  regularized KS-SNS system takes the form of
\begin{equation}\label{3.1}
\left\{
\begin{aligned}
&\mathrm{d} n^{\epsilon} +  u^{\epsilon}\cdot \nabla n^{\epsilon} \textrm{d}t = \Delta n^{\epsilon} \textrm{d}t  -   \textrm{div} (n^{\epsilon}(\nabla c^{\epsilon} *\rho^{\epsilon}) )\textrm{d}t + \left(  n^{\epsilon}-     (n^{\epsilon} )^2\right) \textrm{d}t, \\
&\mathrm{d} c^{\epsilon}+  u^{\epsilon}\cdot \nabla c^{\epsilon}  \textrm{d}t = \Delta c^{\epsilon}\textrm{d}t-  c^{\epsilon}(n^{\epsilon}*\rho^{\epsilon})\textrm{d}t, \\
&\mathrm{d} u^{\epsilon}+ \textbf{P} (u^{\epsilon}\cdot \nabla) u^{\epsilon}\textrm{d}t = - \textbf{P}(-\Delta)^\alpha u^{\epsilon}\textrm{d}t+ \textbf{P}(n^{\epsilon}\nabla \phi)*\rho^{\epsilon}\textrm{d}t+ \textbf{P} f(t,u^{\epsilon})  {\mathrm{d} W_t} ,\\
&(n^{\epsilon},c^{\epsilon}, u^{\epsilon} )|_{t=0}=( n_0*\rho^{\epsilon}, c_0*\rho^{\epsilon}, u_0*\rho^{\epsilon}),
\end{aligned}
\right.
\end{equation}
where $\rho^{\epsilon}(\cdot) $ is a standard mollifier,and the initial data and nonlinear terms in \eqref{3.1} are regularized by the mollifier through the convolution $f* \rho^{\epsilon}:= \int_{\mathbb{R}^2}f(x-y)\rho^{\epsilon}(y) \textrm{d}y$.
$\textbf{P}:L^2(\mathbb{R}^2)\mapsto L^2_\sigma(\mathbb{R}^2) = \{u\in L^2(\mathbb{R}^2);~ \textrm{div} u=0 \}$
denotes the Holmholtz-Leray projection (cf.  {\cite[Proposition 1.16]{majda2002vorticity}}) defined by
$$
 \widehat{\textbf{P} u} (\xi)= \left(\textrm{Id}-\frac{\xi\otimes \xi}{|\xi|^2}\right)\widehat{u} (\xi),~~\textrm{for all}~\xi \in \mathbb{R}^2\setminus \{\textbf{0}\}.
$$

Unlike its deterministic counterpart, it is not obvious to construct smooth solutions to \eqref{3.1} directly.  {To overcome this difficulty, we are inspired by \cite{majda2002vorticity} and \cite{breit2018stochastically} to consider a further regularized system with additional parameters $R$ and $k$.}

\textbf{(II)} Denote
$$
\textbf{u}^\epsilon =\begin{pmatrix}
n^{\epsilon}\\
  c^{ \epsilon}\\
u^{ \epsilon}
\end{pmatrix} , ~~\textbf{A}^\alpha=\begin{pmatrix}
-\Delta&0&0\\
  0&-\Delta&0\\
0&0&\textbf{P}\dot{\Lambda}^{2\alpha}
\end{pmatrix} , ~~\textbf{B}(\textbf{u}^\epsilon)=\begin{pmatrix}
 (u^{\epsilon}\cdot \nabla) n^{\epsilon}  \\
  (u^{\epsilon}\cdot \nabla) c^{\epsilon} \\
\textbf{P}(u^{\epsilon}\cdot \nabla) u^{\epsilon}
\end{pmatrix} ,~~ {\mathcal {W}_t=\begin{pmatrix}
  0 \\
  0\\
  W_t
\end{pmatrix}},
$$
where $\dot{\Lambda}^s:= (-\Delta)^{s/2}$ for $s\in \mathbb{R}$,   {$\mathcal {W}_t$ is a cylindrical Wiener process in $\textbf{U}:=  \mathbb{R} \times \mathbb{R} \times U$}, and
$$
\textbf{F}^\epsilon(\textbf{u}^\epsilon)
=\begin{pmatrix}
- \textrm{div}\left(n^{\epsilon}(\nabla c^{\epsilon}*\rho^{\epsilon})\right)+  n^{\epsilon}-     (n^{\epsilon} )^2\\
 -  c^{\epsilon}(n^{\epsilon}*\rho^{\epsilon})\\
\textbf{P}(n^{\epsilon}\nabla \phi)*\rho^{\epsilon}
\end{pmatrix} ,~~ {\textbf{G}  (t,\textbf{u}^\epsilon)=\begin{pmatrix}
0&0&0\\
  0&0&0\\
0&0& \textbf{P} f(t,u^\epsilon)
\end{pmatrix}}.
$$
Then the system \eqref{3.1} can be reformulated into the following compact form:
\begin{equation}\label{3.2}
\left\{
\begin{aligned}
&\mathrm{d}  \textbf{u}^\epsilon+ \textbf{A}^\alpha\textbf{u}^\epsilon\textrm{d}t+\textbf{B}(\textbf{u}^\epsilon)\textrm{d}t=\textbf{F}^\epsilon (\textbf{u}^\epsilon)\textrm{d}t+\textbf{G}  (t,\textbf{u}^\epsilon)\textrm{d} \mathcal {W}_t,\\
&\textbf{u}^\epsilon(0)=\textbf{u}^\epsilon_0.
\end{aligned}
\right.
\end{equation}

For each $k\in \mathbb{N}^+$, we define the frequency truncation operators $\mathcal {J}_k$ by
$$\widehat{\mathcal {J}_k f} (\xi)= \textbf{1}_{\{\xi\in \mathbb{R}^2;~ \frac{1}{k}\leq |\xi|\leq k\}}(\xi)  \widehat{f}(\xi),$$  where $\textbf{1}_A(\cdot)$ denote the characteristic function on a set $A$.   For any $R>0$,  a smooth cut-off function $\theta_R : [0,\infty)\rightarrow [0,1]$ need to be selected  such that
$\theta_R (x)= 1$ if $0\leq x \leq R$; $\theta_R (x)=0$ if $x \geq 2 R$. Then, the further regularized system with  {cut-off functions} is provided by the following infinite-dimensional SDEs:
\begin{equation}\label{3.3}
\left\{
\begin{aligned}
&\mathrm{d}  \textbf{u}^{k,R,\epsilon} =\widetilde{\textbf{F}}^{k,R,\epsilon}(\textbf{u}^{k,R,\epsilon} )\textrm{d}t+\textbf{G}(\textbf{u}^{k,R,\epsilon} )\textrm{d} \mathcal {W}_t,\\
&\textbf{u}^{k,R,\epsilon}(0)= \textbf{u}^\epsilon_0,
\end{aligned}
\right.
\end{equation}
where
$$
\widetilde{\textbf{F}}^{k,R,\epsilon}(\textbf{u}) := - \mathcal {J}_k^2\textbf{A}^\alpha\textbf{u}-\theta_R(\|\textbf{u}\|_{\textbf{W}^{1,\infty}}) \mathcal {J}_k\textbf{B}(\mathcal {J}_k\textbf{u}) +\theta_R(\|\textbf{u}\|_{\textbf{W}^{1,\infty}})\mathcal {J}_k\textbf{F} ^\epsilon (\mathcal {J}_k\textbf{u}).
$$

The existence and uniqueness of solutions to \eqref{3.3} is guaranteed by the following result.

\begin{lemma}\label{lem1}
Let $s>5$, $k\in\mathbb{N}^+, R\geq1$ and $\epsilon \in (0,1)$. Assume that the conditions (H1)-(H2) hold. Then for any $T>0$, the system \eqref{3.3} admits a unique solution in $C([0,T];\textbf{H}^s(\mathbb{R}^2))$, $\mathbb{P}$-a.s.
\end{lemma}

\begin{proof}
\emph{Step 1 (Existence and uniqueness).} Recall that $\textbf{H}^s(\mathbb{R}^2)$ is a Banach algebra for any $s>1$, and
$$
\textrm{supp} \ \widehat{\mathcal {J}_k f} (\cdot)\subseteq \{\xi \in\mathbb{R}^2;~1/k\leq |\xi|\leq k\}.
$$
According to the Bernstein inequality (cf.  {\cite[Lemma 2.1]{bahouri2011fourier}}), there exists  a constant $C>0$ such that
\begin{equation}\label{3.4}
\begin{split}
C^{-(\ell+1)}k^\ell\|\mathcal {J}_k f\|_{L^2}\leq \sup_{|\alpha|=\ell}\|\partial^\alpha \mathcal {J}_k f\|_{L^2} \leq C^{\ell+1}k^\ell \|\mathcal {J}_k f\|_{L^2},\quad \forall \ell \in \mathbb{N}.
\end{split}
\end{equation}
For all $t>0$, we claim that
\begin{equation}\label{3.5}
\begin{split}
 &(\textbf{u},\widetilde{\textbf{F}}^{k,R,\epsilon}(\textbf{u}) )_{\textbf{H}^s} \leq C(\epsilon,\phi,k,R) \left(1+\|\textbf{u}\|_{\textbf{H}^s}^2 \right),\\
 &\|\textbf{G}(t,\textbf{u}) \|_{L_2(\textbf{U};\textbf{H}^s)}^2 \leq C \left(1+\|\textbf{u}\|_{\textbf{H}^s}^2 \right),
\end{split}
\end{equation}
which implies that the mappings $
\widetilde{\textbf{F}}^{k,R,\epsilon}:\textbf{H}^s(\mathbb{R}^2)\mapsto \textbf{H}^s(\mathbb{R}^2)$ and $ \textbf{G}: \textbf{H}^s(\mathbb{R}^2)\mapsto L_2(\textbf{U};\textbf{H}^s(\mathbb{R}^2)) $
are well-defined, and hence \eqref{3.3} can be viewed as SDEs in $\textbf{H}^s(\mathbb{R}^2)$. Indeed, {It suffices to verify the first inequality in \eqref{3.5}, since the second one is a direct consequence of (H2).

It follows from the boundedness of the operators $\textbf{P}$ and $\mathcal {J}_k$ that
\begin{equation} \label{a1}
\begin{split}
 \| \mathcal {J}_k^2\textbf{A}^\alpha\textbf{u}\|_{\textbf{H}^s} \leq C\left(\| \Delta\mathcal {J}_k n\|_{H^s}+\| \Delta\mathcal {J}_k c\|_{H^s}+\|\Delta^\alpha \mathcal {J}_k u\|_{H^s}\right).
\end{split}
\end{equation}
For the first two terms on the R.H.S. of \eqref{a1}, we get from \eqref{3.4} that
\begin{equation*}
\begin{split}
 \| \Delta\mathcal {J}_k n\|_{H^s}+\| \Delta\mathcal {J}_k c\|_{H^s} \leq C k^2 (\|n\|_{H^s}+\|c\|_{H^s}).
\end{split}
\end{equation*}
For the third term, by using the definition of $\mathcal {J}_k$ and the Fourier-Plancherel Formula (cf. \cite[Theorem 1.25]{bahouri2011fourier}), we get $\|\Delta^\alpha \mathcal {J}_k u\|_{H^s} \leq C k^{2\alpha}\|u\|_{H^s}$. Therefore, we get from \eqref{a1} that
\begin{equation} \label{a2}
\begin{split}
 \| \mathcal {J}_k^2\textbf{A}^\alpha\textbf{u}\|_{\textbf{H}^s} \leq C(k^2+k^\alpha) \|\textbf{u}\|_{\textbf{H}^s}.
\end{split}
\end{equation}
By using the Moser-type estimate (cf.  {\cite[Corollary 4.4]{miao2012littlewood}}) and \eqref{3.4}, we have
\begin{equation*}
\begin{split}
\left\|(\mathcal {J}_ku\cdot \nabla) \mathcal {J}_kn\right\|_{H^s} &\leq C(\left\|\mathcal {J}_ku \right\|_{H^s}\left\|\nabla\mathcal {J}_kn\right\|_{L^\infty}+ \left\|\mathcal {J}_ku \right\|_{L^\infty}\left\|\nabla\mathcal {J}_kn\right\| _{H^s} )\\
&\leq C(\left\| u \right\|_{H^s}\left\|\nabla n\right\|_{L^\infty}+ k\left\| u \right\|_{L^\infty}\left\| n\right\| _{H^s} )\\
&\leq C k(\left\| u \right\|_{L^\infty}+\left\|\nabla n\right\|_{L^\infty})(\left\| u \right\|_{H^s}+\left\| n\right\| _{H^s} ).
\end{split}
\end{equation*}
Similarly, we have
\begin{equation*}
\begin{split}
\left\|(\mathcal {J}_ku\cdot \nabla) \mathcal {J}_kc\right\|_{H^s} &\leq  C k(\left\| u \right\|_{L^\infty}+\left\|\nabla c\right\|_{L^\infty})(\left\| u \right\|_{H^s}+\left\|c\right\| _{H^s} ),\\
\left\|(\mathcal {J}_ku\cdot \nabla) \mathcal {J}_ku\right\|_{H^s} &\leq  C k(\left\| u \right\|_{L^\infty}+\left\|\nabla u\right\|_{L^\infty}) \left\|u\right\| _{H^s}.
\end{split}
\end{equation*}
Putting the last three estimates together, we infer that
\begin{equation}\label{a8}
\begin{split}
&\theta_R(\|\textbf{u}\|_{\textbf{W}^{1,\infty}}) \|\mathcal {J}_k\textbf{B}(\mathcal {J}_k\textbf{u})\|_{\textbf{H}^s}
\leq C k R \left\|\textbf{u}\right\| _{\textbf{H}^s}.
\end{split}
\end{equation}
Now let us estimate the term $\mathcal {J}_k\textbf{F} ^\epsilon (\mathcal {J}_k\textbf{u})$. First,
we get from the Young inequality that
\begin{equation*}
\begin{split}
&\left\|- \textrm{div}\left(\mathcal {J}_kn (\nabla \mathcal {J}_kc *\rho^{\epsilon})\right)+  \mathcal {J}_kn -     (\mathcal {J}_kn )^2\right\|_{H^s}\\
 &\quad \leq C(\|\mathcal {J}_kn\|_{H^{s+1}}\| \nabla \mathcal {J}_kc *\rho^{\epsilon} \|_{L^\infty}+ \|\nabla \mathcal {J}_kc *\rho^{\epsilon}\|_{H^{s+1}}\| \mathcal {J}_kn\|_{L^\infty}+ \|\mathcal {J}_kn\|_{H^s}+\|\mathcal {J}_kn \|_{H^s} \| \mathcal {J}_kn\|_{L^\infty})\\
 &\quad \leq C(1+(1+k^2)^{\frac{1}{2}}+  \epsilon^{-2}) (\| n\|_{L^\infty}+\|\nabla c\|_{L^\infty})  (\| n\|_{H^{s}}+\| c \|_{H^{s }} ).
\end{split}
\end{equation*}
Second, we get by the Moser-type estimate that
\begin{equation*}
\begin{split}
\left\|\mathcal {J}_kc(\mathcal {J}_kn*\rho^{\epsilon})\right\|_{H^s} &\leq  C(\| \mathcal {J}_kc\|_{H^{s}}\| \mathcal {J}_kn*\rho^{\epsilon}\|_{L^\infty}+ \| \mathcal {J}_kn*\rho^{\epsilon}\|_{H^{s}}\|\mathcal {J}_kc\|_{L^\infty})\\
& \leq C(\| n\|_{L^\infty}+\| c\|_{L^\infty})  (\| n\|_{H^{s}}+\| c \|_{H^{s }} ).
\end{split}
\end{equation*}
Third, by the estimate $\|\rho^{\epsilon}*u\|_{H^s}\leq C \epsilon^{r-s} \| u\|_{H^r}$ with $s\geq r$ (cf. \cite[Lemma 3.5]{majda2002vorticity}), we have
\begin{equation*}
\begin{split}
\left\|\textbf{P}(\mathcal {J}_kn\nabla \phi)*\rho^{\epsilon}\right\|_{H^s}  \leq  C \epsilon^{-s}\|\nabla \phi\|_{L^\infty}\left\|n\right\|_{H^s}.
\end{split}
\end{equation*}
Therefore, the last three inequalities imply that
\begin{equation}\label{a9}
\begin{split}
&\theta_R(\|\textbf{u}\|_{\textbf{W}^{1,\infty}})\left\|\mathcal {J}_k\textbf{F}^\epsilon(\mathcal {J}_k\textbf{u}^\epsilon)\right\|_{H^s}
\\
&\quad\leq C \theta_R(\|\textbf{u}\|_{\textbf{W}^{1,\infty}})\Big(  (1+(1+k^2)^{\frac{1}{2}}+ C \epsilon^{-2} ) (\| n\|_{L^\infty}+\|\nabla c\|_{L^\infty})  (\| n\|_{H^{s}}+\| c \|_{H^{s }} )\\
&\quad\quad+(\| n\|_{L^\infty}+\| c\|_{L^\infty})  (\| n\|_{H^{s}}+\| c \|_{H^{s }} ) +\epsilon^{-s}\|\nabla \phi\|_{L^\infty}\left\|n\right\|_{H^s}\Big)\\
&\quad\leq C \left( (1+(1+k^2)^{\frac{1}{2}}+  \epsilon^{-2} ) R + \epsilon^{-s}\|\nabla \phi\|_{L^\infty} \right)(1+\left\|\textbf{u}\right\|_{\textbf{H}^s}).
\end{split}
\end{equation}
In view of the definition of $\widetilde{\textbf{F}}^{k,R,\epsilon}(\textbf{u})$ (cf. \eqref{3.3}), we get from \eqref{a2}-\eqref{a9} that
\begin{equation}\label{f1}
\begin{split}
\|\widetilde{\textbf{F}}^{k,R,\epsilon}(\textbf{u}) \|_{\textbf{H}^s} 
& \leq C(k^2+k^\alpha) \|\textbf{u}\|_{\textbf{H}^s} +C k R \left\|\textbf{u}\right\| _{\textbf{H}^s} +C \left( (1+(1+k^2)^{\frac{1}{2}}+   \epsilon^{-2} ) R + \epsilon^{-s}\|\nabla \phi\|_{L^\infty} \right)\left\|\textbf{u}\right\|_{\textbf{H}^s}\\
& \leq C(\epsilon,\phi,k,R)\left(1+\left\|\textbf{u}\right\|_{\textbf{H}^s} \right),
\end{split}
\end{equation}
which implies the first inequality in \eqref{3.5}.}

In a similar manner, one can also verify that for each $r>0$
\begin{equation}\label{3.6}
\begin{split}
&\left\|\widetilde{\textbf{F}}^{k,R,\epsilon}(\textbf{u}_1 ) -\widetilde{\textbf{F}}^{k,R,\epsilon}(\textbf{u}_2 ) \right\|_{\textbf{H}^s} \leq C(r,R,k,\phi,\epsilon)  \left\|\textbf{u}_1 -\textbf{u}_2\right \|_{\textbf{H}^s} ,\\
& \left\|\textbf{G}(\textbf{u}_1 )-\textbf{G}(\textbf{u}_2 ) \right\|_{L_2(\textbf{U};\textbf{H}^s)} \leq C( \epsilon)  \left\|\textbf{u}_1 -\textbf{u}_2 \right\|_{\textbf{H}^s},
\end{split}
\end{equation}
with $\|\textbf{u}_1\|_{\textbf{H}^s}\leq r$ and $\|\textbf{u}_2\|_{\textbf{H}^s}\leq r$.

As a consequence of the estimates \eqref{3.5} and \eqref{3.6}, for each initial data $\mathcal {J}_k \textbf{u}^\epsilon_0 \in \textbf{H}^s(\mathbb{R}^2)$ and $T>0$, one can conclude from the well-known Leha-Ritter Theorem (cf. \cite[Theorems 5.1.1-5.1.2]{kallianpur1995stochastic}) that the SDEs \eqref{3.3} has a unique global strong solution $\textbf{u}^{\epsilon,k,R}$ in $\textbf{H}^s(\mathbb{R}^2)$.

\emph{Step 2 (Continuity in time).} We are going to show that $  \textbf{u} ^{k,R,\epsilon}\in C([0,T];\textbf{H}^s(\mathbb{R}^2))$, $\mathbb{P}$-a.s. Indeed, by using \eqref{3.5} and applying the It\^{o} formula to $\|\textbf{u}^{k,R,\epsilon}(t) \|_{\textbf{H}^s}^4$, we find
\begin{equation}\label{r1}
\begin{split}
 \|\textbf{u}^{k,R,\epsilon}(t) \|_{\textbf{H}^s}^4\leq& \|\textbf{u}^\epsilon_0 \|_{\textbf{H}^s}^4+ C(\epsilon,\phi,k,R) \int_0^t (1+ \|\textbf{u}^{k,R,\epsilon} \|_{\textbf{H}^s} ^4 )\textrm{d}r +4 \int_0^t\|\textbf{u}^{k,R,\epsilon}  \|_{\textbf{H}^s}^2 (\textbf{u}^{k,R,\epsilon},\textbf{G}(\textbf{u}^{k,R,\epsilon} ))_{\textbf{H}^s}\textrm{d} \mathcal {W}_r.
\end{split}
\end{equation}
By employing the Burkholder-Davis-Gundy (BDG) inequality (cf. \cite[Chapter 4]{1991Continuous}) to the stochastic integral on the R.H.S. of \eqref{r1}, we get
\begin{equation*}
\begin{split}
&\mathbb{E}\sup_{r\in [0,t]}\left|\int_0^r\|\textbf{u}^{k,R,\epsilon}  \|_{\textbf{H}^s}^2 (\textbf{u}^{k,R,\epsilon},\textbf{G}(\textbf{u}^{k,R,\epsilon} )\textrm{d} \mathcal {W}_r)_{\textbf{H}^s}\right|\\
&\quad \leq C\mathbb{E} \left(\int_0^t\|\textbf{u}^{k,R,\epsilon}  \|_{\textbf{H}^s}^6 \|\textbf{G}(\textbf{u}^{k,R,\epsilon} )\|_{L_2(\textbf{U};\textbf{H}^s)}^2\textrm{d}r\right)^{1/2}\\
&\quad \leq \frac{1}{8} \mathbb{E}\sup_{r\in [0,t]}\|\textbf{u}^{k,R,\epsilon}(r) \|_{\textbf{H}^s}^4+ C \mathbb{E} \int_0^t (1+ \|\textbf{u}^{k,R,\epsilon} \|_{\textbf{H}^s} ^4 )\textrm{d}r,
\end{split}
\end{equation*}
which combined with \eqref{r1} leads to
\begin{equation*}
\begin{split}
\mathbb{E}\sup_{r\in [0,t]}\|\textbf{u}^{k,R,\epsilon}(r) \|_{\textbf{H}^s}^4\leq& 2\|\textbf{u}^\epsilon_0 \|_{\textbf{H}^s}^4+ C(\epsilon,\phi,k,R) \mathbb{E}\int_0^t (1+ \|\textbf{u}^{k,R,\epsilon} \|_{\textbf{H}^s} ^4 )\textrm{d}r.
\end{split}
\end{equation*}
Thanks to the Gronwall Lemma, we arrive at
\begin{equation}\label{r2}
\begin{split}
\mathbb{E}\sup_{t\in [0,T]}\|\textbf{u}^{k,R,\epsilon}(t) \|_{\textbf{H}^s}^4\leq 2\|\textbf{u}^\epsilon_0 \|_{\textbf{H}^s}^4 \exp\{C(\epsilon,\phi,k,R)T \}.
\end{split}
\end{equation}
According to \eqref{3.3}, for any $r,t\in [0,T]$, we have
\begin{equation*}
\begin{split}
  \mathbb{E}  \|\textbf{u}^{k,R,\epsilon}(t)-\textbf{u}^{k,R,\epsilon}(r) \|_{\textbf{H}^s}^4&\leq 8\mathbb{E}\left\|\int_r^t\widetilde{\textbf{F}}^{k,R,\epsilon}(\textbf{u}^{k,R,\epsilon} )\textrm{d} \varsigma\right\|_{\textbf{H}^s}^4+8\mathbb{E}  \left\|\int_r^t\textbf{G}(\textbf{u}^{k,R,\epsilon}  )\textrm{d} \mathcal {W}_\varsigma\right\|_{\textbf{H}^s}^4\\
&:= \mathcal {A}(t,r)+\mathcal {B}(t,r).
\end{split}
\end{equation*}
For $\mathcal {A}(t,r)$, it follows from \eqref{3.5} and \eqref{f1} that
\begin{equation*}
\begin{split}
\mathcal {A}(t,r)
&\leq C(\epsilon,\phi,k,R)\bigg(1+ \mathbb{E}\sup_{t\in [0,T]} \|\textbf{u}^{k,R,\epsilon} (t)\|_{\textbf{H}^s}^4  \bigg) |t-r|^{4}\\
&\leq C(\epsilon,\phi,k,R)\big(1+2\|\textbf{u}^\epsilon_0 \|_{\textbf{H}^s}^4 \exp\{C(\epsilon,\phi,k,R)T \} \big) |t-r|^{4}.
\end{split}
\end{equation*}
For $\mathcal {B}(t,r)$, by using the BDG inequality, we get from the estimates in \eqref{3.5} and \eqref{r2} that
\begin{equation*}
\begin{split}
\mathcal {B}(t,r)\leq  C\mathbb{E}\left(\int_r^t \|\textbf{G}(\textbf{u}^{k,R,\epsilon}(\varsigma)  ) \|_{L_2(\textbf{U};\textbf{H}^s)}^2\textrm{d} \varsigma\right)^2
\leq  C
\big(1+2\|\textbf{u}^\epsilon_0 \|_{\textbf{H}^s}^4 \exp\{C(\epsilon,\phi,k,R)T \} \big)|t-r|^2 .
\end{split}
\end{equation*}
Therefore, it follows from the last three estimates that
\begin{equation*}
\begin{split}
  \mathbb{E}  \left( \|\textbf{u}^{k,R,\epsilon}(t)-\textbf{u}^{k,R,\epsilon}(r) \|_{\textbf{H}^s}^4 \right)\leq C(\textbf{u}^\epsilon_0,\epsilon,\phi,k,R) |t-r|^{2}.
\end{split}
\end{equation*}
By applying the Kolmogorov-\v{C}entsov Continuity Theorem (cf. \cite[Theorem 2.8]{karatzas1991brownian}), there exists a continuous modification, say $\widetilde{\textbf{u}^{k,R,\epsilon}}$, for the solution $\textbf{u}^{k,R,\epsilon}$ on $[0,T]$. In the following, we still denote this continuous version by $\textbf{u}^{k,R,\epsilon}$ for simplicity. The proof of Lemma \ref{lem1} is now completed.
\end{proof}

\subsection{A priori estimates}

\begin{lemma} \label{lem3}
Let $s>5$, $R>1$, $\epsilon>0$ and $\alpha\in[\frac{1}{2},1]$. Assume that the conditions (H1)-(H3) hold. Let $\textbf{u}^{k,R,\epsilon}$ be the solution to \eqref{3.3} with respect to the initial data $\textbf{u} ^{k, \epsilon}_0$. Then for any $T>0$ and $p\geq2$, there is a positive constant $C $ independent of $k \in \mathbb{N}^+$ such  that
\begin{equation}\label{3.9}
\begin{split}
\sup_{k \in\mathbb{N}^+} \mathbb{E} \sup_{t\in [0,T]}\|\textbf{u}^{k,R,\epsilon}(t)\|^p_{\textbf{H}^s} \leq C,
\end{split}
\end{equation}
and for all $\theta \in (0,\frac{p-2}{2p} )$
\begin{equation}\label{3.10}
\begin{split}
& \sup_{k \in\mathbb{N}^+}\mathbb{E}\left( \|(n^{k,R,\epsilon},c^{k,R,\epsilon})\|^p_{\emph{\textrm{Lip}}([0,T]; \textbf{H}^{s-2})} + \|u^{k,R,\epsilon} \|^p_{C^\theta([0,T]; H^{s-2\alpha} )} \right)  \leq C.
\end{split}
\end{equation}
\end{lemma}

\begin{proof}  \emph{Step 1 ($L^2$ momentum estimate).}  Applying $\Lambda^s= (1-\Delta)^{\frac{s}{2}}$ ($ s> 2$) to both sides of $n $-equation in \eqref{3.3}, and taking the scalar product with  $\Lambda^s n $ over $\mathbb{R}^2$, we get
\begin{equation}\label{3.11}
\begin{split}
\frac{1}{2}  \textrm{d} \| n\|^2_{H^s}+\|\nabla \mathcal {J}_kn\|^2_{H^s} \textrm{d}t\leq& \|\mathcal {J}_kn\|_{H^s}^2\textrm{d}t -   \theta_R\left(\|\textbf{u}\|_{\textbf{W}^{1,\infty}}\right)\left( \Lambda^s \mathcal {J}_kn,\Lambda^s ( \mathcal {J}_kn( n*\rho^\epsilon ))\right)_{L^2}\textrm{d}t \\
&- \theta_R\left(\|\textbf{u}\|_{\textbf{W}^{1,\infty}}\right)\left( \Lambda^s \mathcal {J}_kn,\Lambda^s \left(\mathcal {J}_k u\cdot \nabla \mathcal {J}_k n \right)\right)_{L^2}\textrm{d}t\\
 &-  \theta_R\left(\|\textbf{u}\|_{\textbf{W}^{1,\infty}}\right)\left( \Lambda^s \nabla \mathcal {J}_k n,\Lambda^s  \left(\mathcal {J}_kn(\nabla  c*\rho^\epsilon)\right)\right)_{L^2}\textrm{d}t\\
:=&  \|\mathcal {J}_kn\|_{H^s}^2 \textrm{d}t + (L_1+L_2+L_3) \textrm{d}t.
\end{split}
\end{equation}
Here and in the sequel, we omit the superscript $k,R$ and $\epsilon$ in $\textbf{u}^{k,R,\epsilon} $ for simplicity.

For $L_1$, we have
\begin{equation}\label{3.12}
\begin{split}
 L_1 &\leq C \theta_R\left(\|\textbf{u}\|_{\textbf{W}^{1,\infty}}\right)\| \mathcal {J}_kn\|_{H^s}\| \mathcal {J}_kn( n*\rho^\epsilon) \|_{H^s} \leq C R\| n \|_{H^s}^2.
\end{split}
\end{equation}
For $L_2$, by using the commutator estimate  {(cf.  {\cite[Lemma 2.99 and Lemma 2.100]{bahouri2011fourier}})}, we get
\begin{equation}\label{3.13}
\begin{split}
 L_2 &\leq \theta_R\left(\|\textbf{u}\|_{\textbf{W}^{1,\infty}}\right)\|[\Lambda^s \left(\mathcal {J}_k u\cdot \nabla \mathcal {J}_k n \right)- \left(\mathcal {J}_k u\cdot \nabla \Lambda^s \mathcal {J}_k n \right)]\|_{L^2}\| \Lambda^s \mathcal {J}_kn\|_{L^2}\\
 &\leq C\theta_R\left(\|\textbf{u}\|_{\textbf{W}^{1,\infty}}\right)\| \mathcal {J}_kn\|_{H^s}(\|\mathcal {J}_k u\|_{H^s}\|\nabla \mathcal {J}_k n\|_{L^\infty}+\|\nabla \mathcal {J}_k u\|_{L^\infty}\|\mathcal {J}_k n\|_{H^s})\\
 &\leq CR \left(\|\mathcal {J}_k u\|_{H^s}^2+\|\mathcal {J}_k n\|_{H^s}^2\right).
\end{split}
\end{equation}
For $L_3$, we get from the Moser-type estimate  that
\begin{equation}\label{3.14}
\begin{split}
 L_3 &\leq C \theta_R\left(\|\textbf{u}\|_{\textbf{W}^{1,\infty}}\right)\left( \|\Lambda^s \nabla \mathcal {J}_k n\| _{L^2}^2+\|\Lambda^s \left(\mathcal {J}_kn( \nabla c *\rho)\right)\|  _{L^2}^2\right)\\
 & \leq  \frac{1}{2} \|\nabla \mathcal {J}_k n\| _{H^s}^2+ C( \epsilon)  \theta_R\left(\|\textbf{u}\|_{\textbf{W}^{1,\infty}}\right) \big(\|n\|_{L^\infty}\|\nabla \mathcal {J}_kc\|_{H^{s-1}} +\|n\|_{H^{s}}\| \nabla \mathcal {J}_kc\|_{L^\infty} \big)^2\\
 & \leq  \frac{1}{2} \|\nabla  \mathcal {J}_kn\| _{H^s}^2+  {C(\epsilon) R^2} \big( \|c\|_{H^s}^2 +\|n\|_{H^s} ^2 \big).
\end{split}
\end{equation}
Plugging the estimates \eqref{3.12}-\eqref{3.14} into \eqref{3.11}, it follows that
\begin{equation}\label{3.15}
\begin{split}
\mathbb{E}\sup_{r\in [0,t]}\| n(r)\|^2_{H^s}+ \mathbb{E}\int_0^t\|\nabla \mathcal {J}_kn\|^2_{H^s}\textrm{d}r \leq \| n_0\|^2_{H^s}+ {C( \epsilon)(R+R^2)}  \mathbb{E}\int_0^t\big( \|c\|_{H^s}^2 +\|n\|_{H^s} ^2 \big)\textrm{d}r.
\end{split}
\end{equation}
In a similar manner, we have
\begin{equation}\label{3.16}
\begin{split}
\mathbb{E}\sup_{r\in [0,t]}\|c(t)\|^2_{H^s}+ \mathbb{E}\int_0^t\|\nabla \mathcal {J}_kc\|^2_{H^s}\textrm{d}r\leq \| c_0\|^2_{H^s}+C(\epsilon, R )  \mathbb{E}\int_0^t\big( \|n\|_{H^s} ^2 +\|c\|_{H^s}^2 \big)\textrm{d}r.
\end{split}
\end{equation}

To deal with the $u$-equation, we 
apply  the It\^{o} formula  to $\|\Lambda^s u(t)\|_{L^2}^2$ to find
\begin{equation}\label{3.17}
\begin{split}
& \| \Lambda^s u(t)\|_{L^2}^2 +2 \int_0^t\|(-\Delta)^\frac{\alpha}{2}  \Lambda^s \mathcal {J}_ku\|_{L^2 }^2\textrm{d}r =\| \Lambda^s u_0\|_{L^2}^2+ \int_0^t(K_1 +K_2 +K_3 ) \textrm{d}r+\int_0^tK_4  \textrm{d} W_r,
\end{split}
\end{equation}
where $ K_1 =    \| \Lambda^s\textbf{P} f (t,u)\|_{L_2(U;L^2)}^2$, $
 K_2 =2 (\Lambda^s \mathcal {J}_k u, \Lambda^s\textbf{P}\mathcal {J}_k((n\nabla \phi)*\rho^\epsilon))_{L^2}$, $
 K_3 =- 2  \theta_R\left(\|\textbf{u}\|_{\textbf{W}^{1,\infty}}\right)(\Lambda^s \mathcal {J}_k u$, $ \Lambda^s\textbf{P}(\mathcal {J}_k u\cdot \nabla \mathcal {J}_k u ))_{L^2}$ and $K_4 =2   (\Lambda^s u,   \Lambda^s\textbf{P} f(t,u))_{L^2}$.

For $K_1 $, we have
\begin{equation}\label{3.18}
\begin{split}
& |K_1 | \leq C \left(1+ \|u\|_{H^s}^2\right) .
\end{split}
\end{equation}
For $K_2 $, we get by the Cauchy inequality that
\begin{equation}\label{3.19}
\begin{split}
  |K_2  | &\leq C  \left(\|\Lambda^s \mathcal {J}_k u\|_{L^2}^2+\| \Lambda^s\textbf{P}\mathcal {J}_k\left((n\nabla \phi)*\rho^\epsilon\right)\|_{L^2}^2\right) \leq C \|\nabla\phi\|_{L^\infty} \left(\|u\|_{H^s}^2+ \|n\|_{H^s}^2\right).
\end{split}
\end{equation}
For $K_3 $, we first note that
$$
\left(\Lambda^s \mathcal {J}_k u,\textbf{P}\Lambda^s\left(\mathcal {J}_k u\cdot \nabla \mathcal {J}_k u \right)\right)_{L^2}= (\Lambda^s \mathcal {J}_k u, \textbf{P}[\Lambda^s,\mathcal {J}_k u \cdot \nabla] \mathcal {J}_k u)_{L^2},
$$
due to the incompressible condition $ \textrm{div} \mathcal {J}_k u = \mathcal {J}_k\textrm{div} u=0$, where $[A,B]=AB-BA$. Then we get by applying the commutator estimates that
\begin{equation}\label{3.20}
\begin{split}
  |K_3  | &\leq C  \theta_R\left(\|\textbf{u}\|_{\textbf{W}^{1,\infty}}\right)\|\nabla \mathcal {J}_ku\|_{L^\infty}\|\Lambda^s \mathcal {J}_k u\| _{L^2}\|\mathcal {J}_ku\| _{H^{s}} \leq CR \|u\| _{H^{s}}^2.
\end{split}
\end{equation}
For $K_4 $, it follows from the BDG inequality that
\begin{align}\label{3.21}
\mathbb{E}\sup_{r\in [0,t]}\left| \int_0^r K_4   \textrm{d} W_t\right|&\leq C\mathbb{E}\left(\sum_{j\geq 1}\int_0^ t (\Lambda^s u(r),   \Lambda^s\textbf{P} f_j(r,u))_{L^2}^2 \textrm{d}r\right)^{1/2}\nonumber\\
&\leq C\mathbb{E}\left[\sup_{r\in [0,t]}\|\Lambda^s u(r)\|_{L^2}\left( \sum_{j\geq 1}\int_0^ t\|\Lambda^s\textbf{P} f_j(r,u)\|_{L^2}^2 \textrm{d}r\right)^{1/2}\right]\nonumber\\
&\leq \frac{1}{2} \mathbb{E}\sup_{r\in [0,t]}\|\Lambda^s u(r)\|_{L^2}^2+ C  \mathbb{E} \int_0^ t \left(1+\|u(r)\|_{H^{s}}^2\right) \textrm{d}r.
\end{align}
 {Plugging the estimates \eqref{3.18}-\eqref{3.21} into \eqref{3.17}},  we get
\begin{equation*}
\begin{split}
 &\mathbb{E}\sup_{r\in [0,t]}\| u(r)\|_{H^s}^2+4  \mathbb{E}\int_0^t\|(-\Delta)^\frac{\alpha}{2}  \mathcal {J}_ku\|_{H^s}^2\textrm{d}t \\
 &\quad \leq 2 \| u_0\|_{H^s}^2+      {C(1+R+\|\nabla\phi\|_{L^\infty})} \mathbb{E} \int _0^t\left(1+ \|(u,n)\|_{\textbf{H}^s}^2\right) \textrm{d}r,
\end{split}
\end{equation*}
which together with \eqref{3.15} and \eqref{3.16} imply
\begin{equation*}
\begin{split}
\mathbb{E}\sup_{r\in [0,t]}\|\textbf{u}(r)\|^2_{\textbf{H}^s} \leq 2\|\textbf{u}(0)\|^2_{\textbf{H}^s}+C( R,\phi, \epsilon)  \mathbb{E}\int_0^t\left(1+\|\textbf{u} \|^2_{\textbf{H}^s}\right)\textrm{d}r.
\end{split}
\end{equation*}
By the Gronwall Lemma, we get
$$
\mathbb{E}\sup_{t\in [0,T]}\|\textbf{u}(t)\|^2_{\textbf{H}^s} \leq C \exp\{C( R,\phi, \epsilon) T\}\|\textbf{u}(0)\|^2_{\textbf{H}^s},\quad\forall T>0.
$$

\emph{Step 2 ($L^p$ momentum estimate).} By applying the chain rule to $\| n(t)\|^p_{H^s}=(\| n(t)\|^2_{H^s})^{ \frac{p}{2}}$, we find
\begin{equation}\label{3.22}
\begin{split}
 \| n(t)\|^p_{H^s}=& \| n_0\|^p_{H^s} -p\int_0^t\| n \|_{H^s}^{p-2}\|\nabla \mathcal {J}_kn\|^2_{H^s} \textrm{d}r +  p\int_0^t\| n \|_{H^s}^{p-2} \|\mathcal {J}_kn\|_{H^s}^2 \textrm{d}r \\
&+p\int_0^t\| n \|_{H^s}^{p-2} (L_1+L_2+L_3) \textrm{d}r.
\end{split}
\end{equation}
In view of the estimates obtained in Step 1 for $L_i$, $i=1,2,3$, we have
\begin{equation}\label{3.23}
\begin{split}
 & |L_1|+|L_2|+|L_3| \leq \frac{1}{2} \|\nabla  \mathcal {J}_kn\| _{H^s}^2 + C( \epsilon)(R+R^2)  \|\textbf{u}\|_{\textbf{H}^s}^2 .
\end{split}
\end{equation}
By \eqref{3.22} and \eqref{3.23}, we get
\begin{equation}\label{3.24}
\begin{split}
& \| n(t)\|^p_{H^s}+\frac{p}{2}\int_0^t\| n \|_{H^s}^{p-2}\|\nabla \mathcal {J}_kn\|^2_{H^s} \textrm{d}r \\
&\quad\leq  \| n_0\|^p_{H^s} +p \int_0^t\| n \|_{H^s}^{p} \textrm{d}r+  {p C( \epsilon)(R+R^2)} \int_0^t\| n \|_{H^s}^{p-2}   \|\textbf{u} \|_{\textbf{H}^s} ^2   \textrm{d}r\\
& \quad \leq \| n_0\|^p_{H^s} +  {p C( \epsilon) (R+R^2)}  \int_0^t \| \textbf{u} \|_{\textbf{H}^s}^{p}  \textrm{d}r.
\end{split}
\end{equation}
Similarly,
\begin{equation}\label{3.25}
\begin{split}
& \| c(t)\|^p_{H^s}+\frac{p}{2}\int_0^t\| c \|_{H^s}^{p-2}\|\nabla \mathcal {J}_k c\|^2_{H^s} \textrm{d}r \leq \| n_0\|^p_{H^s} +C(p,\epsilon, R ) \int_0^t \| \textbf{u} \|_{\textbf{H}^s}^{p} \textrm{d}r.
\end{split}
\end{equation}
Now we apply the It\^{o} formula  to  {$ \|\Lambda^s u(t)\|^p_{L^2}= (\|\Lambda^s u(t)\|^2_{L^2})^{p/2}$}, it follows from \eqref{3.13} that
\begin{equation}
\begin{split}\label{3.26}
& \|\Lambda^s u(t)\|^p_{L^2}+ p \int_0^t\|\Lambda^s u \| _{L^2} ^{p-2} \|(-\Delta)^\frac{\alpha}{2}  \Lambda^s \mathcal {J}_ku\|_{L^2 }^2\textrm{d}r\\
&\quad= \|\Lambda^s u_0\|^p_{L^2} + \frac{p}{2} \int_0^t\|\Lambda^s u \| _{L^2} ^{p-2} (K_1 +K_2 +K_3 ) \textrm{d}r\\
 &\quad\quad+\frac{p(p-2)}{4}\sum_{j \geq 1} \int_0^t \|\Lambda^s u \|_{L^2} ^{p-4} \left(\Lambda^s u,   \Lambda^s\textbf{P} f(r,u)e_j\right)_{L^2}^2 \textrm{d}r\\
 &\quad\quad+\frac{p}{2}\sum_{j \geq 1}\int_0^t  \|\Lambda^s u \| _{L^2} ^{p-2}  \left(\Lambda^s u,   \Lambda^s\textbf{P} f(r,u)e_j\right)_{L^2} \textrm{d} W^j_r \\
 &\quad:= \|\Lambda^s u_0\|^p_{L^2} + J_1+J_2+J_3,
\end{split}
\end{equation}
where $K_i $, $i=1,2,3$, are defined in \eqref{3.17}. From the estimates \eqref{3.18}-\eqref{3.20}, we deduce that
\begin{equation*}
\begin{split}
 \int_0^t\|\Lambda^s u\| _{L^2} ^{p-2} K_1 \textrm{d}r &\leq C  \int_0^t\|\Lambda^s u\| _{L^2} ^{p-2}\left(1+ \|u\|_{H^s}^2\right) \textrm{d}r \leq C p  \int_0^t\left(1+ \|u\|_{H^s}^p\right)\textrm{d}r,\\
 \int_0^t\|\Lambda^s u\| _{L^2} ^{p-2} K_2 \textrm{d}r &\leq  C \|\nabla\phi\|_{L^\infty}  \int_0^t \| u\| _{H^S} ^{p-2}\left(\|u\|_{H^s}^2+ \|n\|_{H^s}^2\right)\textrm{d}r\leq C p \|\nabla\phi\|_{L^\infty} \int_0^t \|(n,u)\|_{\textbf{H}^s}^p \textrm{d}r,\\
 \int_0^t\|\Lambda^s u\| _{L^2} ^{p-2} K_3 \textrm{d}r &\leq C R   \int_0^t\|u\| _{H^{s}}^p \textrm{d}r.
\end{split}
\end{equation*}
Hence, the term $J_1$ can be estimated as
\begin{equation}\label{3.27}
\begin{split}
 |J_1|\leq   {Cp^2 (1+R+\|\nabla\phi\|_{L^\infty})}  \int_0^t\left(1+\|n\|_{H^s}^p+ \|u\|_{H^s}^p\right)\textrm{d}r.
\end{split}
\end{equation}
For $J_2$, we get by the Young inequality that
\begin{equation}\label{3.28}
\begin{split}
|J_2|&\leq \sum_{j \geq 1}\int_0^t \|\Lambda^s u \|_{L^2} ^{p-2}\| \Lambda^s\textbf{P} f(r,u)e_j \|_{L^2} ^2  \textrm{d}r\\
&\leq C \int_0^t \|\Lambda^s u \|_{L^2} ^{p-2}  \left(1+ \|\Lambda^s u \|_{L^2}^2\right) \textrm{d}r \leq C p \int_0^t \left(1+ \| u \|_{H^s}^p\right) \textrm{d}r.
\end{split}
\end{equation}
For $J_3$, we obtain by using the BDG inequality that
\begin{equation}
\begin{split}\label{3.29}
\mathbb{E}\sup_{r \in [0,t]}|J_3|&\leq  {Cp}\mathbb{E}\left(\sum_{j \geq 1}\int_0^t\|\Lambda^s u(r)\| _{L^2} ^{2p-4}   \left(\Lambda^s u,   \Lambda^s\textbf{P} f(r,u)e_j\right)_{L^2}^2 \textrm{d}r\right)^{\frac{1}{2}}\\
&\leq Cp\mathbb{E}\left(\sup_{r\in [0,t]}\|\Lambda^s u(r)\|^p_{L^2}\int_0^t  \|\Lambda^s u \| _{L^2} ^{ p-2} (1+\|\Lambda^s u\|^2_{L^2})\textrm{d}r\right)^{\frac{1}{2}}\\
  & \leq  \frac{1}{2}\mathbb{E} \sup_{r\in [0,t]}\|\Lambda^s u(r)\|^p_{L^2}  + Cp\mathbb{E}\int_0^t\left (1+ \|\Lambda^s u \|^p_{L^2}\right) \textrm{d}r.
\end{split}
\end{equation}
Thereby, by taking the supremum over $[0,t]$ on both sides of \eqref{3.26}, we get from the estimates \eqref{3.27}-\eqref{3.29} that
\begin{equation*}
\begin{split}
&\frac{1}{2} \mathbb{E}\sup_{r \in [0,t]}\|\Lambda^s u(r)\|^p_{L^2}+ p \mathbb{E}\int_0^t\|\Lambda^s u \| _{L^2} ^{p-2} \left\|(-\Delta)^\frac{\alpha}{2}  \Lambda^s \mathcal {J}_ku\right\|_{L^2 }^2\textrm{d}r\\
&\quad \leq \| u_0\|^p_{H^s} +C(\phi,p,R) \int_0^t\left(1+\|(n,u)\|_{\textbf{H}^s}^p\right)\textrm{d}r,
\end{split}
\end{equation*}
which together with \eqref{3.24} and \eqref{3.25} leads to
\begin{equation*}
\begin{split}
  \mathbb{E}\sup_{r \in [0,t]}\|\textbf{u}(r)\|^p_{\textbf{H}^s}\leq \|\textbf{u}(0)\|^p_{\textbf{H}^s} + C(p, R, \phi,\epsilon) \int_0^t \| \textbf{u} \|_{\textbf{H}^s}^{p} \textrm{d}r,\quad \forall t \in [0,T].
\end{split}
\end{equation*}
In view of the Gronwall Lemma, we get
\begin{equation*}
\begin{split}
 \mathbb{E}\sup_{t \in [0,T]}\|\textbf{u}(t)\|^p_{\textbf{H}^s} \leq \|\textbf{u}(0)\|^p_{H^s} \exp\{C(p, R, \phi,\epsilon) T\},\quad \forall T> 0.
\end{split}
\end{equation*}

\emph{Step 3 (H\"{o}lder continuous).}  Since $\textbf{u}=(n,c,u)$ is uniformly bounded in $L^p(\Omega; C([0,T];\textbf{H}^s(\mathbb{R}^2)))$, it follows from the equations \eqref{3.3}$_1$ and \eqref{3.3}$_2$ that
\begin{equation}\label{3.30}
\begin{split}
(n,c)\in L^p\left(\Omega; \textrm{Lip}([0,T]; \textbf{H}^{s-2}(\mathbb{R}^2) \right),\quad \forall T> 0.
\end{split}
\end{equation}
Next we show that $u $ is H\"{o}lder continuous in time.  Indeed, we get from the $u$-equation that
\begin{equation}\label{3.31}
\begin{split}
\|u(t)-u(r)\|_{H^{s-2\alpha}}\leq &\left\|\int_r^t(-\Delta)^\alpha \mathcal {J}_k^2u\textrm{d} \varsigma \right\|_{H^{s-2\alpha}}+ \left\|\int_r^t\textbf{P}\mathcal {J}_k ((n\nabla \phi)*\rho  )\textrm{d} \varsigma\right\|_{H^{s}}\\
 &+\left\|\int_r^t\theta_R\left(\|\textbf{u}\|_{\textbf{W}^{1,\infty}}\right)\textbf{P}\mathcal {J}_k  (\mathcal {J}_k u \cdot \nabla \mathcal {J}_k u  ) \textrm{d} \varsigma\right\|_{H^{s-1}}\\
 &+\left\|\int_r^t\theta_R\left(\|\textbf{u}\|_{\textbf{W}^{1,\infty}}\right)\textbf{P} f(\varsigma,u ) \textrm{d} W_\varsigma\right\|_{H^{s }} \\
:=& M_1+M_2+M_3+M_4(r,t).
\end{split}
\end{equation}
First, we get by the uniform bound in Step 2 that
$$
M_1+M_2+M_3\leq C(\phi,R)\int_r^t  (1+\|n\|_{H^s}+\|u\|_{H^s} ) \textrm{d}r,
$$
which implies that
\begin{equation}\label{3.32}
\begin{split}
\mathbb{E}\left(M_1+M_2+M_3\right)^p
&\leq C( \phi,R)^p\mathbb{E}  \sup_{\varsigma\in [r,t]}\left(1+\|(n,u)(\varsigma)\|_{\textbf{H}^s}^p \right) |t-r|^p \\
&\leq C( \phi,R)^p\|\textbf{u}(0)\|^p_{H^s}\exp\{C(p, R,\phi, \epsilon )T \} |t-r|^p.
\end{split}
\end{equation}
Second, for any $\gamma>0$, there is a subinterval $[r',t']\subset [r,t]$ such that
$
\sup_{t\neq r}\frac{M_4(r,t)}{|t-r|^\sigma}< \sup_{t\neq r}\frac{M_4(r',t')}{|t'-r'|^\sigma}+\gamma^{\frac{1}{p}}.
$
By applying the BDG inequality, we obtain
\begin{equation*}
\begin{split}
\mathbb{E} \left(\sup_{t\neq r}\frac{M_4(r,t)}{|t-r|^\sigma}\right) ^p 
&\leq 2^{p-1}   \left(\frac{\mathbb{E}(\int_{r'}^{t'} \|f(t,u )\|_{L_2(U;L^2)}^2 \textrm{d}r)^\frac{p}{2}}{|t'-r'|^{\sigma p}} + \gamma\right)\\
&\leq C2^{p-1}\left( \frac{\mathbb{E}(\int_{r'}^{t'} (1+ \|u \|_{H^s}^2) \textrm{d}r)^\frac{p}{2}}{|t'-r'|^{\sigma p}} + \gamma\right)\\
&\leq  C2^{p-1} \left(|t'-r'|^{\frac{p}{2}-\sigma p} \mathbb{E}\sup_{t\in [0,T]}\left(1+ \|u(t)\|_{H^s}^p\right)  + \gamma\right)\\
& \leq  {C2^{p-1}\|\textbf{u}(0)\|^p_{H^s} \exp\{C(p, R, \phi,\epsilon) T\} |t'-r'|^{\frac{p}{2}-\sigma p}+C2^{p-1}\gamma.}
\end{split}
\end{equation*}
Since $\gamma>0$ is arbitrary and $\frac{p}{2}-\sigma p>0$, the above estimates implies that
\begin{equation}\label{3.33}
\begin{split}
\mathbb{E}  \left\|\int_r^t \textbf{P} f(\varsigma,u ) \textrm{d} W_\varsigma\right\|_{H^{s }}^p \leq C(p, \textbf{u}_0,R, \epsilon,\phi,T) |t-r|^{\sigma p}.
\end{split}
\end{equation}
Combining \eqref{3.32} with \eqref{3.33} leads to
$$
\mathbb{E}  \|u(t)-u(r)\|_{H^{s-2\alpha}} ^p  \leq C(p, \textbf{u}_0,R, \epsilon,\phi,T) |t-r|^{\sigma p}.
$$
According to the Kolmogorov-\v{C}entsov Continuity Theorem (cf. \cite[Theorem 2.8]{karatzas1991brownian}), $u(t)$ has a H\"{o}lder continuity version in $C^\theta([0,T];H^{s-2\alpha}(\mathbb{R}^2))$ with  $  0\leq \theta\leq\sigma -\frac{1}{p} \leq\frac{1}{2}-\frac{1}{p}$. Moreover, there holds
$$
\mathbb{E} \|u\|_{C^\theta([0,T];H^{s-2\alpha})} ^p  \leq C(p, \textbf{u}_0,R, \epsilon,\phi,T).
$$
The proof of Lemma \ref{lem3} is now completed.
\end{proof}

\subsection{Convergence in $k\in \mathbb{N}^+$}

The aim of this subsection is to show that, for fixed $R>0$ and $0<\epsilon < 1$, the family of $\{\mathbf{u}^{k,R,\epsilon}\}_{k \in \mathbb{N}}$ contains a subsequence that converges in $C([0, T], \textbf{H}^{s}(\mathbb{R}^2))$ almost surely. To this end, we shall first prove the convergence in $C([0, T], \textbf{H}^{s-3}(\mathbb{R}^2))$ by adopting the ideas in \cite{li2021stochastic}, and then achieve the goal by raising the spacial regularity of solutions in $\textbf{H}^{s}(\mathbb{R}^2)$.

For each $k, {\ell\in \mathbb{N}^+}$, it follows from \eqref{3.3} that
\begin{equation}\label{3.34}
  \quad\left\{
\begin{aligned}
&\textrm{d} \textbf{u}^{k, {\ell}}(t)
=\left(\widetilde{\textbf{F}}^k(\textbf{u}^{k})-\widetilde{\textbf{F}}^{ {\ell}}
(\textbf{u}^{ {\ell}} )\right) \textrm{d}t
+\left( \textbf{G} (t,\textbf{u}^{k} )- \textbf{G} (t,\textbf{u}^{ {\ell}} )\right)
\textrm{d} \mathcal {W}_t ,\\
&\textbf{u}^{k, {\ell}}(0)=\textbf{0}.
\end{aligned}
\right.
\end{equation}
 {Here for simplicity, since the parameters $k$ and $R$ are fixed, we have used  the notations $\mathbf{u}^{k}:=\mathbf{u}^{k, R,\epsilon}$, $\mathbf{u}^{k, {\ell}}:=\mathbf{u}^{k, {\ell};R,\epsilon}=\mathbf{u}^{k,R,\epsilon}-\mathbf{u}^{\ell,R,\epsilon}$, and write $\mathbf{u}^{k}$ and $\widetilde{\textbf{F}}^{k}(\cdot)$ instead of $\mathbf{u}^{k,R,\epsilon}$ and $\widetilde{\textbf{F}}^{k,R,\epsilon}(\cdot)$, respectively.} Using these notations, the coefficients in \eqref{3.34} are formulated by
\begin{equation*}
\begin{split}
\widetilde{\textbf{F}}^k(\textbf{u}^{k} )-\widetilde{\textbf{F}}^\ell(\textbf{u}^{\ell} )
&= ( \mathcal {J}_\ell ^2-\mathcal {J}_k^2)\textbf{A}^\alpha\textbf{u}^{\ell}+ \mathcal {J}_k^2 \textbf{A}^\alpha  \textbf{u}^{\ell,k}+(\theta_R(\|\textbf{u}^{\ell}\|_{W^{1,\infty}}) -\theta_R(\|\textbf{u}^{k}\|_{W^{1,\infty}}))  \mathcal {J}_\ell \textbf{B}(\mathcal {J}_{\ell}\textbf{u}^{\ell})\\
& + \theta_R(\|\textbf{u}^{k}\|_{\textbf{W}^{1,\infty}})
  \mathcal {J}_\ell  \big(\textbf{B}( \mathcal {J}_\ell \textbf{u}^{k})-\textbf{B}(\mathcal {J}_k\textbf{u}^{k}) \big) +\theta_R(\|\textbf{u}^{k}\|_{\textbf{W}^{1,\infty}})
  \mathcal {J}_\ell  \big(\textbf{B}( \mathcal {J}_\ell \textbf{u}^{\ell})-\textbf{B}( \mathcal {J}_\ell \textbf{u}^{k}) \big)\\
&+ \theta_R(\|\textbf{u}^{k}\|_{\textbf{W}^{1,\infty}})
( \mathcal {J}_\ell -\mathcal {J}_k)\textbf{B}(\mathcal {J}_k\textbf{u}^{k}) + \theta_R(\|\textbf{u}^{\ell}\|_{\textbf{W}^{1,\infty}})(\mathcal {J}_k - \mathcal {J}_\ell )\textbf{F} ^\epsilon ( \mathcal {J}_\ell \textbf{u}^{\ell})  \\
 &+\theta_R(\|\textbf{u}^{\ell}\|_{\textbf{W}^{1,\infty}})\mathcal {J}_k(\textbf{F} ^\epsilon (\mathcal {J}_k\textbf{u}^{\ell})- \textbf{F} ^\epsilon ( \mathcal {J}_\ell \textbf{u}^{\ell}))  +\theta_R(\|\textbf{u}^{\ell}\|_{\textbf{W}^{1,\infty}})(\mathcal {J}_k\textbf{F} ^\epsilon (\mathcal {J}_k\textbf{u}^{k})-\mathcal {J}_k\textbf{F} ^\epsilon (\mathcal {J}_k\textbf{u}^{\ell}))           \\
 &+  (\theta_R(\|\textbf{u}^{k}\|_{\textbf{W}^{1,\infty}})-\theta_R(\|\textbf{u}^{\ell}\|_{\textbf{W}^{1,\infty}})) \mathcal {J}_k\textbf{F} ^\epsilon (\mathcal {J}_k\textbf{u}^{k})  \\
 & := \textbf{p}_1+\cdots +\textbf{p}_{10},
\end{split}
\end{equation*}
and
\begin{equation*}
\begin{split}
  \textbf{G} (t,\textbf{u}^{k} )- \textbf{G} (t,\textbf{u}^{\ell} ) := \textbf{p}_{11} .
\end{split}
\end{equation*}
To get proper estimates for $\textbf{u}^{k,\ell}$, we apply the It\^{o} formula  to $ \|\textbf{u}^{k,\ell}(t)\|_{\textbf{H}^{s-3}}^2$ to obtain
\begin{equation}\label{3.35}
\begin{split}
\|\textbf{u}^{k,\ell}(t)\|_{\textbf{H}^{s-3}}^2\leq & 2\sum_{i=1}^{10}\int_0^t (\textbf{u}^{k,\ell},\textbf{p}_i)_{\textbf{H}^{s-3}} \textrm{d}r
+  \int_0^t\|\textbf{p}_{11}(r)\|_{L_2(\textbf{U}; \textbf{H}^{s-3})}^2 \textrm{d}r \\
&+2 \sum_{j\geq 1} \int_0^t(\textbf{u}^{k,\ell},\textbf{p}_{11}^j(r)
)_{\textbf{H}^{s-3}}\textrm{d} \mathcal {W}^j_r ,
\end{split}
\end{equation}
where $\textbf{p}_{11}^j(r):=\textbf{p}_{11}(r)\textbf{e}_j$,  and  $\{\textbf{e}_j\}_{j\geq 1}$ is the orthogonal basis in $\textbf{U}$.

\begin{lemma}\label{lem3.5}
Let $s>0$. There holds
\begin{equation} \label{3.36}
\begin{split}
\sum_{i=1}^{10} |(\mathbf{u}^{k, {\ell};R,\epsilon},\textbf{p}_i)_{\textbf{H}^{s-3}}|  &\leq C(\epsilon)\left(1+\| \textbf{u}^{k,R,\epsilon}\| _{\textbf{H}^{s}}^2+\| \textbf{u}^{\ell,R,\epsilon}\| _{\textbf{H}^{s}}^2\right)\|\textbf{u}^{k,\ell;,R,\epsilon}\|_{\textbf{H}^{s-3}}^2\\
&+ C(\epsilon)\max\{\frac{1}{k^2},\frac{1}{ \ell^2 }\}\left(\| \textbf{u}^{k,R,\epsilon}\| _{\textbf{H}^{s}}^2+\| \textbf{u}^{\ell,R,\epsilon}\| _{\textbf{H}^{s}}^2 +\|\textbf{u}^{k,R,\epsilon}\|_{\textbf{H}^s}^6+\|\textbf{u}^{\ell,R,\epsilon}\|_{\textbf{H}^s}^6\right) .
\end{split}
\end{equation}
\end{lemma}

\begin{proof}  Recalling the notations $
\mathbf{u}^{k, {\ell}} =\mathbf{u}^{k, R,\epsilon}
-\mathbf{u}^{l, R,\epsilon}$, $\mathbf{u}^{k}:=\mathbf{u}^{k, R,\epsilon}$ and $ \widetilde{\textbf{F}}^{k}(\cdot):=\widetilde{\textbf{F}}^{k,R,\epsilon}(\cdot)$.
For $(\textbf{u}^{k,\ell},\textbf{p}_1)_{\textbf{H}^{s-3}}$, we have
\begin{equation*}
\begin{split}
(\textbf{u}^{k,\ell},\textbf{p}_1)_{\textbf{H}^{s-3}} &\leq \|\textbf{u}^{k,\ell}\|_{\textbf{H}^{s-3}}  \| (\mathcal {J}_{\ell} +\mathcal {J}_k )(\mathcal {J}_{\ell} -\mathcal {J}_k )\textbf{A}^\alpha\textbf{u}^{\ell} \| _{\textbf{H}^{s-3}}\\
&\leq \|\textbf{u}^{k,\ell}\|_{\textbf{H}^{s-3}}^2+C\max\{\frac{1}{k^2},\frac{1}{ \ell^2 }\}\| \textbf{u}^{\ell}\| _{\textbf{H}^{s}}^2 .
\end{split}
\end{equation*}
For $(\textbf{u}^{k,\ell},\textbf{p}_2)_{\textbf{H}^{s-3}}$, first recall that $\textbf{A}^\alpha$ (see  \eqref{3.2}) is a selfadjoint operator with the unique square root $ \sqrt{\textbf{A}^\alpha}  $. It follows from the {Fourier-Plancherel Formula (cf. \cite[Theorem 1.25]{bahouri2011fourier}}) that
$$
(\textbf{A}^\alpha \textbf{u},\textbf{u})_{\textbf{H}^s}=\|\sqrt{\textbf{A}^\alpha} \textbf{u}\|_{\textbf{H}^s}\geq 0,~~~ \textrm{for all}  ~s>0,
$$
which implies that
\begin{equation*}
\begin{split}
(\textbf{u}^{k,\ell},\textbf{p}_2)_{\textbf{H}^{s-3}}=-(\textbf{u}^{k,\ell},\mathcal {J}_k^2 \textbf{A}^\alpha  \textbf{u}^{k,\ell})_{\textbf{H}^{s-3}}=-\|\sqrt{\textbf{A}^\alpha} \mathcal {J}_k  \textbf{u}^{k,\ell}\|_{\textbf{L}^2}^2 \leq 0.
\end{split}
\end{equation*}
For $(\textbf{u}^{k,\ell},\textbf{p}_3)_{\textbf{H}^{s-3}}$, by virtue of the Mean Value Theorem and the estimate \eqref{A.1}, we have
\begin{equation*}
\begin{split}
(\textbf{u}^{k,\ell},\textbf{p}_3)_{\textbf{H}^{s-3}}&\leq |\theta_R'(\xi^{k,\ell})| \|\textbf{u}^{k,\ell}\|_{W^{1,\infty}}  \| \textbf{u}^{\ell})\|_{\textbf{H}^{s-3}}\|\textbf{u}^{k,\ell}\|_{\textbf{H}^{s-3}}
\\
&\leq C   \|  \textbf{B}( \mathcal {J}_\ell \textbf{u}^{\ell})\|_{\textbf{H}^{s-3}}
\|\textbf{u}^{k,\ell}\|_{\textbf{H}^{s-3}}^2\\
&\leq C   \|  \textbf{u}^{\ell} \|_{\textbf{H}^{s }}^2
\|\textbf{u}^{k,\ell}\|_{\textbf{H}^{s-3}}^2,
\end{split}
\end{equation*}
where $\xi^{k,\ell}$ take values between $\|\textbf{u}^{\ell}\|_{W^{1,\infty}}$ and $ \|\textbf{u}^{k}\|_{W^{1,\infty}}$.

For $(\textbf{u}^{k,\ell},\textbf{p}_4)_{\textbf{H}^{s-3}}$, first note that
\begin{equation*}
\begin{split}
\textbf{B}( \mathcal {J}_\ell \textbf{u}^k)-\textbf{B}(\mathcal {J}_k\textbf{u}^k)=\begin{pmatrix}
 ( \mathcal {J}_\ell u^k\cdot \nabla)  \mathcal {J}_\ell n^k-(\mathcal {J}_ku^k\cdot \nabla) \mathcal {J}_kn^k  \\
  ( \mathcal {J}_\ell u^k\cdot \nabla)  \mathcal {J}_\ell c^k-(\mathcal {J}_ku^k\cdot \nabla) \mathcal {J}_kc^k \\
\textbf{P}( \mathcal {J}_\ell u^k\cdot \nabla)  \mathcal {J}_\ell u^k-\textbf{P}(\mathcal {J}_ku^k\cdot \nabla) \mathcal {J}_ku^k
\end{pmatrix}:=\begin{pmatrix}
    p_{3,1} \\
   p_{3,2}\\
 p_{3,3}
\end{pmatrix},
\end{split}
\end{equation*}
where
\begin{equation*}
\begin{split}
&p_{3,1}=  ((\mathcal {J}_{\ell} -\mathcal {J}_k)u^k\cdot \nabla) \mathcal {J}_{\ell}n^k + (\mathcal {J}_ku^k \cdot \nabla)(\mathcal {J}_{\ell}  -\mathcal {J}_k )n^k,\\
&p_{3,2}=  ((\mathcal {J}_{\ell} -\mathcal {J}_k)u^k\cdot \nabla) \mathcal {J}_{\ell}c^k + (\mathcal {J}_ku^k \cdot \nabla)(\mathcal {J}_{\ell}  -\mathcal {J}_k )c^k,\\
&p_{3,3}=  ((\mathcal {J}_{\ell} -\mathcal {J}_k)u^k\cdot \nabla) \mathcal {J}_{\ell}u^k + (\mathcal {J}_ku^k \cdot \nabla)(\mathcal {J}_{\ell}  -\mathcal {J}_k )u^k.
\end{split}
\end{equation*}
Then, it follows that
\begin{equation*}
\begin{split}
 (\textbf{u}^{k,\ell},\textbf{p}_4)_{\textbf{H}^{s-3}}
 &=\theta_R(\|\textbf{u}^{k}\|_{W^{1,\infty}})
\Big ((n^{k,\ell}, \mathcal {J}_k p_{3,1})_{H^{s-3}}+(c^{k,\ell}, \mathcal {J}_k p_{3,2})_{H^{s-3}}+(u^{k,\ell}, \mathcal {J}_k p_{3,3}) _{H^{s-3}}\Big).
\end{split}
\end{equation*}
By the Sobolev embedding $H^{s-3}(\mathbb{R}^2)\subset W^{1,\infty}(\mathbb{R}^2)$ and the conditions $\textrm{div} u^k=\textrm{div} u^\ell=0$, we infer that
\begin{equation*}
\begin{split}
(n^{k,\ell}, \mathcal {J}_k p_{3,1})_{H^{s-3}} &\leq C\|n^{k,\ell}\|_{H^{s-3}}\left\| (( \mathcal {J}_\ell  -\mathcal {J}_k)u^k\cdot \nabla)  \mathcal {J}_\ell n^k + (\mathcal {J}_ku^k \cdot \nabla)( \mathcal {J}_\ell   -\mathcal {J}_k )n^k\right\|_{H^{s-1}}\\
&\leq  {C\left(\max\{\frac{1}{k},\frac{1}{\ell}\}+\frac{1}{k}\right)}\|n^{k,\ell}\|_{H^{s-3}}\| u^k\|_{H^{s }}\| n^k \|_{H^{s }}\\
&\leq \|n^{k,\ell}\|_{H^{s-3}}^2+ C \max\{\frac{1}{k^2},\frac{1}{ \ell^2 }\} \| u^k\|_{H^{s }}^2\| n^k \|_{H^{s }} ^2.
\end{split}
\end{equation*}
Similarly, one can deduce that
\begin{equation*}
\begin{split}
  (c^{k,\ell}, \mathcal {J}_k p_{3,2})_{H^{s-3}} &\leq  \|c^{k,\ell}\|_{H^{s-3}}^2+ C \max\{\frac{1}{k^2},\frac{1}{ \ell^2 }\} \| u^k\|_{H^{s }}^2\| c^k \|_{H^{s }} ^2,\\
 (u^{k,\ell}, \mathcal {J}_k p_{3,3})_{H^{s-3}} &\leq  \|u^{k,\ell}\|_{H^{s-3}}^2+ C \max\{\frac{1}{k^2},\frac{1}{ \ell^2 }\} \| u^k\|_{H^{s }}^4.
\end{split}
\end{equation*}
Thereby, we get
\begin{equation*}
\begin{split}
 (\textbf{u}^{k,\ell},\textbf{p}_4)_{\textbf{H}^{s-3}}\leq \|\textbf{u}^{k,\ell}\|_{\textbf{H}^{s-3}}^2+ C\max\{\frac{1}{k^2},\frac{1}{ \ell^2 }\} \| \textbf{u}^k\|_{\textbf{H}^{s }}^4.
\end{split}
\end{equation*}
For $(\textbf{u}^{k,\ell},\textbf{p}_5)_{\textbf{H}^{s-3}}$, we use \eqref{A.2} to obtain
\begin{equation*}
\begin{split}
 (\textbf{u}^{k,\ell},\textbf{p}_5)_{\textbf{H}^{s-3}} &= \theta_R(\|\textbf{u}^{k}\|_{W^{1,\infty}})\left(\textbf{u}^{k,\ell},
   \mathcal {J}_\ell (\textbf{B}( \mathcal {J}_\ell \textbf{u}^{\ell})-\textbf{B}( \mathcal {J}_\ell \textbf{u}^{k}))  \right)_{\textbf{H}^{s-3}}\\
 &\leq \left|\left(\mathcal {J}_{\ell}\textbf{u}^{\ell}-\mathcal {J}_{\ell}\textbf{u}^{k},  \textbf{B}(\mathcal {J}_\ell \textbf{u}^{\ell})-\textbf{B}( \mathcal {J}_\ell \textbf{u}^{k} ) \right)_{\textbf{H}^{s-3}} \right|  \\
 &\leq C\left(\| \textbf{u}^k\|_{\textbf{H}^{s }} + \| \textbf{u}^\ell\|_{\textbf{H}^{s }}\right ) \|\textbf{u}^{k,\ell}\|_{\textbf{H}^{s-3}}^2.
\end{split}
\end{equation*}
For $(\textbf{u}^{k,\ell},\textbf{p}_6)_{\textbf{H}^{s-3}}$, we have
\begin{equation*}
\begin{split}
 (\textbf{u}^{k,\ell},\textbf{p}_6)_{\textbf{H}^{s-3}} &\leq \| \textbf{u}^{k,\ell}\|_{\textbf{H}^{s-3}}
\|( \mathcal {J}_\ell -\mathcal {J}_k)\textbf{B}(\mathcal {J}_k\textbf{u}^{k})\|_{\textbf{H}^{s-3}} \leq C \max\{\frac{1}{k^2},\frac{1}{ \ell^2 }\}
\| \textbf{u}^{k} \|_{\textbf{H}^{s }}^2\|\textbf{u}^{k,\ell}\|_{\textbf{H}^{s-3}}.
\end{split}
\end{equation*}
For $(\textbf{u}^{k,\ell},\textbf{p}_7)_{\textbf{H}^{s-3}}$, we have
\begin{equation*}
\begin{split}
(\textbf{u}^{k,\ell},\textbf{p}_7)_{\textbf{H}^{s-3}}\leq C(\epsilon,\phi)\max\{\frac{1}{k},\frac{1}{\ell}\}\|\textbf{u}^{k,\ell}\|_{\textbf{H}^{s-3}}\| \textbf{u}^{\ell} \|_{\textbf{H}^{s}}^2.
\end{split}
\end{equation*}
To estimate $(\textbf{u}^{k,\ell},\textbf{p}_8)_{\textbf{H}^{s-3}}$, we observe that
\begin{equation*}
\begin{split}
\textbf{F} ^\epsilon (\mathcal {J}_k\textbf{u}^{\ell})- \textbf{F} ^\epsilon ( \mathcal {J}_\ell (\mathcal {J}_{\ell}\textbf{u}^{\ell})=\begin{pmatrix}
\textrm{div}\left([(\mathcal {J}_{\ell} -\mathcal {J}_k)n^{\ell}](\nabla c^{\ell}*\rho^{\epsilon})+\mathcal {J}_kn^{\ell}(\nabla ( \mathcal {J}_\ell -\mathcal {J}_k)c^{\ell}*\rho^{\epsilon})\right)\\+  (\mathcal {J}_k- \mathcal {J}_\ell \mathcal {J}_\ell -\mathcal {J}_k)n^{\ell}( \mathcal {J}_\ell +\mathcal {J}_k)n^{\ell}\\
 (( \mathcal {J}_\ell -\mathcal {J}_k)c^{\ell})( \mathcal {J}_\ell n^{\ell}*\rho^{\epsilon})
 +\mathcal {J}_k c^{\ell}(( \mathcal {J}_\ell -\mathcal {J}_k)n^{\ell}*\rho^{\epsilon}) \\
\textbf{P}((\mathcal {J}_k- \mathcal {J}_\ell )n^{\ell}\nabla \phi)*\rho^{\epsilon}
\end{pmatrix}:=\begin{pmatrix}
    p_{8,1} \\
   p_{8,2}\\
 p_{8,3}
\end{pmatrix},
\end{split}
\end{equation*}
which implies
\begin{equation*}
\begin{split}
(\textbf{u}^{k,\ell},\textbf{p}_8)_{\textbf{H}^{s-3}}
=\theta_R(\|\textbf{u}^{\ell}\|_{W^{1,\infty}})\left( (n^{k,\ell},\mathcal {J}_k p_{8,1} )_{H^{s-3}}+(c^{k,\ell},\mathcal {J}_k p_{8,2} )_{H^{s-3}}+(u^{k,\ell},\mathcal {J}_k p_{8,3} )_{H^{s-3}}\right).
\end{split}
\end{equation*}
For the three terms on the R.H.S., we have
\begin{equation*}
\begin{split}
&(n^{k,\ell},\mathcal {J}_k p_{8,1} )_{H^{s-3}}\\
& \quad \leq  C\| n^{k,\ell}\|_{H^{s-3}} \Big( \|(( \mathcal {J}_\ell  -\mathcal {J}_k)n^{\ell})(\nabla \mathcal {J}_{\ell}c^{\ell}*\rho^{\epsilon})\|_{H^{s-2}}+\|\mathcal {J}_kn^{\ell}(\nabla (\mathcal {J}_{\ell}-\mathcal {J}_k)c^{\ell}*\rho^{\epsilon}) \|_{H^{s-2}}\\
&\quad \quad+ \| (\mathcal {J}_k- \mathcal {J}_\ell )n^{\ell}\|_{H^{s-3}}
+\|( \mathcal {J}_\ell -\mathcal {J}_k)n^{\ell}( \mathcal {J}_\ell +\mathcal {J}_k)n^{\ell}\|_{H^{s-3}}\Big)\\
&\quad \leq  C(\epsilon)\max\{\frac{1}{k },\frac{1}{\ell }\}\| n^{k,\ell}\|_{H^{s-3}} \Big( \| n^{\ell} \|_{H^{s}}\| c^{\ell} \|_{H^{s}} + \| n^{\ell}\|_{H^{s}} \Big) .
\end{split}
\end{equation*}
In a similar manner,
\begin{equation*}
\begin{split}
&(c^{k,\ell},\mathcal {J}_k p_{8,2} )_{H^{s-3}} \leq  \| c^{k,\ell}\|_{H^{s-3}}^2+C\max\{\frac{1}{k^2},\frac{1}{ \ell^2 }\} \| n^{\ell} \|_{H^{s}}^2\| c^{\ell} \|_{H^{s}}^2,\\
&(u^{k,\ell},\mathcal {J}_k p_{8,3} )_{H^{s-3}} \leq  \| u^{k,\ell}\|_{H^{s-3}}^2+C\max\{\frac{1}{k^2},\frac{1}{ \ell^2 }\} \| n^{\ell} \|_{H^{s}}^2 .
\end{split}
\end{equation*}
Thereby, we get from the Cauchy inequality that
\begin{equation*}
\begin{split}
(\textbf{u}^{k,\ell},\textbf{p}_8)_{\textbf{H}^{s-3}}
\leq \| \textbf{u}^{k,\ell}\|_{\textbf{H}^{s-3}}^2+ C(\epsilon )\max\{\frac{1}{k^2},\frac{1}{ \ell^2 }\} \left(\| \textbf{u}^{k}\|_{\textbf{H}^{s}}^2+\| \textbf{u}^{\ell}\|_{\textbf{H}^{s}}^2+\| \textbf{u}^{\ell}\|_{\textbf{H}^{s}}^4\right).
\end{split}
\end{equation*}
For $(\textbf{u}^{k,\ell},\textbf{p}_9)_{\textbf{H}^{s-3}}$, we get from \eqref{A.4} that
\begin{equation*}
\begin{split}
 (\textbf{u}^{k,\ell},\textbf{p}_9)_{\textbf{H}^{s-3}}&\leq\left|\left(\mathcal {J}_k\textbf{u}^{k }-\mathcal {J}_k\textbf{u}^{\ell},  \textbf{F} ^\epsilon (\mathcal {J}_k\textbf{u}^{k})-\textbf{F} ^\epsilon (\mathcal {J}_k\textbf{u}^{\ell})\right)_{\textbf{H}^{s-3}}\right|\\
 &\leq  C(\epsilon)\left(\|\textbf{u}^{k }\|_{\textbf{H}^s}+\|\textbf{u}^{\ell}\|_{\textbf{H}^s}\right) \|\textbf{u}^{k,\ell}\|_{\textbf{H}^s}^2.
\end{split}
\end{equation*}
By using the Mean Value Theorem, the term involving $\textbf{p}_{10}$ can be estimated as
\begin{equation*}
\begin{split}
 (\textbf{u}^{k,\ell},\textbf{p}_{10})_{\textbf{H}^{s-3}}&\leq \left|\theta_R(\|\textbf{u}^{k}\|_{\textbf{W}^{1,\infty}})-\theta_R(\|\textbf{u}^{\ell}\|_{\textbf{W}^{1,\infty}})\right| \left|(\textbf{u}^{k,\ell}, \mathcal {J}_k\textbf{F} ^\epsilon (\mathcal {J}_k\textbf{u}^{k}))_{\textbf{H}^{s-3}}\right|\\
 &\leq C \|\textbf{u}^{k}-\textbf{u}^{\ell}\|_{\textbf{W}^{1,\infty}} \|\textbf{u}^{k,\ell}\|_{\textbf{H}^{s-3}}\| \textbf{F} ^\epsilon (\mathcal {J}_k\textbf{u}^{k})\|_{\textbf{H}^{s-3}}\\
 &\leq C(\epsilon)\left(\|\textbf{u}^{k }\|_{\textbf{H}^s}^3+\|\textbf{u}^{\ell }\|_{\textbf{H}^s}^3\right)\|\textbf{u}^{k,\ell}\|_{\textbf{H}^{s-3}}.
\end{split}
\end{equation*}
Collecting all of the above estimates for terms $(\textbf{u}^{k,\ell},\textbf{p}_{i})_{\textbf{H}^{s-3}}$, $i=1,...,10$, we obtain \eqref{3.36}.
\end{proof}

\begin{lemma}\label{lem3..4}
For fixed $R$ and $\epsilon$, assume that $\{\mathbf{u}^{k,R,\epsilon}\}_{k \in \mathbb{N}^+}$ are approximate solutions in Lemma \ref{lem1}, then  {there exists a progressively measurable element  $\textbf{u}^{R,\epsilon} \in L^2(\Omega;L^\infty (0,T;\textbf{H}^{s}(\mathbb{R}^2)))$} and a subsequence of $\{\textbf{u}^{k,R,\epsilon}\}_{k\in \mathbb{N}^+}$, still denoted by itself, such that
\begin{equation*}
\begin{split}
 \textbf{u}^{k,R,\epsilon}\rightarrow \textbf{u}^{R,\epsilon}~~ \textrm{in} ~~C([0,T];\textbf{H}^{s-3}(\mathbb{R}^2))~ ~\textrm{as} ~~k\rightarrow\infty,~~\mathbb{P}\textrm{-a.s.}
\end{split}
\end{equation*}
\end{lemma}

\begin{proof} {For simplicity, we use the notations $\mathbf{u}=\mathbf{u}^{ R,\epsilon}$, $ \mathbf{u}^{k, {\ell}} = \mathbf{u}^{k, R,\epsilon}
-\mathbf{u}^{l, R,\epsilon}$ and $\mathbf{u}^{k}:=\mathbf{u}^{k, R,\epsilon}$.} For each $N>1$ and $T>0$, define
$$
\textbf{t}_{N,k,\ell}(T) := \textbf{t}_{N,k}(T) \wedge \textbf{t}_{N,\ell}(T),
$$
where
$$
\textbf{t}_{N,k}(T) = \inf  \left\{t\geq 0;~\|\textbf{u}^k(t)\|_{H^s}\geq N \right\}\wedge T .
$$
By \eqref{3.35} and the BDG inequality, we infer that
\begin{equation}\label{3.37}
\begin{split}
&\mathbb{E}\sup_{r \in [0,\textbf{t}_{N,k}(T)]}\|\textbf{u}^{k,\ell}(r)\|_{\textbf{H}^{s-3}}^2\\
&\quad \leq   C(\epsilon)  \mathbb{E}\int_0^{\textbf{t}_{N,k}(T)}\bigg[\left(1+\| \textbf{u}^{k}\| _{\textbf{H}^{s}}^2+\| \textbf{u}^{\ell}\| _{\textbf{H}^{s}}^2\right)\|\textbf{u}^{k,\ell}\|_{\textbf{H}^{s-3}}^2+  \max\{\frac{1}{k^2},\frac{1}{ \ell^2 }\} \\
&\quad\quad \times\left(\| \textbf{u}^{k}\| _{\textbf{H}^{s}}^2+\| \textbf{u}^{\ell}\| _{\textbf{H}^{s}}^2 +\|\textbf{u}^{k }\|_{\textbf{H}^s}^6+\|\textbf{u}^{\ell}\|_{\textbf{H}^s}^6\right) \bigg] \textrm{d}r\\
&\quad\quad +C \mathbb{E} \left( \sup_{r \in [0,\textbf{t}_{N,k}(T)]}\|\textbf{u}^{k,\ell}(r)\|_{\textbf{H}^{s-3}}^2\int_0^{\textbf{t}_{N,k}(T)} \sum_{j\geq 1}\|\textbf{p}_{11}^j(r)
\|_{\textbf{H}^{s-3}} ^2\textrm{d}r\right)^{\frac{1}{2}}\\
&\quad\quad+  \mathbb{E}\int_0^{\textbf{t}_{N,k}(T)}\|\textbf{p}_{11}(r)\|_{L_2(\textbf{U}; \textbf{H}^{s-3})}^2 \textrm{d}r \\
&\quad \leq  \frac{1}{2}\mathbb{E}\sup_{r \in [0,\textbf{t}_{N,k}(T)]}\|\textbf{u}^{k,\ell}(r)\|_{\textbf{H}^{s-3}}^2+ C(\epsilon) \max\{\frac{1}{k^2},\frac{1}{ \ell^2 }\} (N^2+N^6)T \\
&\quad\quad+C(\epsilon)\left( T (1+2N^2)+ 1+N^2 \right) \mathbb{E}\int_0^{\textbf{t}_{N,k}(T)} \|\textbf{u}^{k,\ell}(r)\|_{\textbf{H}^{s-3}}^2\textrm{d}r ,
\end{split}
\end{equation}
where the second inequality used
\begin{equation*}
\begin{split}
 \|\textbf{p}_{11}\|_{L_2(\textbf{U}; \textbf{H}^{s-3})}^2\leq  & C  \|\textbf{u}^{k,\ell}  \|_{\textbf{H}^{s-3}}^2 +C\|\textbf{u}^{k,\ell}  \|_{\textbf{H}^{s-3}}^2 \left(1+\|u^{k}\|_{H^{s}}^2 \right) \\
\leq& C \left(1+N^2  \right)\|\textbf{u}^{k,\ell}  \|_{\textbf{H}^{s-3}}^2,
\end{split}
\end{equation*}
for any $t\in [0,\textbf{t}_{N,k}(T)]$. By applying the Gronwall Lemma to \eqref{3.37}, we get
\begin{equation*}
\begin{split}
&\mathbb{E}\sup_{r \in [0,\textbf{t}_{N,k}(T)]}\|\textbf{u}^{k,\ell}(r)\|_{\textbf{H}^{s-3}}^2\\
 &\quad \leq C(\epsilon)\max\{\frac{1}{k^2},\frac{1}{ \ell^2 }\} T(N^2+N^6)\exp\left\{ C(\epsilon)\left (1+ T (1+2N^2)T+ 1+N^2  \right)\right\},
\end{split}
\end{equation*}
which implies that
\begin{equation}\label{3.38}
\begin{split}
 \lim _{k\rightarrow\infty}\sup_{\ell\geq k}\mathbb{E}\sup_{r \in [0,\textbf{t}_{N,k}(T)]}\|\textbf{u}^{k,\ell}(r)\|_{\textbf{H}^{s-3}}^2=0,\quad \forall N\geq 1,~\epsilon> 0.
\end{split}
\end{equation}
By using the Chebyshev inequality,  we have
\begin{equation*}
\begin{split}
&  \mathbb{P}\left\{\sup_{t \in [0,T]} \|\textbf{u}^{k }-\textbf{u}^{\ell}\|_{\textbf{H}^{s-3}} >\eta\right\}\\
&\quad = \mathbb{P}\left\{([\textbf{t}_{N,k}(T)<T]\cup [\textbf{t}_{N,k}(T)=T]) \cap \{\sup_{t \in [0,T]} \|\textbf{u}^{k }-\textbf{u}^{\ell}\|_{\textbf{H}^{s-3}} >\eta\}\right\}\\
&\quad\leq \mathbb{P} \{[\textbf{t}_{N,k}(T)<T]\} + \mathbb{P} \{[\textbf{t}_{N,k}(T)=T]\} +\mathbb{P}  \left\{\sup_{t \in [0,\textbf{t}_{N,k}(T)]} \|\textbf{u}^{k }-\textbf{u}^{\ell}\|_{\textbf{H}^{s-3}} >\eta\right\} \\
&\quad\leq   \mathbb{P} \left \{\sup_{t \in [0,\textbf{t}_{N,k}(T)]} \|\textbf{u}^{k }-\textbf{u}^{\ell}\|_{\textbf{H}^{s-3}} >\eta\right\}+ \frac{C(p,R,\textbf{u}_0,\chi,\kappa,\epsilon,T)}{N^2}.
\end{split}
\end{equation*}
By \eqref{3.38}, we get from the last estimate that
\begin{equation*}
\begin{split}
 \lim _{k\rightarrow\infty}\sup_{\ell\geq k} \mathbb{P}\left\{\sup_{t \in [0,T]} \|\textbf{u}^{k }-\textbf{u}^{\ell}\|_{\textbf{H}^{s-3}} >\eta\right\} \leq \frac{C(p,R,\textbf{u}_0,\chi,\kappa,\epsilon,T)}{N^2}.
\end{split}
\end{equation*}
Taking $N\rightarrow\infty$ in the last inequality, we deduce that
$$
\textbf{u}^{k }\rightarrow \textbf{u}~~~  \textrm{in}~ C([0,T];\textbf{H}^{s-3}(\mathbb{R}^2))~\textrm{in probability},~~ \textrm{as}~ k\rightarrow\infty.
$$
By the Riesz Theorem, it follows that there exists a subsequence of $\{\textbf{u}^{k }\}_{k\in \mathbb{N}^+}$, still denoted by itself, such that $\textbf{u}^{k }\rightarrow \textbf{u}$ in $C([0,T];\textbf{H}^{s-3}(\mathbb{R}^2))$ as $k\rightarrow\infty$, $\mathbb{P}$-a.s.

 {Now we prove that the limit $\textbf{u}\in  L^2(\Omega;L^\infty(0,T;\textbf{H}^{s}(\mathbb{R}^2) )$. Indeed, in view of the uniform bound in Lemma \ref{lem3}, we get by taking $p=2$ that
$
\sup_{k \in\mathbb{N}^+} \mathbb{E} \sup_{t\in [0,T]}\|\textbf{u}^{k}(t)\|^2_{\textbf{H}^s} \leq C,
$
which indicates that there exists a subsequence of $\{\textbf{u}^k\}_{k\geq1}$, still denoted by itself, such that $\textbf{u}^k\rightarrow \varrho$ weak star in $L^2(\Omega;L^\infty(0,T;\textbf{H}^{s}(\mathbb{R}^2) )$.
Since we have proved in Step 2 that $\textbf{u}^{k }\rightarrow \textbf{u}$ in $C([0,T];\textbf{H}^{s-3}(\mathbb{R}^2))$ as $k\rightarrow\infty$, $\mathbb{P}$-a.s., we infer that $\textbf{u}=\varrho \in L^2(\Omega;L^\infty(0,T;\textbf{H}^{s}(\mathbb{R}^2) )$.}
Finally,  since for each $k\geq1$, $ \textbf{u}^k $ are progressively measurable processes, so is $\textbf{u}$.
\end{proof}

\begin{lemma}\label{lem3.4}
Let $s>5$, $R>0$ and $\epsilon \in (0,1)$. For any  $T>0$, under the assumptions (H1)-(H3),  there exists a unique pathwise solution $ \textbf{u}^{R,\epsilon} \in L^2(\Omega;C([0,T];\textbf{H}^{s}(\mathbb{R}^2)))$ to \eqref{3.1} with cut-off functions, that is,
\begin{equation}\label{3.39}
\begin{split}
& \textbf{u}^{R,\epsilon}(t)-\textbf{u}^{ \epsilon}_0+\int_0^t \textbf{A}^\alpha\textbf{u}^{R,\epsilon}\textrm{d}r+\int_0^t\theta_R(\|\textbf{u}^{R,\epsilon}\|_{\textbf{W}^{1,\infty}}) \textbf{B}(\textbf{u}^{R,\epsilon})\textrm{d}r\\
&\quad =\int_0^t\theta_R(\|\textbf{u}^{R,\epsilon}\|_{\textbf{W}^{1,\infty}}) \textbf{F}^\epsilon (\textbf{u}^{R,\epsilon})\textrm{d}r+\int_0^t \textbf{G}  (r,\textbf{u}^{R,\epsilon})\textrm{d} \mathcal {W}_r,
 \end{split}
\end{equation}
for any $t\in [0,T]$, $\mathbb{P}$-a.s.
\end{lemma}

\begin{proof}  We use the notations $\mathbf{u}=\mathbf{u}^{ R,\epsilon}$, $
\mathbf{u}^{k, {\ell}} =\mathbf{u}^{k, R,\epsilon}
-\mathbf{u}^{l, R,\epsilon}$, $\mathbf{u}^{k} =\mathbf{u}^{k, R,\epsilon}$ and $ \widetilde{\textbf{F}}^{k}(\cdot) =\widetilde{\textbf{F}}^{k,R,\epsilon}(\cdot)$ for simplicity. By applying the almost surely convergence in Lemma \ref{lem3..4}, one can take the limit $k\rightarrow\infty$ to find that the limit $\textbf{u}$ satisfies \eqref{3.39} $\mathbb{P}$-a.s. 
It remains to prove that $\textbf{u} \in L^2(\Omega;C([0,T];\textbf{H}^{s}(\mathbb{R}^2)))$. Indeed, recalling that (cf. Lemma \ref{lem3} and Lemma \ref{lem3..4})
$
\textbf{u} \in L^\infty(0,T;\textbf{H}^{s}(\mathbb{R}^2))\cap  C([0,T];\textbf{H}^{s-3}(\mathbb{R}^2)) .
$
It then follows from \cite[Lemma 1.4]{temam2001navier} that $\textbf{u} \in  C_{\textrm{weak}}([0,T];\textbf{H}^{s}(\mathbb{R}^2))$, namely,  for any $r \in [0,T]$ and smooth function $\varphi\in C_0^\infty(\mathbb{R}^2)$, we have
\begin{equation}\label{2.40}
\begin{split}
\lim_{t\rightarrow r} (\textbf{u}(t),\varphi)_{\textbf{H}^{s},(\textbf{H}^{s})'}= (\textbf{u}(r),\varphi)_{\textbf{H}^{s},(\textbf{H}^{s})'}.
 \end{split}
\end{equation}
To prove the continuity of the map $t\mapsto \|\textbf{u}(t)\|_{\textbf{H}^{s}}$,  we  consider another spacial mollifier $\varrho_\eta $ with parameter $\eta> 0$, and apply the convolution operator $ \varrho_\eta* $ to \eqref{3.39} to find
\begin{equation}\label{3.41}
\begin{split}
& \mathrm{d}  \varrho_\eta*\textbf{u} + \textbf{A}^\alpha\varrho_\eta*\textbf{u} \textrm{d}t
+\theta_R(\|\textbf{u} \|_{\textbf{W}^{1,\infty}}) \varrho_\eta*\textbf{B}(\textbf{u} )\textrm{d}t\\
&\quad =\theta_R(\|\textbf{u} \|_{\textbf{W}^{1,\infty}})\varrho_\eta* \textbf{F}^\epsilon (\textbf{u} )\textrm{d}t+ \varrho_\eta*\textbf{G}  (t,\textbf{u} )\textrm{d} \mathcal {W}_t,
 \end{split}
\end{equation}
which can be viewed as a system of SDEs in Hilbert spaces. Utilizing the It\^{o} formula  in Hilbert space {(cf. \cite[Theorem 4.32]{da2014stochastic}) to $ \|\varrho_\eta*\textbf{u}^\epsilon\|_{\textbf{H}^{s} }^2$,} we get from \eqref{3.41} that
\begin{equation*}
\begin{split}
\left|\|\varrho_\eta*\textbf{u}(t_2) \|_{\textbf{H}^s}^{2}-\|\varrho_\eta*\textbf{u} (t_1) \|_{\textbf{H}^s}^{2}\right| &\leq   2\int_{t_1}^{t_2}\theta_R(\|\textbf{u} \|_{\textbf{W}^{1,\infty}})|(\varrho_\eta*\textbf{u} , \varrho_\eta*\textbf{B}(\textbf{u} ))_{\textbf{H}^s}|\textrm{d}t + \int_{t_1}^{t_2} \|\varrho_\eta*\textbf{G}  (t,\textbf{u} )\|_{L_2(\textbf{U};\textbf{H}^s)}^2 \textrm{d}t\\
&+2\int_{t_1}^{t_2}\theta_R(\|\textbf{u} \|_{\textbf{W}^{1,\infty}})(\varrho_\eta*\textbf{u} ,\varrho_\eta* \textbf{F}^\epsilon (\textbf{u} ))_{\textbf{H}^s}\textrm{d}t+2\int_{t_1}^{t_2} (\varrho_\eta*\textbf{u}^\epsilon, \varrho_\eta*\textbf{G}  (t,\textbf{u} )\textrm{d} \mathcal {W}_t)_{\textbf{H}^s}.
 \end{split}
\end{equation*}
For each $N>1$, we set
$$
r^N := \inf\left\{t>0;~\|\textbf{u}(t) \|_{\textbf{H}^s}> N\right\}.
$$
Then it follows from \eqref{3.9} that $r^N\rightarrow \infty$ as $N\rightarrow\infty$. By raising the $3^{th}$ power to the last inequality and taking the expectation, similar to Lemma \ref{lem3}, one can use the assumption on $f$ and the estimates \eqref{A.1} and \eqref{A.3} to derive that
\begin{equation*}
\begin{split}
&\mathbb{E}\left|\|\varrho_\eta*\textbf{u}(t_2\wedge r^N ) \|_{\textbf{H}^s}^{2}-\|\varrho_\eta*\textbf{u} (t_1\wedge r^N ) \|_{\textbf{H}^s}^{2}\right|^3\leq  C(N,R,T) |t_2-t_2|^\frac{3}{2}.
 \end{split}
\end{equation*}
Taking the limit as $\eta\rightarrow 0$,  we get from  {the Fatou Lemma} that
\begin{equation*}
\begin{split}
&\mathbb{E}\left|\| \textbf{u}(t_2\wedge r^N ) \|_{\textbf{H}^s}^{2}-\| \textbf{u} (t_1\wedge r^N ) \|_{\textbf{H}^s}^{2}\right|^3\leq  C(N,R,T) |t_2-t_2|^\frac{3}{2}.
 \end{split}
\end{equation*}
Therefore, the Kolmogorov-\v{C}entsov Continuity Theorem informs us that the process $t\mapsto \| \textbf{u}(t\wedge r^N ) \|_{\textbf{H}^s}^{2}$ has a  continuous version. By taking the limit as $N\rightarrow\infty$ and combining \eqref{2.40}, it holds  that $\textbf{u} \in C([0,T];\textbf{H}^s(\mathbb{R}^2))$, $\mathbb{P}$-a.s. This finishes the proof of Lemma \ref{lem3.4}.
\end{proof}

Before removing the cut-off functions  and constructing local maximal pathwise solutions to the system \eqref{3.1}, let us first establish the following uniqueness result.

\begin{lemma} \label{lem3.7}
Let $(\textbf{u}^\epsilon,\tau^\epsilon)$ and $(\textbf{v}^\epsilon,\bar{\tau}^\epsilon)$ be two local pathwise solutions in $H^s(\mathbb{R}^2)$ ($s>5$) to the system \eqref{3.1} (or \eqref{3.2}) with the same initial data $(n^\epsilon_0,c^\epsilon_0,u^\epsilon_0)$, where $\textbf{u}^\epsilon=(n^\epsilon,c^\epsilon, u ^\epsilon)$ and $\bar{\textbf{u}}^\epsilon=(\bar{n}^\epsilon,\bar{c}^\epsilon,\bar{ u }^\epsilon)$. Then
\begin{equation*}
\begin{split}
\mathbb{P}\{\textbf{u}^\epsilon(t)
=\bar{\textbf{u}}^\epsilon(t) ,~ \forall t\in [0,\tau^\epsilon \wedge \bar{\tau}^\epsilon]\}=1.
\end{split}
\end{equation*}
\end{lemma}

\begin{proof}
Since $\textbf{u}^\epsilon$ and $\bar{\textbf{u}}^\epsilon $ belong to $C([0,T];\textbf{H}^s(\mathbb{R}^2))$, $s>5$, $\mathbb{P}$-a.s., there exists a positive constant $M$ depending only on $\epsilon$ such that $\|\textbf{u}^\epsilon(0) \|_{\textbf{H}^3}=\|\bar{\textbf{u}}^\epsilon(0) \|_{\textbf{H}^3}=\|(n^\epsilon_0,c^\epsilon_0,u^\epsilon_0)\|_{\textbf{H}^3}\leq M$.  For all $J>2M$, one can define a sequence of stopping times
\begin{equation}\label{3.42}
\begin{split}
\textbf{t}_J^\epsilon:=   \inf \left\{t>0;~ \|\textbf{u}^\epsilon(t) \|_{\textbf{H}^3}^2
+\|\bar{\textbf{u}}^\epsilon(t) \|_{\textbf{H}^3}^2>J\right\}.
\end{split}
\end{equation}
Then, $\mathbb{P}\{\textbf{t}_J^\epsilon>0\}=1$. Since $\mathbb{E}(\sup_{t\in [0,T]}\|\textbf{u}^\epsilon (t)\|_{\textbf{H}^3}^2+\sup_{t\in [0,T]}\|\bar{\textbf{u}}^\epsilon (t)\|_{\textbf{H}^3}^2)<\infty$, we have $\mathbb{P}\{\liminf_{J\rightarrow\infty} \textbf{t}_J^\epsilon\geq \tau^\epsilon \wedge \bar{\tau}^\epsilon\}=1$.  Set $\bar{\bar{\textbf{u}}}^\epsilon:= (\bar{\bar{n}}^\epsilon,\bar{\bar{c}}^\epsilon,\bar{\bar{u}}^\epsilon)$ with
$
\bar{\bar{n}}^\epsilon=n^\epsilon-\bar{n}^\epsilon$, $\bar{\bar{c}}^\epsilon=c^\epsilon-\bar{c}^\epsilon$ and $\bar{\bar{u}}^\epsilon= u ^\epsilon-\bar{ u }^\epsilon $.
It remains to show that $\bar{\bar{\textbf{u}}}^{\epsilon}\equiv 0$, for all $ t \in [0, \tau^\epsilon \wedge \bar{\tau}^\epsilon]$, $\mathbb{P}$-a.s. Indeed, by virtue of the $n^\epsilon$-equations in \eqref{3.1}, we get
\begin{equation}\label{3.43}
\begin{split}
\textrm{d} \bar{\bar{n}}^{\epsilon} =&\Delta \bar{\bar{n}}^{\epsilon}\textrm{d}r-\left(
 u ^{\epsilon}\cdot \nabla n^{\epsilon}-
\bar{ u }^{\epsilon}\cdot \nabla \bar{n}^{\epsilon}\right)\textrm{d}r-\left(\textrm{div}\left(n^{\epsilon}(\nabla c^{\epsilon}*\rho^{\epsilon})\right) - \textrm{div}\left(\bar{n}^{\epsilon}(\nabla \bar{c}^{\epsilon} *\rho^{\epsilon})\right)\right)\textrm{d}r\\
&+  \bar{\bar{n}}^{\epsilon}\textrm{d}r-   \left( (n^{\epsilon} )^2 -(\bar{n}^{\epsilon} )^2\right) \textrm{d}r \\
:=& (\mathcal {A}_1+\cdot\cdot\cdot +  \mathcal {A}_5)\textrm{d}r.
\end{split}
\end{equation}
Applying the chain rule to $ \|\bar{\bar{n}}^{\epsilon}(t)\|_{L^2}^2$ and integrating by parts, we get from \eqref{3.43} that
\begin{equation}\label{3.44}
\begin{split}
 \|\bar{\bar{n}}^{\epsilon}(t)\|_{L^2}^2+ 2 \int_0^t\|\nabla \bar{\bar{n}}^{\epsilon}(r)\|_{L^2}^2 \textrm{d}r=  2\sum_{i=2}^5 \int_0^t(\bar{\bar{n}}^{\epsilon},  \mathcal {A}_i)_{L^2}\textrm{d}r.
\end{split}
\end{equation}
For  $\mathcal {A}_2$, using $(\bar{\bar{n}}^{\epsilon},
 \bar{\bar{u}} ^{\epsilon}\cdot \nabla \bar{n}^{\epsilon} )_{L^2}  =-(\bar{n}^{\epsilon},\bar{\bar{u}} ^{\epsilon} \cdot\nabla\bar{\bar{n}}^{\epsilon}
 )_{L^2} $ obtained by $\textrm{div} u^{\epsilon}=0$, we can derive that
\begin{equation}\label{x1}
\begin{split}
2(\bar{\bar{n}}^{\epsilon},\mathcal {A}_2)_{L^2} &\leq 2|(\bar{\bar{n}}^{\epsilon},
 \bar{\bar{u}} ^{\epsilon}\cdot \nabla \bar{n}^{\epsilon} )_{L^2} |+2|(\bar{\bar{n}}^{\epsilon},
u^{\epsilon}\cdot \nabla \bar{\bar{n}}^{\epsilon} )_{L^2} | \\
&=2|(\bar{n}^{\epsilon},\bar{\bar{u}} ^{\epsilon} \cdot\nabla\bar{\bar{n}}^{\epsilon})_{L^2} |\\
 & \leq 2 \|\bar{n}^{\epsilon}\|_{L^\infty}\|\bar{\bar{u}}^{\epsilon}\|_{L^2}\|\nabla\bar{\bar{n}}^{\epsilon}\|_{L^2} \\
 &\leq C\|\bar{n}^{\epsilon}\|_{H^2}^2\|\bar{\bar{u}}^{\epsilon}\|_{L^2}^2+\frac{1}{4}\|\nabla\bar{\bar{n}}^{\epsilon}\|_{L^2}^2,
\end{split}
\end{equation}
where the last inequality used the young inequality. For $\mathcal {A}_3$, we have
\begin{equation}\label{x2}
\begin{split}
2(\bar{\bar{n}}^{\epsilon},\mathcal {A}_3)&= 2 (\nabla\bar{\bar{n}}^{\epsilon},\bar{\bar{n}}^{\epsilon}(\nabla c^{\epsilon}*\rho^{\epsilon}))+2  (\nabla\bar{\bar{n}}^{\epsilon},\bar{n}^{\epsilon} ( \nabla \bar{\bar{c}}^{\epsilon} *\rho^{\epsilon}))\\
&\leq \|\nabla\bar{\bar{n}}^{\epsilon}\|_{L^2}\|\bar{\bar{n}}^{\epsilon}\|_{L^2}\| \nabla c^{\epsilon}*\rho^{\epsilon}\|_{L^\infty}+ \| \nabla\bar{\bar{n}}^{\epsilon}\|_{L^2}\|\bar{n}^{\epsilon}\|_{L^2}\| \nabla \bar{\bar{c}}^{\epsilon} *\rho^{\epsilon}\|_{L^\infty}\\
&\leq C(\epsilon)\left(\|\nabla\bar{\bar{n}}^{\epsilon}\|_{L^2}\|\bar{\bar{n}}^{\epsilon}\|_{L^2}\|   c^{\epsilon} \|_{L^2}+ \| \nabla\bar{\bar{n}}^{\epsilon}\|_{L^2}\|\bar{n}^{\epsilon}\|_{L^2}\|  \bar{\bar{c}}^{\epsilon} \|_{L^2}\right)\\
&\leq C(\epsilon)\left(\|\bar{\bar{n}}^{\epsilon}\|_{L^2}^2\|  c^{\epsilon} \|_{L^2}^2+\|\bar{n}^{\epsilon}\|_{L^2}^2\| \bar{\bar{c}}^{\epsilon} \|_{L^2}^2\right)+\frac{1}{4}\|\nabla\bar{\bar{n}}^{\epsilon}\|_{L^2}^2,
\end{split}
\end{equation}
where the last inequality used the Young inequality and the second inequality used the following imortant facts of Friedrichs mollifier (see, e.g., \cite[identity (2.2)]{li2021stochastic}):
$$
\| \nabla \bar{\bar{c}}^{\epsilon} *\rho^{\epsilon}\|_{L^\infty} \leq C\| \bar{\bar{c}}^{\epsilon} *\rho^{\epsilon}\|_{W^{1,\infty}}\leq C\| \bar{\bar{c}}^{\epsilon} *\rho^{\epsilon}\|_{H^{3}}\leq \frac{C}{\epsilon^3}\| \bar{\bar{c}}^{\epsilon}  \|_{L^2},
$$
and
$$
\| \nabla c^{\epsilon}*\rho^{\epsilon}\|_{L^\infty}\leq C \| c^{\epsilon} *\rho^{\epsilon}\|_{H^{3}}\leq \frac{C}{\epsilon}\| c^{\epsilon}  \|_{H^2}.
$$
Apparently, we have
\begin{equation}\label{x3}
\begin{split}
(\bar{\bar{n}}^{\epsilon},\mathcal {A}_4)&= \| \bar{\bar{n}}^{\epsilon}\|_{L^2}^2.
\end{split}
\end{equation}
For $\mathcal {A}_5$, we have
\begin{equation}\label{x4}
\begin{split}
\|(\bar{\bar{n}}^{\epsilon}, \mathcal {A}_5)_{L^2}\|_{L^2}&\leq C  \| n^{\epsilon} + \bar{n}^{\epsilon}  \|_{L^\infty}\| \bar{\bar{n}}^{\epsilon}\|_{L^2}^2\leq C\| \bar{\bar{n}}^{\epsilon}\|_{L^2}^2(\| n^{\epsilon}\|_{H^2}+\| \bar{n}^{\epsilon}\|_{H^2}).
\end{split}
\end{equation}
Putting the estimates \eqref{x1}-\eqref{x4} into \eqref{3.44} and using the definition of $\textbf{t}_J^\epsilon$, we arrive at
\begin{equation} \label{3.48}
\begin{split}
 \mathbb{E}\sup_{r\in [0,\textbf{t}_J^\epsilon\wedge t]}\|\bar{\bar{n}}^{\epsilon}(r)\|_{L^2}^2 \leq C(\epsilon,J)\mathbb{E}\int_0^{\textbf{t}_J^\epsilon \wedge t}\|(\bar{\bar{n}}^{\epsilon},\bar{\bar{c}}^{\epsilon},\bar{\bar{u}}^{\epsilon})\|_{\textbf{L}^{2} } ^2\textrm{d}r,\quad \forall t>0.
\end{split}
\end{equation}
For the $c^\epsilon$-equation, one can make the similar estimate to \eqref{3.48} and deduce that
\begin{equation} \label{3.49}
\begin{split}
  \mathbb{E}\sup_{r\in [0,\textbf{t}_J^\epsilon \wedge t]}\|\bar{\bar{c}}^{\epsilon}(r)\|_{L^2}^2 \leq C(\epsilon,J )\mathbb{E}\int_0^{\textbf{t}_J^\epsilon\wedge t}\|(\bar{\bar{n}}^{\epsilon},\bar{\bar{c}}^{\epsilon},\bar{\bar{u}}^{\epsilon})\|_{\textbf{L}^{2}} ^2\textrm{d}r.
\end{split}
\end{equation}
The discussion of the $u^\epsilon$-equation will be decomposed into two cases.

\textbf{Caes 1: $\alpha\in [\frac{1}{2},1)$.} According to the fluid equations in \eqref{3.1}, we have
\begin{equation*}
\begin{split}
\textrm{d}  \bar{\bar{u}}^{\epsilon}(t)=& - (-\Delta)^\alpha  \bar{\bar{u}}^{\epsilon}\textrm{d}t-\nabla(P-\bar{P}) \textrm{d}t+ (\bar{\bar{n}}^{\epsilon}\nabla \phi)*\rho^{\epsilon}\textrm{d}t -\left[  ( u ^{\epsilon}\cdot \nabla) u ^{\epsilon}- (\bar{ u }^{\epsilon}\cdot \nabla) \bar{ u }^{\epsilon}\right]\textrm{d}t\\
& + \left(   f(t, u ^{\epsilon})-     f(t,\bar{ u }^{\epsilon})\right)\mathrm{d} W_t\\
:= &- (-\Delta)^\alpha  \bar{\bar{u}}^{\epsilon}\textrm{d}t-\nabla(P-\bar{P}) \textrm{d}t+(\mathcal {B}_1+\mathcal {B}_2)\textrm{d}t+ \mathcal {B}_3 \textrm{d} W_t.
\end{split}
\end{equation*}
Then by applying the It\^{o} formula  to  $ \|\bar{\bar{u}}^{\epsilon}(t)\|_{L^2}^2$ and using the divergence-free condition $\textrm{div} \bar{\bar{ u }}^{\epsilon}=0$, we get
\begin{equation}\label{3.50}
\begin{split}
&\|\bar{\bar{u}}^{\epsilon}(t)\|_{L^2}^2+ 2\int_0^t\|(-\Delta)^\frac{\alpha}{2}\bar{\bar{u}}^\epsilon\|_{L^2}^2 \textrm{d}r\\
&\quad =\int_0^t \left(2 ( \bar{\bar{u}}^{\epsilon},\mathcal {B}_1)_{L^2}+2 ( \bar{\bar{u}}^{\epsilon},\mathcal {B}_2)_{L^2}  + \|\mathcal {B}_3\|_{L_2(U; L^2)}^2\right) \textrm{d}r +2\int_0^t( \bar{\bar{u}}^{\epsilon},\mathcal {B}_3 \textrm{d} W_r)_{L^2}.
\end{split}
\end{equation}
For  $\mathcal {B}_1$, we use the property $\|f*\rho^{\epsilon}\|_{L^2}\leq C\| \rho^{\epsilon}\|_{L^2}  $ to obtain
\begin{align*}
 2( \bar{\bar{u}}^{\epsilon},\mathcal {B}_1)_{L^2}&\leq C \|\bar{\bar{u}}^{\epsilon} \|_{L^2}\|(\bar{\bar{n}}^{\epsilon}\nabla \phi)*\rho^{\epsilon}\|_{L^2} \leq C \|\nabla\phi\|_{L^\infty}(\|\bar{\bar{u}}^{\epsilon} \|_{L^2}^2+\| \bar{\bar{n}}^{\epsilon} \|_{L^2}^2).
\end{align*}
For  $\mathcal {B}_2$, we get from the condition $\textrm{div} \bar{ u }^{\epsilon}=0$ and  inequality $ \|\nabla   u^{\epsilon}\|_{\frac{2}{\alpha}}\leq C \|\Delta^{\frac{\alpha}{2}}u^{\epsilon}\|_{L^2}^{2\alpha -1}\|\Delta^{\frac{\alpha+1}{2}}u^{\epsilon}\|_{L^2}^{2-2\alpha }  $, $\alpha\in[1/2,1)$ that
\begin{equation*}
\begin{split}
2( \bar{\bar{u}}^{\epsilon},\mathcal {B}_2)_{L^2}&= 2(\bar{\bar{u}}^{\epsilon}, (\bar{\bar{u}}^{\epsilon}\cdot \nabla)  u ^{\epsilon})_{L^2}\\
&\leq C \|\bar{\bar{u}}^{\epsilon}\|_{L^{ \frac{2}{1-\alpha}}}   \|\nabla   u^{\epsilon}\|_{\frac{2}{\alpha}} \|\bar{\bar{ u }}^{\epsilon}\| _{L^2}\\
&\leq C\|\Delta^{\frac{\alpha}{2}}\bar{\bar{u}}^{\epsilon}\|_{L^2}  \|\Delta^{\frac{\alpha}{2}}u^{\epsilon}\|_{L^2}^{2\alpha -1}\|\Delta^{\frac{\alpha+1}{2}}u^{\epsilon}\|_{L^2}^{2-2\alpha }   \|\bar{\bar{ u }}^{\epsilon}\| _{L^2}\\
&\leq C\|\Delta^{\frac{\alpha}{2}}\bar{\bar{u}}^{\epsilon}\|_{L^2}  \|\Delta^{\frac{\alpha}{2}}u^{\epsilon}\|_{H^1} \|\bar{\bar{ u }}^{\epsilon}\| _{L^2}\\
& \leq C  \|\Delta^{\frac{\alpha}{2}}u^{\epsilon}\|_{H^1}^2 \|\bar{\bar{ u }}^{\epsilon}\| _{L^2}^2+ \|\Delta^{\frac{\alpha}{2}}\bar{\bar{u}}^{\epsilon}\|_{L^2}^2.
\end{split}
\end{equation*}
By using the assumption on $f$, it is clear that $
 \|\mathcal {B}_3\|_{L_2(U; L^2)}^2 \leq C \| \bar{\bar{u}}^{\epsilon}\|_{L^2}^2$. Therefore, we get from \eqref{3.50} that
\begin{equation} \label{x5}
\begin{split}
&\|\bar{\bar{u}}^{\epsilon}(t)\|_{L^2}^2+ \int_0^t\|(-\Delta)^\frac{\alpha}{2}\bar{\bar{u}}^\epsilon\|_{L^2}^2 \textrm{d}r =\int_0^t (1+\|\Delta^{\frac{\alpha}{2}}u^{\epsilon}\|_{H^1}^2)\left(  \|\bar{\bar{u}}^{\epsilon} \|_{L^2}^2+\| \bar{\bar{n}}^{\epsilon} \|_{L^2}^2 \right) \textrm{d}r +2\int_0^t( \bar{\bar{u}}^{\epsilon},\mathcal {B}_3 \textrm{d} W_r)_{L^2}.
\end{split}
\end{equation}
By using the BDG inequality and the assumption on $f$, we have
\begin{equation*}
\begin{split}
\mathbb{E}\sup_{r\in [0, t]}\bigg|2\int_0^r( \bar{\bar{u}}^{\epsilon},\mathcal {B}_3 \textrm{d} W_r)_{L^2}\bigg|&\leq C \mathbb{E} \bigg(\int_0^t\|\bar{\bar{u}}^{\epsilon}\|_{L^2}^2\|\mathcal {B}_3\|_{L_2(U;L^2)}^2 \textrm{d}r\bigg)^{1/2}\\
&\leq  \frac{1}{2}\mathbb{E} \sup_{r\in [0, t]}\|\bar{\bar{u}}^{\epsilon}(r)\|_{L^2}^2+ C\int_0^t\|\bar{\bar{u}}^{\epsilon}(r)\|_{L^2}^2\textrm{d}r.
\end{split}
\end{equation*}
By taking the supremum on both sides of \eqref{x5} and then expectation, we get
\begin{equation*}
\begin{split}
&\mathbb{E} \sup_{r\in [0, t]}\|\bar{\bar{u}}^{\epsilon}(r)\|_{L^2}^2+ 2 \mathbb{E}\int_0^t\|(-\Delta)^\frac{\alpha}{2}\bar{\bar{u}}^\epsilon\|_{L^2}^2 \textrm{d}r\leq\mathbb{E}\int_0^t (1+\|\Delta^{\frac{\alpha}{2}}u^{\epsilon}\|_{H^1}^2)\left(  \|\bar{\bar{u}}^{\epsilon} \|_{L^2}^2+\| \bar{\bar{n}}^{\epsilon} \|_{L^2}^2 \right) \textrm{d}r .
\end{split}
\end{equation*}
Since $\alpha<1$, by replacing $t$ with $\textbf{t}_J^\epsilon \wedge t$ in the last estimate, we get from the definition of $\textbf{t}_J^\epsilon$ that
\begin{equation} \label{x6}
\begin{split}
&\mathbb{E} \sup_{r\in [0, \textbf{t}_J^\epsilon \wedge t]}\|\bar{\bar{u}}^{\epsilon}(r)\|_{L^2}^2+ 2 \mathbb{E}\int_0^{\textbf{t}_J^\epsilon \wedge t}\|(-\Delta)^\frac{\alpha}{2}\bar{\bar{u}}^\epsilon\|_{L^2}^2 \textrm{d}r\leq C(J)\mathbb{E}\int_0^{\textbf{t}_J^\epsilon \wedge t} \left(  \|\bar{\bar{u}}^{\epsilon} \|_{L^2}^2+\| \bar{\bar{n}}^{\epsilon} \|_{L^2}^2 \right) \textrm{d}r .
\end{split}
\end{equation}
Putting \eqref{3.48}, \eqref{3.49} and \eqref{x6} together, we obtain
\begin{equation}  \label{x7}
\begin{split}
&\mathbb{E} \sup_{r\in [0, \textbf{t}_J^\epsilon \wedge t]}\|\bar{\bar{\textbf{u}}}^{\epsilon}(r)\|_{L^2}^2 \leq C(\epsilon,J)\mathbb{E}\int_0^{\textbf{t}_J^\epsilon \wedge t} \|\bar{\bar{\textbf{u}}}^{\epsilon}(r)\|_{L^2}^2 \textrm{d}r.
\end{split}
\end{equation}
By applying the Gronwall inequality to \eqref{x7}, we get
\begin{equation*}
\begin{split}
\mathbb{E} \sup_{r\in [0, \textbf{t}_J^\epsilon \wedge t]}\|\bar{\bar{\textbf{u}}}^{\epsilon}(r)\|_{L^2}^2=0,\quad \forall t> 0,
\end{split}
\end{equation*}
which implies that $\bar{\bar{\textbf{u}}}^{\epsilon} (t)= 0$, for all $ t \in [0,\textbf{t}_J^\epsilon\wedge \tau^\epsilon \wedge \bar{\tau}^\epsilon]$. Sending $J\rightarrow+\infty$ in the last identity leads to the desired result.

\textbf{Caes 2: $\alpha=1$.} The proof  is similar to Case 1, and the only difference is the estimate for $( \bar{\bar{u}}^{\epsilon},\mathcal {B}_2)_{L^2}$. Indeed, it follows from the H\"{o}lder inequality that
$
2( \bar{\bar{u}}^{\epsilon},\mathcal {B}_2)_{L^2} \leq C \|\bar{\bar{u}}^{\epsilon}\|_{L^2}^2 \|\nabla u ^{\epsilon}\|_{L^\infty}\leq C \|\bar{\bar{u}}^{\epsilon}\|_{L^2}^2 \| u ^{\epsilon}\|_{H^3},
$
which together with the same argument in Case 1 lead to \eqref{x7}. The proof of Lemma \ref{lem3.7} is completed.
\end{proof}

\begin{lemma}  \label{lem3.8}
Under the assumptions (H1)-(H3), the system \eqref{3.1} admits a unique local strong pathwise solution. More precisely, there exists an almost surely positive  stopping time $\widetilde{\textbf{t}^\epsilon}$ and a triple $\textbf{u}^\epsilon=(n^\epsilon,c^\epsilon, u ^\epsilon) \in C([0,\widetilde{\textbf{t}^\epsilon}); \textbf{H}^s(\mathbb{R}^2))$, $s>5$, such that the equation
\begin{equation*}
\begin{split}
\textbf{u}^{\epsilon}(t)-\textbf{u}^{ \epsilon}_0+\int_0^t \textbf{A}^\alpha\textbf{u}^{\epsilon}\textrm{d}r+\int_0^t \textbf{B}(\textbf{u}^{\epsilon})\textrm{d}r=\int_0^t\textbf{F}^\epsilon (\textbf{u}^{\epsilon})\textrm{d}r+\int_0^t \textbf{G}  (r,\textbf{u}^{\epsilon})\textrm{d} \mathcal {W}_r
 \end{split}
\end{equation*}
holds for all $t\in [0,\widetilde{\textbf{t}^\epsilon})$, $\mathbb{P}$-a.s.
\end{lemma}

\begin{proof}
 Let $\textbf{u}^{R,\epsilon}$ be the global pathwise solution to \eqref{3.2} constructed in Lemma \ref{lem3.4}.  Note that for initial data $\textbf{u}^{R,\epsilon}(0)$, there exists a constant $M>0$ such that
 $
 \|\textbf{u}^{R,\epsilon}(0)\|_{\textbf{H}^s}\leq M.
 $
Denote by $C_{\textrm{emb}}>0$ the embedding constant $\| \cdot\|_{W^{1,\infty}} \leq C_{\textrm{emb}}\| \cdot\|_{H^s}$. Define
 $$
 \textbf{t}^\epsilon := \inf\left\{t>0;~ \|\textbf{u}^{R,\epsilon}(t)\|_{\textbf{H}^s}>2M\right\}.
 $$
It is clear that $\mathbb{P}\{\textbf{t}^\epsilon>0\}=1$, and for any $t\in [0,\textbf{t}^\epsilon]$
$$
\|\textbf{u}^{R,\epsilon}(t)\|_{\textbf{W}^{1,\infty}} \leq C_{\textrm{emb}}\|\textbf{u}^{R,\epsilon}(t)\|_{\textbf{H}^s} \leq C_{\textrm{emb}} M.
$$
Therefore,
$$
\theta_R(\|\textbf{u}^{R,\epsilon}\|_{W^{1,\infty}})\equiv1,~~\textrm{for all}~ R>C_{\textrm{emb}} M.
$$
By choosing a fixed $R_0>0$ large enough, the process $\textbf{u}^\epsilon:= \textbf{u}^{R_0,\epsilon}$ is a local pathwise solution to \eqref{3.1} or \eqref{3.2} over the interval $[0,\textbf{t}^\epsilon]$.

To extend the solution $  \textbf{u} ^\epsilon $ to a maximal existence time $\widetilde{\textbf{t}^\epsilon}$, we denote by $\textbf{T}^\epsilon$ the set of all strictly positive stopping times with respect to solutions starting from $(n^\epsilon_0,c^\epsilon_0,u^\epsilon_0)$. Clearly, $\textbf{T}^\epsilon\neq \emptyset$ (since $\textbf{t}^\epsilon \in \textbf{T}^\epsilon$), and for any $
\textbf{t}_1, \textbf{t}_2 \in \textbf{T}^\epsilon \Rightarrow \textbf{t}_1 \vee \textbf{t}_2 \in \textbf{T}^\epsilon$, $\textbf{t}_1 \wedge \textbf{t}_2 \in \textbf{T}^\epsilon.$
Define
\begin{equation} \label{3.55}
\begin{split}
\widetilde{\textbf{t}^\epsilon}:=  \operatorname{ess} \sup\{ \textbf{t};~\textbf{t}\in \textbf{T}^\epsilon\},
\end{split}
\end{equation}
which is strictly positive $\mathbb{P}$-a.s., and there is an increasing sequence $\{\textbf{t}_r\} \subset \textbf{T}^\epsilon$ such that $\lim _{r\rightarrow \infty} \textbf{t}_r=\widetilde{\textbf{t}^\epsilon}$. Setting $\textbf{u}^\epsilon := \textbf{u}^\epsilon|_{[0,\textbf{t}_r]}$, we infer from the uniqueness result established in Lemma \ref{lem3.7} that $\textbf{u}^\epsilon$ is a solution defined on $\cup_{r>0} [0, \textbf{t}_r]$. For each $L \in \mathbb{R}_+$, define
$
b_L := \widetilde{\textbf{t}^\epsilon}\wedge \inf \left\{t \in[0, T] \mid\|\textbf{u}^\epsilon(t)\|_{\textbf{W}^{1,\infty}} \geq L\right\},
$
which is a sequence of positive stopping times if $L>M$. Then the quadruple $(\textbf{u}^\epsilon , \textbf{t}_R^{'})$, with $\textbf{t}_L^{'}=\textbf{t}_L \vee b_L$, provides a local pathwise solution to \eqref{3.1}. Assume that $\mathbb{P}(\textbf{t}_L^{'}=\widetilde{\textbf{t}^\epsilon}<T)>0$, then the solution starting from $\textbf{u}^\epsilon(\textbf{t}_L^{'}) $ can be uniquely extended to $[0, \textbf{t}_L^{'}+\sigma]$ for some strictly positive stopping time $\sigma$, which implies that
$
\textbf{t}_L^{'}+\sigma \in \textbf{T}^\epsilon$ and $ \mathbb{P}^{\epsilon}(\widetilde{\textbf{t}^\epsilon} <\textbf{t}_L^{'}+\sigma)>0,
$
which contradicts to the maximality of $b_L $ in \eqref{3.55}, and we infer that $\textbf{t}_L^{'}\rightarrow\widetilde{\textbf{t}^\epsilon}$ as $L\rightarrow\infty$, and hence
$\sup _{t \in[0, \textbf{t}_L^{'}]}\|\textbf{u}^\epsilon(t)\|_{\textbf{W}^{1,\infty}} \geq L$  on $[\widetilde{\textbf{t}^\epsilon}<T]$. The proof of Lemma \ref{lem3.8} is completed.
\end{proof}

\begin{lemma} \label{lem3.9}
The local pathwise solution $ \textbf{u} ^\epsilon $  to the regularized system \eqref{3.1} constructed in Lemma \ref{lem3.8} exists globally, namley, $\mathbb{P}\{\widetilde{\textbf{t}^\epsilon}=+\infty\}=1$.
\end{lemma}

\begin{proof}
We prove Lemma \ref{lem3.9}  by using some ideas in \cite[Theorem 3.3]{zhai20202d}. First, similar to \cite[Proposition 3.1]{nie2020global}, by using $\textrm{div} u^\epsilon =0$ and the nonnegativity of $c^\epsilon$ and $n^\epsilon$, it is standard to derive that
$$
\mathbb{E}\sup_{t\in [0,T]}\|\textbf{u} ^\epsilon(t)\|_{\textbf{L}^2}^2 +\mathbb{E} \int_0^T (\|( n^\epsilon, c^\epsilon)\|_{\textbf{H}^1}^2+\|  u ^\epsilon  \|_{H^{\alpha}}^2) \textrm{d}t <C(\textbf{u}_0 ,\epsilon,\phi).
$$
Moreover, we have
\begin{equation}\label{3.56}
\begin{split}
&\mathbb{E}\sup_{t\in [0,T\wedge  n _K^\epsilon]}\|\textbf{u} ^\epsilon(t)\|_{\textbf{H}^1}^2 +\mathbb{E} \int_0^{T\wedge  n _K^\epsilon} \left(\|( n^\epsilon, c^\epsilon)(t  )\|_{\textbf{H}^2}^2+\|  u ^\epsilon (t ) \|_{H^{1+\alpha}}^2\right) \textrm{d}t \leq C(\textbf{u}_0 ,\epsilon,\phi,K),
\end{split}
\end{equation}
where
$$
 n _K^\epsilon:= \inf\left\{t>0;~ \sup_{r\in [0,t]}\|\textbf{u} ^\epsilon(r)\|_{\textbf{L}^2}^2>K~ \textrm{or}~  \int_0^t (\|( n^\epsilon, c^\epsilon)\|_{\textbf{H}^1}^2+\|  u ^\epsilon  \|_{H^{\alpha}}^2) \textrm{d}r>K\right\} .
$$
Clearly, $ n _K^\epsilon\rightarrow\infty$ as $K\rightarrow\infty$, $\mathbb{P}$-a.s. To get an estimate of  $\textbf{u} ^\epsilon$ in $\textbf{H}^s(\mathbb{R}^2)$ for all $s>1$, one need first to estimate the norm of $\|\nabla  u ^{\epsilon}\|_{L^\infty}$, which can be achieved by estimating the norm of vorticity $\| v ^{\epsilon}\|_{H^{1+\alpha}}= \|\nabla \wedge  u ^{\epsilon}\|_{H^{1+\alpha}}$. Indeed, we define for each $J>0$
$
 \textbf{t}_ {R}^\epsilon:= \inf\{t>0;  \int_0^t \| u ^\epsilon(r)\|_{H^{1+\alpha}}^2 \textrm{d}r> J\}.
$
By \eqref{3.56}, we infer that $ T\wedge  n _K^\epsilon\wedge \textbf{t}_ {J}^\epsilon\rightarrow T\wedge n _K^\epsilon $ as $J\rightarrow\infty$, $\mathbb{P}$-a.s.  Taking the operator $\nabla(\nabla \wedge)$ to both sides of $\eqref{3.1}_3$ and then applying the It\^{o} formula  to  {$ \|\nabla v ^{\epsilon}(t)\|_{L^2}^2$}, we find
\begin{equation}\label{3.57}
\begin{split}
&\|\nabla v ^{\epsilon}(t)\|_{L^2}^2+ 2\int_0^t\| (-\Delta)^\frac{\alpha}{2} \nabla v ^{\epsilon}\|_{L^2}^2\textrm{d}r  = \|\nabla v ^{\epsilon}_0\|_{L^2}^2\\
&\quad\underbrace{-2\int_0^t(\nabla v ^{\epsilon},\nabla [( u ^{\epsilon}\cdot \nabla) v ^{\epsilon}])_{L^2}\textrm{d}r}_{:= A } +\underbrace{2\int_0^t(\nabla v ^{\epsilon},\textbf{P}\nabla\{\nabla \wedge [(n^{\epsilon}\nabla \phi)*\rho^{\epsilon}]\})_{L^2}\textrm{d}r}_{:= B }\\
& \quad+ \int_0^t\|\textbf{P}\nabla[\nabla \wedge f(r, u ^{\epsilon})]\|_{L_2(U;L^2)}^2\textrm{d}r +2\int_0^t(\nabla v ^{\epsilon},\textbf{P}\nabla[\nabla \wedge f(r, u ^{\epsilon})] \mathrm{d} W_r)_{L^2}.
\end{split}
\end{equation}
For  $\alpha =1$, we get by the Young inequality that
\begin{equation*}
\begin{split}
&  A =2\int_0^t(\Delta v ^{\epsilon},( u ^{\epsilon}\cdot \nabla)  v ^{\epsilon})_{L^2}\textrm{d}r \leq  \frac{1}{2} \int_0^t\|\Delta v ^{\epsilon}\|_{L^2}^2\textrm{d}r+ C \int_0^t\|\nabla  u ^{\epsilon}\|_{L^2}^2\|\nabla v ^{\epsilon}\|_{L^2}^2 \textrm{d}r.
\end{split}
\end{equation*}
For $\frac{1}{2}\leq\alpha <1$, recall the Sobolev embedding in  {\cite[P. 75]{miao2012littlewood}}
$$
\dot{H}^\alpha(\mathbb{R}^2)\subset B_{\frac{2}{1- \alpha},2}^0(\mathbb{R}^2)\subset L^{\frac{2}{1- \alpha}}(\mathbb{R}^2)  ~~ ~\textrm{and} ~~~\dot{H}^{1-\alpha}(\mathbb{R}^2) \subset L^ \frac{2}{\alpha}(\mathbb{R}^2).
$$
It follows that
\begin{equation*}
\begin{split}
 A =2\int_0^t(\nabla v ^{\epsilon},(\nabla  v ^{\epsilon}\cdot \nabla) u ^{\epsilon}  )_{L^2}\textrm{d}r\leq \frac{1}{2}  \int_0^t\|\nabla v ^{\epsilon}\|_{\dot{H}^\alpha}^2\textrm{d}r+ C \int_0^t\|\nabla u ^{\epsilon}\|_{\dot{H}^{1-\alpha}}^2\|\nabla v ^{\epsilon}\|_{L^2}^2 \textrm{d}r.
\end{split}
\end{equation*}
An application of the Young inequality implies
\begin{equation*}
\begin{split}
&  B \leq  C(\epsilon,\phi) \int_0^t (\|\nabla  v ^{\epsilon} \|_{L^2}^2+ \| n^{\epsilon} \|_{L^2}^2) \textrm{d}r \leq  C(\epsilon,\phi,n_0) e^{Ct} + C(\epsilon,\phi)\int_0^t \|\nabla v ^{\epsilon} \|_{L^2}^2\textrm{d}r.
\end{split}
\end{equation*}
Plugging the estimates of $ A $ and $ B $ into \eqref{3.57}, we get from the Gronwall Lemma that
\begin{equation}\label{4.73}
\begin{split}
&\mathbb{E}\sup_{r \in [0,t\wedge  n _K^\epsilon\wedge \textbf{t}_ {R}^\epsilon]}\|\nabla v ^{\epsilon}(r)\|_{L^2}^2+\mathbb{E}\int_0^{t\wedge  n _K^\epsilon\wedge \textbf{t}_ {R}^\epsilon}\| (-\Delta)^\frac{\alpha}{2} \nabla v ^{\epsilon}\|_{L^2}^2\textrm{d}r \\
 &\quad\leq  C(\epsilon,\phi,n_0,u_0,R) e^{C t\wedge  n _K^\epsilon\wedge \textbf{t}_ {R}^\epsilon} \bigg(1  +\mathbb{E}\int_0^{t\wedge  n _K^\epsilon\wedge \textbf{t}_ {R}^\epsilon}\| \nabla[\nabla \wedge f(r, u ^{\epsilon})]\|_{L_2(U;L^2)}^2\textrm{d}r \\
 & \quad\quad+ \mathbb{E}\sup_{r \in [0,t\wedge  n _K^\epsilon\wedge \textbf{t}_ {R}^\epsilon]}\bigg|\int_0^r(\nabla v ^{\epsilon},\textbf{P}\nabla[\nabla \wedge f(\varsigma, u ^{\epsilon})] \mathrm{d} W_\varsigma)_{L^2}\bigg|\bigg).
\end{split}
\end{equation}
Applying the BDG inequality to the last term, we infer from \eqref{4.73} that
\begin{equation}\label{3.59}
\begin{split}
&\mathbb{E}\sup_{r \in [0,T\wedge  n _K^\epsilon\wedge \textbf{t}_ {J}^\epsilon]}\| u ^{\epsilon}(r)\|_{H^2}^2+\mathbb{E}\int_0^{T\wedge  n _K^\epsilon\wedge \textbf{t}_ {J}^\epsilon}\| u ^{\epsilon}\|_{H^{2+\alpha}}^2\textrm{d}r \leq  \exp\{C(\epsilon,\phi,n_0,u_0,J) \exp\{C T \}\},
\end{split}
\end{equation}
where we used the facts of
$$
\|\nabla v ^{\epsilon}\|_{L^2}^2=\|\Delta u ^{\epsilon}\|_{L^2}^2=\|  u ^{\epsilon}\|_{\dot{H}^{2 }}^2 ~~~ \textrm{and} ~~~\|\nabla v ^{\epsilon}\|_{\dot{H}^\alpha}^2=\|\Delta u ^{\epsilon}\|_{\dot{H}^\alpha}^2=\| u ^{\epsilon}\|_{\dot{H}^{2+\alpha }}^2.
$$

For any $S>0$, we define
$$
\textbf{t}_{T,S,J,K}^{''}=T\wedge  m _S^\epsilon\wedge n _K^\epsilon\wedge \textbf{t}_ {J}^\epsilon ,\quad \forall T>0,
$$
where
$
  m _S^\epsilon:=\inf\{t>0; \int_0^{t} \|\nabla  u ^{\epsilon}\|_{L^\infty}\textrm{d}r > S ~\textrm{or}~\int_0^{t} \| n^{\epsilon}\|_{H^2}^2\textrm{d}r > S~\textrm{or}~\int_0^{t} \|  c^{\epsilon}\|_{H^2}^2\textrm{d}r > S \}.
$
Then \eqref{3.59} implies that
$ \textbf{t}_{T,S,J,K}^{''}\rightarrow T\wedge   n _K^\epsilon\wedge \textbf{t}_ {J}^\epsilon$ as $S\rightarrow\infty$, $\mathbb{P}$-a.s.

For the first two random PDEs in \eqref{3.1}, one can derive that
\begin{equation}\label{4.75}
\begin{split}
& \mathbb{E}\sup_{t \in [0,T\wedge\textbf{t}_{T,S,J,K}^{''}]} \|(n^{\epsilon},c^{\epsilon})(t) \|_{\textbf{H}^s}^2
+\mathbb{E}\int_0^{T\wedge\textbf{t}_{T,S,J,K}^{''}}\| (n^{\epsilon},c^{\epsilon}) \|_{\textbf{H}^{s+1}}^2\textrm{d}t \\
&\quad \leq\|(n^{\epsilon}_0,c^{\epsilon}_0)\|_{\textbf{H}^s}^2+ C( \epsilon,S) \mathbb{E}\int_0^{T\wedge\textbf{t}_{T,S,J,K}^{''}} \|(n^{\epsilon},c^{\epsilon}) \|_{\textbf{H}^s}^2 \textrm{d}t.
\end{split}
\end{equation}
Recall the Littlewood-Paley blocks $\triangle_j$ defined in Section 2, we get by the It\^{o} formula that
\begin{align}\label{4.76}
& \sup_{r\in [0,\textbf{t}_{T,S,J,K}^{''}]} \|\triangle_j  u ^\epsilon (r)\|_{L^2}^2+2 \int_0^{\textbf{t}_{T,S,J,K}^{''}}\|\triangle_j (-\Delta)^\frac{\alpha}{2} u ^{\epsilon}\|_{L^2}^2\textrm{d}r\nonumber\\
&\quad = \|\triangle_j u_0^\epsilon\|_{L^2}^2+ \int_0^{\textbf{t}_{T,S,J,K}^{''}}\|\triangle_j\textbf{P} f(r, u ^{\epsilon}) \|_{L_2(U;L^2)}^2 \textrm{d}r\nonumber+2 \int_0^{\textbf{t}_{T,S,J,K}^{''}} (\triangle_j  u ^\epsilon, \triangle_j\textbf{P}(n^{\epsilon}\nabla \phi)*\rho^{\epsilon} )_{L^2}\textrm{d}r\\
&\quad\quad+2  \int_0^{\textbf{t}_{T,S,J,K}^{''}}(\triangle_j  u ^\epsilon,\textbf{P}[ u ^{\epsilon}\cdot \nabla,\triangle_j]  u ^{\epsilon}  )_{L^2}\textrm{d}r +2  \sup_{r\in [0,\textbf{t}_{T,S,J,K}^{''}]} \left|\int_0^r(\triangle_j  u ^\epsilon,  \triangle_j\textbf{P} f(\varsigma, u ^{\epsilon}) \mathrm{d} W_\varsigma)_{L^2}\right| \nonumber\\
 &\quad := \|\triangle_j u_0^\epsilon\|_{L^2}^2+ \int_0^{\textbf{t}_{T,S,J,K}^{''}}\left(\mathcal {Q}_1^j+\mathcal {Q}_2^j+\mathcal {Q}_3^j\right)\textrm{d}r+ \mathcal {Q}_4^j.
\end{align}
By using the Minkowski inequality (cf.  \cite[Proposition 1.3]{bahouri2011fourier}), we have
\begin{equation*}
\begin{split}
 \|\{2^{2js}\mathcal {Q}_2^j\}_{j\geq -1}\|_{l^1}
   &\leq C \left\|\{2^{js} \|\triangle_j  u ^\epsilon\|_{L^2}\cdot 2^{js}\| \triangle_j\textbf{P}(n^{\epsilon}\nabla \phi)*\rho^{\epsilon} \|_{L^2}\}_{j\geq -1}\right\|_{l^1}\\
   &\leq C  \left\|\{2^{ js} \|\triangle_j  u ^\epsilon\|_{L^2}\}_{j\geq -1}\right\|_{ \ell^2 }\left\|\{2^{ js} \| \triangle_j (n^{\epsilon}\nabla \phi)*\rho^{\epsilon} \|_{L^2}\}_{j\geq -1}\right\|_{ \ell^2 }\\
   &\leq C  \|  u ^{\epsilon}\|_{H^s} \| (n^{\epsilon}\nabla \phi)*\rho^{\epsilon} \|_{H^s} \\
   &\leq C \|\phi\|_{W^{1,\infty}} \| ( u ^{\epsilon},n^{\epsilon})\|_{H^s}^2,
\end{split}
\end{equation*}
for all $t\in[0,\textbf{t}_{T,S,J,K}^{''}]$. By using the discrete H\"{o}lder inequality and {the commutator estimate (cf. {\cite[Lemma 2.100]{bahouri2011fourier}})}, we get
\begin{equation*}
\begin{split}
 \|\{2^{2js}\mathcal {Q}_3^j\}_{j\geq -1}\|_{l^1}
 &\leq 2  \big\|\{2^{ js}\|\triangle_j  u ^\epsilon\|_{L^2} \cdot2^{js}\| [ u ^{\epsilon}\cdot \nabla,\triangle_j]  u ^{\epsilon}\|_{L^2} \}_{j\geq -1}\big\|_{l^1}\\
 &\leq C \| u ^{\epsilon}\|_{H^s} \|\{ 2^{js}\| [ u ^{\epsilon}\cdot \nabla,\triangle_j] u ^{\epsilon}\|_{L^2} \}_{j\geq -1}\big\|_{ \ell^2 } \\
  &\leq  CS \| u ^{\epsilon}\|_{H^s}^2 .
\end{split}
\end{equation*}
Multiplying both sides of \eqref{4.76} by $2^{2js}$ and summing up with respect to $j\geq-1$. After taking the mathematical expectation, we get from the Monotone Convergence Theorem that
\begin{equation*}
\begin{split}
&\mathbb{E}\sup_{r\in [0,\textbf{t}_{T,S,J,K}^{''}]}\|  u ^{\epsilon}(r)\|_{H^s}^2  \leq \| u_0^{\epsilon}\|_{H^s}^2+C(\phi,S) \mathbb{E}\int_0^{\textbf{t}_{T,S,J,K}^{''}} \left(1+ \| ( u ^{\epsilon},n^{\epsilon})\|_{\textbf{H}^s}^2\right)\textrm{d}r\\
& \quad +C\sum_{j\geq -1} 2^{2js}\mathbb{E}\sup_{r\in [0,\textbf{t}_{T,S,J,K}^{''}]} \left|\int_0^r(\triangle_j  u ^\epsilon,  \triangle_j\textbf{P} f(\varsigma, u ^{\epsilon}) \mathrm{d} W_\varsigma)_{L^2}\right|\\
&\leq\| u_0^{\epsilon}\|_{H^s}^2+C(\phi,S) \mathbb{E}\int_0^{\textbf{t}_{T,S,J,K}^{''}}\left (1+ \| ( u ^{\epsilon},n^{\epsilon})\|_{\textbf{H}^s}^2\right)\textrm{d}r\\
& \quad+\frac{1}{2}\sum_{j\geq -1} 2^{2js}\mathbb{E} \sup_{r\in [0,\textbf{t}_{T,S,J,K}^{''}]}\|\triangle_j  u ^\epsilon\|_{L^2}^2
 + C\sum_{j\geq -1} 2^{2js}\mathbb{E}  \int_0^{\textbf{t}_{T,S,J,K}^{''}} \|\triangle_j  f(r, u ^{\epsilon})\|_{L_2(U;L^2)}^2 \textrm{d}r \\
&\quad\leq \| u_0^{\epsilon}\|_{H^s}^2+\frac{1}{2}\mathbb{E}\sup_{r\in [0,\textbf{t}_{T,S,J,K}^{''}]}\|  u ^{\epsilon}(r)\|_{H^s}^2 +C(\phi,S) \mathbb{E}\int_0^{\textbf{t}_{T,S,J,K}^{''}} \left(1+ \| ( u ^{\epsilon},n^{\epsilon})\|_{\textbf{H}^s}^2\right)\textrm{d}r,
\end{split}
\end{equation*}
which implies
$$
\mathbb{E}\sup_{t\in [0,\textbf{t}_{T,S,J,K}^{''}]}\| u ^{\epsilon}(t)\|_{H^s}^2\leq 2 \| u_0^{\epsilon}\|_{H^s}^2 +C(\phi,S) \mathbb{E}\int_0^{\textbf{t}_{T,S,J,K}^{''}}\left (1+ \| ( u ^{\epsilon},n^{\epsilon})\|_{\textbf{H}^s}^2\right)\textrm{d}t.
$$
Adding this estimate to \eqref{4.75} leads to
\begin{equation} \label{4.77}
\begin{split}
& \mathbb{E}\sup_{t \in [0,T\wedge\textbf{t}_{T,S,J,K}^{''}]} \| \textbf{u} ^{\epsilon} (t)\|_{\textbf{H}^s}^2
 \leq \exp\{C(\textbf{u}_0, \epsilon,T,\phi,S,J,K)\}.
\end{split}
\end{equation}
Define
$
\textbf{u}_{J,S,K}^\epsilon(t):=\textbf{u}^\epsilon(t\wedge m _S^\epsilon\wedge n _K^\epsilon\wedge \textbf{t}_ {J}^\epsilon)$, $\forall J,S,K>0.
$
Then from estimate \eqref{4.77}, we infer that $\textbf{u}_{J,S,K}^\epsilon(t)$ is a solution to \eqref{3.1} over $[0,T]$, such that
$\mathbb{E}\sup_{t \in [0,T ]} \| \textbf{u}_{J,S,K} ^{\epsilon} (t)\|_{\textbf{H}^s}^2
 \leq \exp\{C(\textbf{u}_0,\epsilon,\phi, T,S,J,K)\}.
$
Therefore, it follows from the last inequality that  for any given $ J,S,K>0$
\begin{equation} \label{3.64}
\begin{split}
\widetilde{\textbf{t}^\epsilon}\geq T\wedge  m _S^\epsilon\wedge n _K^\epsilon\wedge \textbf{t}_ {J}^\epsilon,\quad \mathbb{P}\textrm{-a.s.,}
\end{split}
\end{equation}
where $\widetilde{\textbf{t}^\epsilon}$ is the maximal existence time. Sending $S\rightarrow \infty$, $J\rightarrow \infty$, $K\rightarrow \infty$ and $T\rightarrow+\infty$ successively in \eqref{3.64}, we obtain $\mathbb{P}\{\widetilde{\textbf{t}^\epsilon}=\infty\}=1$, which implies that the solution $\textbf{u} ^{\epsilon} $ exists globally. This finishes the proof of Lemma \ref{lem3.9}.
\end{proof}

\section{Identification of the limit as $\epsilon\rightarrow 0$}
Let $\mathbf{u}^\epsilon=\left(n^\epsilon, c^\epsilon, u^\epsilon\right) \in L^p\left(\Omega ; C\left([0, T] ; \mathbf{H}^s\left(\mathbb{R}^2\right)\right)\right.$, $s>5$, be smooth approximate solutions constructed in the Section 3. The objective of this section is to establish the main result, i.e., Theorem \ref{main}, by identifying the limit as $\epsilon \rightarrow 0$ (up to a subsequence). The proof relies on a series of entropy and energy estimates that are uniform in $\epsilon$, as well as the stochastic compactness method.

\subsection {A priori estimates}\label{sec3.1}
\begin{lemma}\label{lem4.0}
  {Under the assumptions of (H1)-(H3),} for any $T>0$, there holds
$$
 {n^\epsilon (t,x)> 0 ,~~~c^\epsilon (t,x) > 0,}
$$
for all $(t,x)\in [0,T]\times \mathbb{R}^2$, $\mathbb{P}$-a.s.
\end{lemma}

\begin{proof}
 Define $n^{\epsilon,-}:= \max\{-n^{\epsilon},0\}$. Multiplying both sides of \eqref{3.1}$_1$ by $n^{\epsilon,-}$, integrating by parts over $\mathbb{R}^2$ and using the fact of $\textrm{div} u^\epsilon=0$, we find
\begin{equation*}
\begin{split}
\frac{1}{2} \frac{\textrm{d}}{\textrm{d}t} \|n^{\epsilon,-}(t)\|_{L^2}^2  + \|\nabla n^{\epsilon,-}(t)\|_{L^2}^2
&= \int_{\mathbb{R}^2}\left[-   n^{\epsilon,-}\textrm{div} (n^{\epsilon}(\nabla c^{\epsilon} *\rho^{\epsilon}) )  +   (n^{\epsilon,-})^2-     n^{\epsilon,-}(n^{\epsilon} )^2 \right]\textrm{d}x\\
& \leq \| n^{\epsilon,-}(t)\|_{L^2}^2-\int_{\mathbb{R}^2}  n^{\epsilon,-}\textrm{div} (n^{\epsilon}(\nabla c^{\epsilon} *\rho^{\epsilon}) )  \textrm{d}x,\quad \mathbb{P}\textrm{-a.s.}
\end{split}
\end{equation*}
Integrating by parts and using the H\"{o}lder inequality, we have
\begin{equation*}
\begin{split}
-\int_{\mathbb{R}^2}  n^{\epsilon,-}\textrm{div} (n^{\epsilon}(\nabla c^{\epsilon} *\rho^{\epsilon}) )  \textrm{d}x&= \int_{\mathbb{R}^2}(\nabla c^{\epsilon} *\rho^{\epsilon}) n^{\epsilon,-}\nabla n^{\epsilon,-}  \textrm{d}x\\
&\leq \frac{1}{2}\|\nabla n^{\epsilon,-}(t)\|_{L^2}^2+ \frac{1}{2} \|\nabla c^{\epsilon} *\rho^{\epsilon} \|_{L^\infty}^2\| n^{\epsilon,-}(t)\|_{L^2}^2,
\end{split}
\end{equation*}
which combined with the last inequality leads to
\begin{equation*}
\begin{split}
 \frac{\textrm{d}}{\textrm{d}t} \|n^{\epsilon,-}(t)\|_{L^2}^2  +  \|\nabla n^{\epsilon,-}(t)\|_{L^2}^2
 \leq \left(2+  \|\nabla c^{\epsilon} *\rho^{\epsilon} \|_{L^\infty}^2\right)\| n^{\epsilon,-}(t)\|_{L^2}^2,\quad \mathbb{P}\textrm{-a.s.}
\end{split}
\end{equation*}
Applying the Gronwall Lemma to the last inequality leads to
\begin{equation*}
\begin{split}
 \|n^{\epsilon,-}(t)\|_{L^2}^2 \leq \|n_0^{\epsilon,-}\|_{L^2}^2 \exp\left\{\int_0^t \left(2+  \|\nabla c^{\epsilon} *\rho^{\epsilon} (r)\|_{L^\infty}^2\right)  \textrm{d}r\right\}\equiv0 , \quad \mathbb{P}\textrm{-a.s.,}
\end{split}
\end{equation*}
which implies that $n^{\epsilon}(t,x)> 0$ for almost all $(x,t)\in \mathbb{R}^2 \times [0,T]$, $\mathbb{P}$-a.s. Moreover, due to $H^s(\mathbb{R}^2)\hookrightarrow C(\mathbb{R}^2)$ with $s>1$, we have that $n^{\epsilon}(t,x)> 0$ for all $(x,t)\in \mathbb{R}^2 \times [0,T]$, $\mathbb{P}$-a.s.

 {Similarly, multiplying \eqref{3.1}$_2$ by $c^{\epsilon,-}:= \max\{-c^{\epsilon},0\}$ and integrating by parts over $\mathbb{R}^2$, one can easily deduce from the facts of $-  (c^{\epsilon,-})^2(n^{\epsilon}*\rho^{\epsilon})\leq 0$ and $\textrm{div} u^\epsilon =0$ that $c^{\epsilon,-}(t,x)=0$, for all $(x,t)\in \mathbb{R}^2 \times [0,T]$, $\mathbb{P}$-a.s., which implies the desired result. This finishes the proof of Lemma \ref{lem4.0}.}
\end{proof}

\begin{lemma}\label{lem4.1}
For any $T>0$, we have $\mathbb{P}$-a.s.
\begin{equation} \label{4.1}
\begin{split}
 \sup_{t\in [0,T]}\|c^\epsilon (t)\|_{ L^2 } +\int_0^T\| c^\epsilon (t)\|_{ H^1}^2 \textrm{d}t \leq \|c_0 \|_{L^2}  ,
 \end{split}
\end{equation}
\begin{equation} \label{4.2}
\begin{split}
\sup_{t\in [0,T]} \|c^\epsilon (t)\|_{ L^1\cap L^\infty}\leq\|c_0\|_{L^1\cap L^\infty}.
\end{split}
\end{equation}
Moreover, there is a constant $C>0$ independent of $\epsilon$ such that
 \begin{equation} \label{4.3}
\begin{split}
\sup_{t\in [0,T]}\| n^\epsilon(t) \|_{ L^1}+ \int_0^T\| n^\epsilon (t) \|_{ L^2}^2 \textrm{d}t \leq    \| n_0\|_{L^1}\exp\{CT\},~~\mathbb{P}\textrm{-a.s}.
\end{split}
\end{equation}
\end{lemma}

\begin{proof}
The estimates \eqref{4.1} and \eqref{4.2} can be obtained by the maximum principle (cf. \cite[Lemma 2.1]{winkler2012global}; \cite[Proposition 4.1]{nie2020global}). Integrating both sides of $\eqref{3.1}_1$ on $\mathbb{R}^2$ and using the identity $\textrm{div} ( u ^\epsilon n^\epsilon)= u ^\epsilon \cdot \nabla n^\epsilon$ lead  to
$$
\frac{\textrm{d} }{\textrm{d}t} \| n^\epsilon(\cdot,t)\|_{L^1}+ \| n^\epsilon(\cdot,t)\|_{L^2}^2\leq \| n^\epsilon(\cdot,t)\|_{L^1},\quad \mathbb{P}\textrm{-a.s.},
$$
which implies \eqref{4.3} by applying the Gronwall Lemma. The proof is completed.
\end{proof}

\begin{lemma} \label{lem4.2}
For any $p\geq1$ and $T>0$, we have
 \begin{equation}\label{3.117}
\begin{split}
 &\mathbb{E}\sup_{t\in [0,T]}\| u ^{\epsilon}(t) \|_{L^2}^{2p} + \mathbb{E}\left(\int_0^T\| u ^{\epsilon}(t)\|_{\dot{H}^\alpha}^2\textrm{d}t\right)^p\leq C \exp\{C T\},
\end{split}
\end{equation}
where the positive constant $C$ is independent of $\epsilon$.
\end{lemma}

\begin{proof}
Applying the It\^{o} formula  to $\| u ^{\epsilon}(r) \|_{L^2}^2$, we find that
\begin{equation}\label{4.5}
\begin{split}
 & \mathbb{E}\sup_{r\in [0,t]}\| u ^{\epsilon}(r) \|_{L^2}^{2p}  +2\mathbb{E}\left(\int_0^t\|(-\Delta)^{\frac{\alpha}{2}} u ^{\epsilon}\|_{L^2}^2\textrm{d}r\right)^p\\
 &\quad\leq C(p)\|u_0 \|_{L^2}^{2p} +C(p,n_0,\phi ,T) \left(1+ \mathbb{E}\int_0^t\|  u ^{\epsilon}(r)\|_{L^2 } ^{2p} \textrm{d}r\right)\\
 &\quad\quad+C(p)\mathbb{E}\sup_{r\in [0,t]}\left|\int_0^r( u ^{\epsilon},  \textbf{P} f(\varsigma, u ^{\epsilon}) \mathrm{d} W_\varsigma)_{L^2}\right|^p,
\end{split}
\end{equation}
where we used
$
( u ^{\epsilon},  (-\Delta)^\alpha  u ^{\epsilon}  )_{L^2}= \|(-\Delta)^{\frac{\alpha}{2}}  u ^{\epsilon}\|_{L^2}^2\geq 0.
$
By applying the BDG inequality, we have
\begin{equation}\label{4.6}
\begin{split}
\mathbb{E}\sup_{r\in [0,t]}\left|\int_0^r( u ^{\epsilon},  \textbf{P} f(\tau, u ^{\epsilon}) \mathrm{d} W_\varsigma )_{L^2}\right|^p &\leq C \mathbb{E} \bigg(\int_0^t \| u ^{\epsilon}\|_{L^2}^2  \|\textbf{P} f(r, u ^{\epsilon})\|_{L_2(U;L^2)}^2 \textrm{d}r\bigg)^{p/2}\\
&\leq \frac{1}{2}\mathbb{E}\sup_{r\in [0,t]}\| u ^{\epsilon}(r) \|_{L^2}^{2p} + C(p) \mathbb{E} \int_0^t\left (1+\| u ^{\epsilon}\|_{L^2}^{2p} \right) \textrm{d}r.
\end{split}
\end{equation}
By \eqref{4.5} and \eqref{4.6}, the desired estimate \eqref{3.117} follows from the Gronwall Lemma.
\end{proof}


\begin{lemma} \label{lem4.3}
For any $T>0$, we have
\begin{equation}\label{4.7}
\begin{split}
\mathbb{E} \sup_{t \in [0,T]}\left(1+\| u ^{\epsilon} (t)\|^4_{L^4}\right)\leq C\exp\{\exp\{ \exp\{C T\}\}\} ,
\end{split}
\end{equation}
where the positive constant $C$ is independent of $\epsilon$.
\end{lemma}

\begin{proof}
Applying the It\^{o} formula  pointiwse in $x$ and the stochastic Fubini theorem (cf.  {\cite[Theorem 4.33]{da2014stochastic}}), we obtain the following $L^p$ version of the It\^{o} Lemma (cf.  {\cite[Theorem 2.1]{kry2010}})
\begin{equation*}
\begin{split}
\textrm{d} \| u ^{\epsilon}(r)\|^4_{L^4}= &  - 4\int_{\mathbb{R}^2} | u ^{\epsilon} |^2 u ^{\epsilon}\textbf{P} ( u ^{\epsilon}\cdot \nabla)  u ^{\epsilon}\textrm{d}x\textrm{d}r  -4 \int_{\mathbb{R}^2}| u ^{\epsilon} |^2  u ^{\epsilon} \cdot \textbf{P}(-\Delta)^\alpha  u ^{\epsilon}\textrm{d}x\textrm{d}r\\
&+ 4 \int_{\mathbb{R}^2}| u ^{\epsilon} |^2  u ^{\epsilon} \cdot\textbf{P}(n^{\epsilon}\nabla \phi)*\rho^{\epsilon}\textrm{d}x\textrm{d}r + 6 \int_{\mathbb{R}^2}| u ^{\epsilon}(s)|^2 u ^{\epsilon} \cdot f(s, u ^{\epsilon} )  \textrm{d}x\textrm{d}r\\
&+ 4 \int_{\mathbb{R}^2}| u ^{\epsilon} |^2  u ^{\epsilon} \cdot\textbf{P} f(r, u ^{\epsilon})\textrm{d}x \mathrm{d} W_r\\
:= &(\mathcal {I} _1 +\mathcal {I} _2+\mathcal {I} _3+\mathcal {I} _4) \textrm{d}r+\mathcal {I} _5 \textrm{d} W_r.
\end{split}
\end{equation*}
Integrating by parts, we have
\begin{equation*}
\begin{split}
\mathcal {I} _1= \int_{\mathbb{R}^2}  u ^{\epsilon} \cdot \nabla | u ^{\epsilon} |^4\textrm{d}x= \int_{\mathbb{R}^2}  | u ^{\epsilon} |^4 \textrm{div}  u ^{\epsilon}\textrm{d}x=0.
\end{split}
\end{equation*}
By virtue of the generalized positive estimate  {for $(-\Delta)^\frac{\alpha}{2}$ (cf. \cite[Proposition 5.5]{miao2012littlewood})}, we find
\begin{equation*}
\begin{split}
-\mathcal {I} _2 \geq 4 \int_{\mathbb{R}^2}\left|(-\Delta)^\frac{\alpha}{2} | u ^{\epsilon}|^2\right|^2\textrm{d}x  = 4 \left\|| u ^{\epsilon}|^2\right\|_{\dot{H}^\alpha}^2\geq 0.
\end{split}
\end{equation*}
For $\mathcal {I} _3$, due to the embedding $\dot{H}^\alpha(\mathbb{R}^2) \subset L^4(\mathbb{R}^2) $ for $ \alpha \in [\frac{1}{2},1] $, we have
\begin{equation*}
\begin{split}
\mathcal {I} _3 
&\leq C\left\|| u ^{\epsilon} |^2\|_{\dot{H}^\alpha} \right\|  u ^{\epsilon} \|_{L^4}\|\textbf{P}(n^{\epsilon}\nabla \phi)*\rho^{\epsilon}\|_{L^{2}}\\
&\leq  \frac{1}{2}\left\|| u ^{\epsilon} |^2\right\|_{\dot{H}^\alpha}^2+ C( \phi)\|n^{\epsilon}\|_{L^{2}}^2\left (1+ \| u ^{\epsilon} \|_{L^4}^4\right).
\end{split}
\end{equation*}
For $\mathcal {I} _4$, we get from the assumption of $ f $ that
\begin{equation*}
\begin{split}
\mathcal {I} _4  &\leq  C \sum_{k\geq1}\| u ^{\epsilon} \|_{L^4}^2 \| u ^{\epsilon}\cdot f(r, u ^{\epsilon} )e_k \|_{L^2} \leq C \left(1+\| u ^{\epsilon} \|_{L^4}^4\right).
\end{split}
\end{equation*}
Collecting the above estimates and using the Gronwall Lemma yield that
\begin{equation}\label{4.13}
\begin{split}
& \exp\left\{-C(\phi) \int_0^t \|n^{\epsilon}\|_{L^{2}}^2\textrm{d}r\right\}(1+\| u ^{\epsilon} \|^4_{L^4})\\
&\quad \leq 1+\| u ^{\epsilon}_0 \|^4_{L^4}+4  \int_0^t\exp\left\{-C(\phi) \int_0^s \|n^{\epsilon}\|_{L^{2}}^2\textrm{d}r\right\}\int_{\mathbb{R}^2}| u ^{\epsilon} |^2  u ^{\epsilon} \cdot\textbf{P} f(r, u ^{\epsilon}) \textrm{d}x \mathrm{d} W _r.
\end{split}
\end{equation}
By applying the BDG inequality, we have
\begin{equation}\label{4.14}
\begin{split}
&\mathbb{E}\sup_{t \in [0,T]}\left(1+\| u ^{\epsilon} \|^4_{L^4}\right)\\
&\quad \leq C\exp\{  \exp\{ C T\}\} \Bigg(1+\| u ^{\epsilon}_0 \|^4_{L^4}+ \mathbb{E} \sup_{t \in [0,T]}\left|\int_0^t\exp\{ \exp\{ C r\}\}\int_{\mathbb{R}^2}| u ^{\epsilon} |^2  u ^{\epsilon} \cdot\textbf{P} f(r, u ^{\epsilon}) \textrm{d}x \mathrm{d} W_r \right|\Bigg) \\
&\quad\leq C\exp\{  \exp\{ C T\}\} \left(1+\| u ^{\epsilon}_0 \|^4_{L^4}+ \mathbb{E} \left(\sum_{k\geq1}\int_0^T\left(\int_{\mathbb{R}^2}| u ^{\epsilon} |^2  u ^{\epsilon} \cdot\textbf{P} f(r, u ^{\epsilon})e_k\textrm{d}x \right)^2\textrm{d}r\right)^{\frac{1}{2}}\right) \\
&\quad\leq \frac{1}{2}\mathbb{E}\sup_{t\in [0,T]}\left(1+\| u ^{\epsilon} \|^4_{L^4}\right) +      C \exp\{ \exp\{  CT\}\}\left(1 + \mathbb{E} \int_0^T (1+\| u ^{\epsilon} \|^4_{L^4})  \textrm{d}r  \right).
\end{split}
\end{equation}
Absorbing the first term on the R.H.S. of \eqref{4.14} and using the Gronwall Lemma, we obtain the desired estimate \eqref{4.7}.
\end{proof}

Based on the Lemma \ref{lem4.3}, one can derive the following key energy inequality for $c^\epsilon$.

\begin{lemma} \label{lem4.4}
Let $T>0$. Then we have
\begin{equation*}
\begin{split}
 &  \|\nabla \sqrt{c^{\epsilon}} (t) \|_{L^2} ^2+  \int_0^t\|\Delta \sqrt{c^{\epsilon}}\|_{L^2}^2\textrm{d}r + \int_0^t\int_{\mathbb{R}^2} \frac{|\nabla\sqrt{c^{\epsilon}}|^4}{  c^{\epsilon}}\textrm{d}x\textrm{d}r \\
 &\quad \leq C(\|n_0\|_{L^1},\|c_0\|_{L^\infty}, T)  \left(1+\sup_{r \in [0,t]}\|  u ^\epsilon(r)\|_{L^4}^4 \right),
\end{split}
\end{equation*}
for all $t \in [0,T]$, $\mathbb{P}$-a.s. Moreover, there exists a constant $C>0$ independent of $\epsilon$ such that
\begin{equation*}
\begin{split}
 & \mathbb{E}\left(\sup_{t\in[0,T]}\|\nabla \sqrt{c^{\epsilon}} (t) \|_{L^2} ^2+ \int_0^T\|\Delta \sqrt{c^{\epsilon}}\|_{L^2}^2\textrm{d}t + \int_0^T\int_{\mathbb{R}^2} \frac{|\nabla\sqrt{c^{\epsilon}}|^4}{  c^{\epsilon}}\textrm{d}x\textrm{d}t\right) \leq  C \exp\{\exp\{\exp\{CT\} \}\}.
\end{split}
\end{equation*}
\end{lemma}

\begin{proof}
Consider $
h(x)=2\sqrt{x}$ and $g(x)=-\frac{1}{2\sqrt{x}}$, for any $ x>0$.
We apply the chain rule to  {$ h(c^\epsilon(t))$} to get
\begin{equation}\label{4.15}
\begin{split}
\textrm{d} \emph{h} (c^{\epsilon}) = \left(- u ^{\epsilon}\cdot \nabla \emph{h}(c^{\epsilon})  + \Delta \emph{h}(c^{\epsilon})-\emph{h} ''(c^{\epsilon}) |\nabla c^{\epsilon}|^2 - \emph{h}'(c^{\epsilon}) c^{\epsilon}(n^{\epsilon}*\rho^{\epsilon})\right)\textrm{d}t,
\end{split}
\end{equation}
where we used
$
\Delta \emph{h}(c^{\epsilon})=\emph{h}''(c^{\epsilon}) |\nabla c^{\epsilon}|^2+\emph{h}'(c^{\epsilon})\Delta c^{\epsilon}.
$
By \eqref{4.15}, we have
\begin{equation} \label{4.16}
\begin{split}
&\frac{1}{2}\textrm{d} \|\nabla\emph{h} (c^{\epsilon}) \|_{L^2} ^2+ \|\Delta \emph{h} (c^{\epsilon})\|_{L^2}^2\textrm{d}t \\
&\quad = \int_{\mathbb{R}^2} \left( u ^{\epsilon}\cdot \nabla \emph{h}(c^{\epsilon})\right)  \Delta \emph{h} (c^{\epsilon})\textrm{d}x\textrm{d}t+\int_{\mathbb{R}^2} \emph{h}'(c^{\epsilon}) c^{\epsilon}(n^{\epsilon}*\rho^{\epsilon})  \Delta \emph{h} (c^{\epsilon})\textrm{d}x\textrm{d}t\\
&\quad\quad+\int_{\mathbb{R}^2} g(c^\epsilon)  |\nabla \emph{h} (c^{\epsilon})|^2  \Delta \emph{h} (c^{\epsilon})\textrm{d}x\textrm{d}t\\
&\quad:= \left(\mathcal {J}_1+\mathcal {J}_2+\mathcal {J}_3 \right)\textrm{d}t.
\end{split}
\end{equation}
For $\mathcal {J}_1$, we get by the Ladyzhenskaya  inequality (cf. \cite[Lemma 2 in Chapter 1]{ladyzhenskaya1969mathematical} that
\begin{equation}\label{4.17}
\begin{split}
\mathcal {J}_1
&\leq \delta_1 \|\Delta \emph{h} (c^{\epsilon})\|_{L^2}^2+C(\delta_1, \|c_0\|_{L^\infty})\| u ^\epsilon\|_{L^4}^2 \left\| \sqrt {|g(c^\epsilon)|}\nabla \emph{h}(c^{\epsilon}) \right\|_{L^4}^2\\
&\leq \delta_1 \|\Delta \emph{h} (c^{\epsilon})\|_{L^2}^2+\delta_2 \left\|\sqrt {|g(c^\epsilon)|} \nabla \emph{h}(c^{\epsilon}) \right\|_{L^4}^4+   C(\delta_1,\delta_2, \|c_0\|_{L^\infty})  \|  u ^\epsilon\|_{L^4}^4,
\end{split}
\end{equation}
where $ \delta_1,\delta_2>0$ will be determined later. For $\mathcal {J}_2$, we get by integrating by parts that
\begin{equation}\label{4.18}
\begin{split}
\mathcal {J}_2=&-\int_{\mathbb{R}^2} \frac{1}{\emph{h}' (c^{\epsilon})}\frac{\textrm{d}}{\textrm{d} c^{\epsilon}}(\emph{h}'(c^{\epsilon}) c^{\epsilon})(n^{\epsilon}*\rho^{\epsilon}) |\nabla \emph{h} (c^{\epsilon})|^2\textrm{d}x\\
& - \int_{\mathbb{R}^2}  \emph{h}'(c^{\epsilon}) c^{\epsilon} \nabla(n^{\epsilon}*\rho^{\epsilon})  \cdot \nabla \emph{h} (c^{\epsilon})\textrm{d}x \\
\leq& -\int_{\mathbb{R}^2}  \emph{h}'(c^{\epsilon}) c^{\epsilon} \nabla(n^{\epsilon}*\rho^{\epsilon})  \cdot \nabla \emph{h} (c^{\epsilon})\textrm{d}x,
\end{split}
\end{equation}
where we use the facts of $n^\epsilon \geq0$ and
$
- \frac{1}{\emph{h}' (c^{\epsilon})}\frac{\textrm{d}}{\textrm{d} c^{\epsilon}}(\emph{h}'(c^{\epsilon}) c^{\epsilon}) = -\frac{1}{2}<0.
$
For $\mathcal {J}_3$, the direct calculation shows that
\begin{equation*}
\begin{split}
 \mathcal {J}_3
 =& -\sum_{i,j=1}^2 \int_{\mathbb{R}^2} \left(g'(c^\epsilon) \partial_j c^\epsilon (\partial_i \emph{h} (c^{\epsilon}))^2+2g(c^\epsilon)\partial_i \emph{h} (c^{\epsilon})\partial_i \partial_j\emph{h} (c^{\epsilon}) \right) \partial_j \emph{h} (c^{\epsilon})\textrm{d}x\\
=&\underbrace{- 2\sum_{i=1}^2   \int_{\mathbb{R}^2} g(c^\epsilon)(\partial_i \emph{h} (c^{\epsilon}))^2\partial_{i}^2 \emph{h} (c^{\epsilon})\textrm{d}x}_{ {  {:= \textbf{A}}}}-2\sum_{i\neq j }  \int_{\mathbb{R}^2}g(c^\epsilon) \partial_i \emph{h} (c^{\epsilon})\partial_j  \emph{h} (c^{\epsilon})\partial_{i }\partial_{j } \emph{h} (c^{\epsilon})\textrm{d}x
\\
& -\sum_{i,j=1 }^2 \int_{\mathbb{R}^2}  (g(c^\epsilon))^2 (\partial_i \emph{h} (c^{\epsilon}))^2(\partial_j  \emph{h} (c^{\epsilon}))^2 \frac{g'(c^\epsilon)}{(g(c^\epsilon))^2h'(c^\epsilon)}\textrm{d}x .
\end{split}
\end{equation*}
From the definition of $\mathcal {J}_3$, we observe that
\begin{equation*}
\begin{split}
  {\textbf{A}}= -2\mathcal {J}_3 +2\sum_{i\neq j}  \int_{\mathbb{R}^2}g(c^\epsilon)(\partial_i \emph{h} (c^{\epsilon}))^2\partial_{ j}^2 \emph{h} (c^{\epsilon})\textrm{d}x,
\end{split}
\end{equation*}
which combined with the last equality leads to
\begin{equation}\label{4.19}
\begin{split}
  \mathcal {J}_3=& \frac{2}{3}\sum_{i\neq j}  \int_{\mathbb{R}^2}g(c^\epsilon)(\partial_i \emph{h} (c^{\epsilon}))^2\partial_{ j}^2 \emph{h} (c^{\epsilon})\textrm{d}x-\frac{2}{3}\sum_{i\neq j }  \int_{\mathbb{R}^2}g(c^\epsilon) \partial_i \emph{h} (c^{\epsilon})\partial_j  \emph{h} (c^{\epsilon})\partial_{i }\partial_{j } \emph{h} (c^{\epsilon})\textrm{d}x  \\
& -\frac{1}{3}\sum_{i,j=1}^2 \int_{\mathbb{R}^2} (g(c^\epsilon))^2 (\partial_i \emph{h} (c^{\epsilon}))^2(\partial_j  \emph{h} (c^{\epsilon}))^2 \textrm{d}x\\
:= & \frac{2}{3}\mathcal {K}_1+ \frac{2}{3}\mathcal {K}_2+ \mathcal {K}_3.
\end{split}
\end{equation}
For $\mathcal {K}_1$, it follows from the Young inequality that
\begin{equation*}
\begin{split}
\mathcal {K}_1&=  \int_{\mathbb{R}^2}\left(g(c^\epsilon)|\partial_1 \emph{h} (c^{\epsilon})|^2\partial_{ 2}^2 \emph{h} (c^{\epsilon})+g(c^\epsilon)|\partial_2 \emph{h} (c^{\epsilon})|^2\partial_{1}^2 \emph{h} (c^{\epsilon})\right)\textrm{d}x\\
&\leq \int_{\mathbb{R}^2} \left(\frac{1}{4} (g(c^\epsilon))^2(|\partial_1 \emph{h} (c^{\epsilon})|^4+|\partial_2 \emph{h} (c^{\epsilon})|^4) + (|\partial_{ 2}^2 \emph{h} (c^{\epsilon})|^2+ |\partial_{1}^2 \emph{h} (c^{\epsilon})|^2)\right)\textrm{d}x\\
&\leq \frac{1}{4}\sum_{i=1}^2\int_{\mathbb{R}^2} (g(c^\epsilon))^2  |\partial_i \emph{h} (c^{\epsilon})|^4  \textrm{d}x + \sum_{i=1}^2\int_{\mathbb{R}^2}|\partial_{i}^2 \emph{h} (c^{\epsilon})|^2 \textrm{d}x .
\end{split}
\end{equation*}
For $\mathcal {K}_2$, we have
\begin{equation*}
\begin{split}
\mathcal {K}_2 &\leq\frac{1}{4}\sum_{i\neq j }  \int_{\mathbb{R}^2}(g(c^\epsilon ))^2 (\partial_i \emph{h} (c^{\epsilon}))^2(\partial_j  \emph{h} (c^{\epsilon}))^2\textrm{d}x+ \sum_{i\neq j }  \int_{\mathbb{R}^2}(\partial_{i }\partial_{j } \emph{h} (c^{\epsilon}))^2\textrm{d}x.
\end{split}
\end{equation*}
Substituting the estimates for $\mathcal {K}_1$ and $\mathcal {K}_2$ into \eqref{4.19}, we get
\begin{equation}\label{4.20}
\begin{split}
\mathcal {J}_3&\leq-\frac{1}{6}\sum_{i, j=1}^2\int_{\mathbb{R}^2}g^2(c^\epsilon) (\partial_i \emph{h} (c^{\epsilon}))^2(\partial_j  \emph{h} (c^{\epsilon}))^2\textrm{d}x  + \frac{2}{3}\sum_{i, j=1}^2  \int_{\mathbb{R}^2}(\partial_{i }\partial_{j } \emph{h} (c^{\epsilon}))^2\textrm{d}x.
\end{split}
\end{equation}
Note that
$$
\sum_{i, j=1}^2 (\partial_i \emph{h} (c^{\epsilon}))^2(\partial_j  \emph{h} (c^{\epsilon}))^2= |\nabla \emph{h} (c^{\epsilon})|^4,
$$
and
$$
\|\Delta f\|_{L^2}^2=\int_{\mathbb{R}^2}\left((\partial_1^2f)^2+(\partial_2^2f)^2+2\partial_1^2f\partial_2^2f\right)\textrm{d}x=\|\nabla^2 f\|_{L^2}^2.
$$
By choosing $\delta_1=\frac{1}{4}$, it follows from \eqref{4.15}, \eqref{4.17}, \eqref{4.18} and \eqref{4.20}  that
\begin{equation*}
\begin{split}
&\textrm{d} \|\nabla \sqrt{c^{\epsilon}}  \|_{L^2} ^2+ \frac{1}{6}\|\Delta \sqrt{c^{\epsilon}}\|_{L^2}^2\textrm{d}t \\
&\quad  \leq C( \delta_2,\|c_0\|_{L^\infty})  \|  u ^\epsilon\|_{L^4}^4-\int_{\mathbb{R}^2}   \frac{2}{\sqrt{c^{\epsilon}}}\nabla(n^{\epsilon}*\rho^{\epsilon})  \cdot \nabla \sqrt{c^{\epsilon}}\textrm{d}x+\int_{\mathbb{R}^2} \left(\delta_2-\frac{1}{6}\right) \frac{|\nabla\sqrt{ c^{\epsilon}}|^4}{  c^{\epsilon}}  \textrm{d}x.
\end{split}
\end{equation*}
By taking $\delta_2=\frac{1}{12}$, we further obtain
\begin{equation} \label{4.21}
\begin{split}
 &\textrm{d} \|\nabla \sqrt{c^{\epsilon}}  \|_{L^2} ^2+ \frac{1}{6}\|\Delta \sqrt{c^{\epsilon}}\|_{L^2}^2\textrm{d}t +\frac{1}{12}\int_{\mathbb{R}^2} \frac{|\nabla\sqrt{c^{\epsilon}}|^4}{  c^{\epsilon}}\textrm{d}x\textrm{d}t \\
 &\quad \leq C(\|c_0\|_{L^\infty})  \|  u ^\epsilon\|_{L^4}^4 \underbrace{-\int_{\mathbb{R}^2}   \frac{2}{\sqrt{c^{\epsilon}}}\nabla(n^{\epsilon}*\rho^{\epsilon})  \cdot \nabla \sqrt{c^{\epsilon}}\textrm{d}x}_{ {:=\textbf{B}}}.
\end{split}
\end{equation}
By integrating by parts,  the term $ {\textbf{B}}$ can be estimated as
\begin{equation} \label{4.22}
\begin{split}
 {| {\textbf{B}}|}&=\left|2\int_{\mathbb{R}^2}(n^{\epsilon}*\rho^{\epsilon})\left(2\sqrt{c^{\epsilon} } \Delta \sqrt{c^{\epsilon}}+  2| \nabla  \sqrt{c^{\epsilon}}|^2\right) \textrm{d}x\right| \\
& \leq 4 \|n^{\epsilon}*\rho^{\epsilon}\|_{L^2}\left(\|\sqrt{c^{\epsilon}}\|_{L^\infty}\|\Delta \sqrt{c^{\epsilon}}\|_{L^2}+  \|\sqrt{c^{\epsilon}}\|_{L^\infty} \left\|  \frac{\nabla\sqrt{c^{\epsilon}}}{\sqrt[4]{c^{\epsilon}}}\right \|_{L^4}^2\right) \\
& \leq \frac{1}{12}\|\Delta \sqrt{c^{\epsilon}}\|_{L^2}^2 +\frac{1}{24} \left\| \frac{\nabla\sqrt{c^{\epsilon}}}{\sqrt[4]{c^{\epsilon}}}\right \|_{L^4}^4 + C(\|c_0\|_{L^\infty}) \|n^{\epsilon} \|_{L^2}^2.
\end{split}
\end{equation}
According to the Biot-Savart law (cf. \cite [Proposition 2.16]{majda2002vorticity}) and the boundedness of singular integral operator in $L^p$ spaces (cf. \cite[Proposition 4 of Chapter VI]{stein1993harmonic},  ), there holds
$$
\|\nabla u ^\epsilon\|_{L^p}\leq C \| v ^\epsilon\|_{L^p},~~  \textrm{for all}~1<p<\infty.
$$
Thanks to the Sobolev embedding
$$
H^\alpha(\mathbb{R}^2)\hookrightarrow L^4(\mathbb{R}^2), ~~\forall\alpha\geq \frac{1}{2} ~~~\textrm{and}~~~ \dot{W}^{1,p}(\mathbb{R}^2)\hookrightarrow L^4(\mathbb{R}^2), ~~\forall p\geq 2.
$$
We deduce from \eqref{4.21} and \eqref{4.22} that
\begin{equation} \label{4.23}
\begin{split}
 &  \|\nabla \sqrt{c^{\epsilon}} (t) \|_{L^2} ^2+ \frac{1}{12}\int_0^t\|\Delta \sqrt{c^{\epsilon}}\|_{L^2}^2\textrm{d}r +\frac{1}{24}\int_0^t\int_{\mathbb{R}^2} \frac{|\nabla\sqrt{c^{\epsilon}}|^4}{  c^{\epsilon}}\textrm{d}x\textrm{d}r \\
 &\quad \leq C(\|n_0\|_{L^1},\|c_0\|_{L^\infty}, T)  \left(1+\sup_{r \in [0,t]}\|  u ^\epsilon(r)\|_{L^4}^4 \right)  ,~~\mathbb{P}\textrm{-a.s.}
\end{split}
\end{equation}
By taking the expectation on both sides of \eqref{4.23} and utilizing  the estimates \eqref{4.3} and \eqref{4.7}, we obtain the desired inequality. The proof is finished.
\end{proof}

Now, one could establish the following entropy-type estimate associated to the component $n^\epsilon$.
\begin{lemma}\label{lem4.5}
For any $T>0$, there holds
\begin{equation*}
\begin{split}
\mathbb{E} \sup_{t\in [0,T]} \left(\||x| n^{\epsilon}(t)\|_{L^1}+\int_{\mathbb{R}^2} n^\epsilon \ln n^\epsilon \textrm{d}x \right) + \mathbb{E}\int_0^T\|\nabla \sqrt{n^{\epsilon}}\|^2_{ L^2}\textrm{d}t \leq C \exp \{\exp\{\exp\{CT\}\}\} ,
\end{split}
\end{equation*}
where $C>0$ is independent of $\epsilon$.
\end{lemma}

\begin{proof}
Applying the chain rule to  {$ (n^\epsilon \ln n^\epsilon )(t)$} with respect to $\eqref{3.1}_1$ and integrating the resulting identity on $\mathbb{R}^2$, we infer that
\begin{equation}\label{4.24}
\begin{split}
&\textrm{d} \int_{\mathbb{R}^2} n^\epsilon \ln n^\epsilon \textrm{d}x+\left(  \| n^\epsilon \|_{L^2}^2 + 4 \|\nabla \sqrt{n^{\epsilon}}\|^2_{L^2}\right) \textrm{d}t\\
&\quad =  \int _{\mathbb{R}^2} n^\epsilon \ln n^\epsilon\textrm{d}x\textrm{d}t+\int _{\mathbb{R}^2} \nabla n^{\epsilon}\cdot(\nabla c^{\epsilon}*\rho^{\epsilon}) \textrm{d}x\textrm{d}t
 -  \int_{\mathbb{R}^2}   (n^{\epsilon} )^2 \ln n^{\epsilon}  \textrm{d}x\textrm{d}t.
\end{split}
\end{equation}
By integrating by parts, we get
\begin{equation*}
\begin{split}
\int _{\mathbb{R}^2} \nabla n^{\epsilon} (\nabla c^{\epsilon}*\rho^{\epsilon})   \textrm{d}x& = 2\int _{\mathbb{R}^2}   n^{\epsilon} \Big( |\nabla \sqrt{c^{\epsilon}}|^2  *\rho^{\epsilon}  +  ( \sqrt{c^{\epsilon}}\Delta \sqrt{c^{\epsilon}} )*\rho^{\epsilon}\Big)  \textrm{d}x \\
&\leq  C(\|c _0\|_{L^\infty}) \left(\| \Delta \sqrt{c^{\epsilon}}\|_{L^2}^2+   \bigg\| \frac{\nabla \sqrt{c^{\epsilon}} }{\sqrt[4]{c^{\epsilon} }} \bigg\|_{L^4}^4+ \| n^{\epsilon}\|_{L^2}^2\right).
\end{split}
\end{equation*}
Since $ x \ln \frac{1}{x}< 1$ for all $x>0$, we have
\begin{equation*}
\begin{split}
-  \int_{\mathbb{R}^2}   (n^{\epsilon} )^2 \ln n^{\epsilon}  \textrm{d}x &=  \int_{\{0 < n^{\epsilon} <e^{-|x|}\}\cup \{1\geq n^{\epsilon}\geq e^{-|x|}\}}   (n^{\epsilon} )^2 \ln \frac{1}{n^{\epsilon}}\textrm{d}x \leq C+   \|\sqrt{|x|}n^{\epsilon}\|_{L^2}^2.
\end{split}
\end{equation*}
Substituting the last two estimates into \eqref{4.24}, we gain from \eqref{4.23} that
\begin{equation}\label{4.244}
\begin{split}
& \sup_{t\in[0,T]}\int_{\mathbb{R}^2} n^\epsilon \ln n^\epsilon \textrm{d}x+   \int_0^T \left(\| n^\epsilon \|_{ L^2}^2  + 4 \|\nabla \sqrt{n^{\epsilon}}\|^2_{ L^2}\right)\textrm{d}t \\
&\quad \leq C \exp\{CT\} +\int_{\mathbb{R}^2} n^\epsilon_0 \ln n^\epsilon_0 \textrm{d}x+    \int_0^T\int _{\mathbb{R}^2} n^\epsilon \ln n^\epsilon\textrm{d}x\textrm{d}t+   \int_0^T\|\sqrt{|x|}n^{\epsilon}\|_{ L^2}^2\textrm{d}t\\
&\quad\quad+ C \int_0^T\left(\| \Delta \sqrt{c^{\epsilon}}\|_{L^2}^2+   \bigg\| \frac{\nabla \sqrt{c^{\epsilon}} }{\sqrt[4]{c^{\epsilon} }} \bigg\|_{L^4}^4 \right)\textrm{d}t\\
 &\quad\leq \int_{\mathbb{R}^2} n^\epsilon_0 \ln n^\epsilon_0 \textrm{d}x+    \int_0^T\int _{\mathbb{R}^2} n^\epsilon \ln n^\epsilon\textrm{d}x\textrm{d}t \\
 &\quad\quad+   \int_0^T\|\sqrt{|x|}n^{\epsilon}\|_{ L^2}^2\textrm{d}t+   C \exp\{CT\}\left(1+ \sup_{t\in [0,T]} \|  u ^\epsilon(t)\|_{L^4}^4\right).
\end{split}
\end{equation}
To deal with the term $\int_0^T\|\sqrt{|x|}n^{\epsilon}\|_{ L^2}^2\textrm{d}t$, we consider a smooth function
$
\varphi_\eta (x)=\sqrt{|x|^2+\eta}$, $\eta >0 .
$
Then we apply the chain rule to $ (\varphi_\eta n^\epsilon) (t) $ to obtain
\begin{equation}\label{4.25}
\begin{split}
&\int_{\mathbb{R}^2} \varphi_\eta n^{\epsilon} \textrm{d}x  + \int_0^t\int_{\mathbb{R}^2}\varphi_\eta  (n^{\epsilon} )^2 \textrm{d}x \textrm{d}r\\
&\quad =  \int_{\mathbb{R}^2}  \varphi_\eta n^{\epsilon}_0 \textrm{d}x+\int_0^t\int_{\mathbb{R}^2}n^{\epsilon}( u ^{\epsilon}\cdot \nabla \varphi_\eta)\textrm{d}x  \textrm{d}r+ \int_0^t\int_{\mathbb{R}^2}\varphi_\eta  n^{\epsilon} \textrm{d}x \textrm{d}r \\
&\quad\quad+ \int_0^t\int_{\mathbb{R}^2}n^{\epsilon}\nabla\varphi_\eta\cdot(\nabla c^{\epsilon} *\rho^{\epsilon} ) \textrm{d}x \textrm{d}r-\int_0^t\int_{\mathbb{R}^2}\nabla\varphi_\eta \cdot \nabla n^{\epsilon}\textrm{d}x  \textrm{d}r.
\end{split}
\end{equation}
Note that $|\nabla \varphi_\eta|\leq 1$, it follows from the Lemma \ref{lem4.1} and Young inequality that
\begin{equation*}
\begin{split}
\int_0^t\int_{\mathbb{R}^2}n^{\epsilon}( u ^{\epsilon}\cdot \nabla \varphi_\eta)\textrm{d}x  \textrm{d}r
 \leq C \left (1+t \left(\int_0^t\|n^\epsilon\|_{ L^2}^2\textrm{d}r\right)^{\frac{1}{2}}+  \sup_{r\in [0,t]}\| u ^\epsilon(r)\|_{L^4}^4 \right),\\
\end{split}
\end{equation*}
\begin{equation*}
\begin{split}
\int_0^t\int_{\mathbb{R}^2}n^{\epsilon}\nabla\varphi_\eta\cdot( \nabla c^{\epsilon} *\rho^{\epsilon}) \textrm{d}x \textrm{d}r
\leq C \sqrt[4]{t} \left(\int_0^t\|n^\epsilon\|_{ L^2}^2\textrm{d}r\right)^{\frac{1}{4}} \left(\int_0^t\bigg\| \frac{\nabla \sqrt{c^{\epsilon}} }{\sqrt[4]{c^{\epsilon} }}\bigg\|_{ L^4}^4\textrm{d}r\right)^{\frac{1}{2}},\\
\end{split}
\end{equation*}
and
\begin{equation*}
\begin{split}
\int_0^t\int_{\mathbb{R}^2}\nabla\varphi_\eta \cdot \nabla n^{\epsilon}\textrm{d}x  \textrm{d}r  
\leq C t  \sup_{r\in [0,t]}\|n^\epsilon(r)\|_{ L^1} \left(\int_0^t\|\nabla \sqrt{n^{\epsilon}}\|_{ L^2}^2\textrm{d}r\right)^{\frac{1}{2}}.
\end{split}
\end{equation*}
By using the facts of $n^\epsilon\geq0$ and $\varphi_\eta\searrow|x|$ as $\eta\rightarrow 0$, we get by the Monotone Convergence Theorem that
$
 \int_{\mathbb{R}^2}  \varphi_\eta n^{\epsilon}_0 \textrm{d}x \rightarrow \int_{\mathbb{R}^2} |x| n^{\epsilon}_0 \textrm{d}x,
$
and
$
\int_0^t\int_{\mathbb{R}^2}\varphi_\eta  n^{\epsilon} \textrm{d}x \textrm{d}r \rightarrow  \int_0^t\int_{\mathbb{R}^2}|x|  n^{\epsilon} \textrm{d}x \textrm{d}r$, as $\eta\rightarrow 0 .
$
Taking the limit as $\eta \rightarrow0$ in \eqref{4.25}, we are able to derive that
\begin{equation}\label{4.26}
\begin{split}
&  \sup_{t\in[0,T]}\||x| n^{\epsilon}(t)\|_{L^1}  +   \int_0^T\|\sqrt{|x|}n^{\epsilon}\|_{ L^2}^2 \textrm{d}t\\
&\quad\leq  \| |x| n^{\epsilon}_0 \|_{L^1}+C \sqrt{T} \|n^\epsilon\|_{L^2_TL^2}^{\frac{1}{2}} \sup_{t\in [0,T]}\| u ^\epsilon(t)\|_{L^2} +C  \sqrt{T} \|n^\epsilon\|_{L^2_TL^2} +C T^2  \|n^\epsilon\|_{L^\infty_TL^1}^2\\
&\quad\quad +\bigg\| \frac{\nabla \sqrt{c^{\epsilon}} }{\sqrt[4]{c^{\epsilon} }}\bigg\|_{L^4_tL^4}^4+ 3\|\nabla \sqrt{n^{\epsilon}}\|_{L^2_TL^2}^2+  \int_0^T\||x|  n^{\epsilon} \|_{L^1} \textrm{d}r\\
&\quad\leq  \| |x| n^{\epsilon}_0 \|_{L^1} +C \exp\{CT\}\left(1+ \sup_{t\in [0,T]} \|  u ^\epsilon(t)\|_{L^4}^4\right)+  \int_0^T\||x|  n^{\epsilon} \|_{L^1} \textrm{d}t.
\end{split}
\end{equation}
Combining the inequalities \eqref{4.244} and \eqref{4.26}, we find
\begin{equation}\label{4.27}
\begin{split}
& \mathbb{E} \sup_{t\in [0,T]} \left(\||x| n^{\epsilon}(t)\|_{L^1}+\int_{\mathbb{R}^2} n^\epsilon \ln n^\epsilon \textrm{d}x \right) +  \mathbb{E}\int_0^T \|\nabla \sqrt{n^{\epsilon}}\|^2_{ L^2}\textrm{d}t\\
&\quad\leq \| |x| n^{\epsilon}_0 \|_{L^1}+\int_{\mathbb{R}^2} n^\epsilon_0 \ln n^\epsilon_0 \textrm{d}x+C \exp\{CT\}\mathbb{E}\left(1+ \sup_{t\in [0,T]} \|  u ^\epsilon(t)\|_{L^4}^4\right)  \\
&\quad\quad  +    \mathbb{E}\int_0^T\bigg(\||x| n^{\epsilon}\|_{L^1}+\int_{\mathbb{R}^2} n^\epsilon \ln n^\epsilon \textrm{d}x \bigg) \textrm{d}t .
\end{split}
\end{equation}
By applying the Gronwall Lemma to  \eqref{4.27}  and using
$
\||x| n_0^{\epsilon}\|_{L^1}+\int_{\mathbb{R}^2} n_0^\epsilon \ln n^\epsilon_0 \textrm{d}x \leq \left\| \sqrt{1+|x|^2} n_0^{\epsilon}\right\|_{L^1}+ \|n_0^\epsilon\|_{L^2}^2\leq C,
$
we get
\begin{equation*}
\begin{split}
 \mathbb{E} \sup_{t\in [0,T]} \left(\||x| n^{\epsilon}(t)\|_{L^1}+\int_{\mathbb{R}^2} n^\epsilon \ln n^\epsilon \textrm{d}x \right) \leq   C\exp\{\exp\{ \exp\{C T\}\}\}.
\end{split}
\end{equation*}
Inserting the last inequality into \eqref{4.27} in turn implies the desired inequality.
\end{proof}


Let $T>0$. For any $N>1$, define
$$
\Omega_{N}^\epsilon:= \left\{\omega \in \Omega;~\int_0^T\| \Delta \sqrt{c^\epsilon}\|_{ L^2}^2\textrm{d}t\vee\int_0^T\bigg\|\frac{\nabla \sqrt{c^\epsilon}}{\sqrt[4]{c^\epsilon}} \bigg\|_{ L^4}^4\textrm{d}t\vee\sup_{t \in [0,T]}\|  u ^\epsilon(t)\|_{L^4}^4 \leq N\right\}.
$$
By the Lemmas \ref{lem4.3}-\ref{lem4.4} and Chebyshev inequality, we find
\begin{equation}\label{4..25}
\begin{split}
 \mathbb{P}\{\Omega_{N}^\epsilon\}\geq& 1- \mathbb{P}\left\{\int_0^T\| \Delta \sqrt{c^\epsilon}\|_{ L^2}^2\textrm{d}t>N\right\}- \mathbb{P}\left\{\int_0^T\bigg\|\frac{\nabla \sqrt{c^\epsilon}}{\sqrt[4]{c^\epsilon}} \bigg\|_{ L^4}^4\textrm{d}t >N\right\}\\
 &- \mathbb{P}\left\{ \sup_{r \in [0,t]}\|  u ^\epsilon(r)\|_{L^4}^4 {> N}\right\} \geq 1-\frac{C}{N},
\end{split}
\end{equation}
for some constant $C>0$ independent of $\epsilon$. This fact will be applied to verify the tightness of the sequence $\{n^\epsilon\}_{\epsilon > 0}$ later.

\begin{lemma}\label{lem4.6}
For any $T>0$, we have
\begin{equation} \label{4.28}
\begin{split}
 \sup_{t\in[0,T]}\|n^\epsilon(t)\|_{ L^2 }^2+ \int_0^T\|n^\epsilon\|_{  H^1 }^2\textrm{d}t+ \int_0^T\|n^\epsilon\|_{ L^3}^3\textrm{d}t \leq C\exp \{ C(1+T+N) \},
\end{split}
\end{equation}
for all $\omega\in \Omega_{N}^\epsilon$. Moreover, there hold
\begin{equation} \label{4.29}
\begin{split}
\mathbb{E}\sup_{t\in[0,T]}\|c^\epsilon(t)\|_{ H^1 }^2 +\mathbb{E}\int_0^T\|c^\epsilon\|_{ H^2}^2 \textrm{d}t \leq C ,
\end{split}
\end{equation}
\begin{equation} \label{4.30}
\begin{split}\mathbb{E}\sup_{t\in[0,T]}\| v ^\epsilon(t)\|_{ L^2 }^2+\mathbb{E}\sup_{t\in[0,T]}\| v ^\epsilon(t)\|_{ L^{\frac{4}{3}}}^{ \frac{4}{3}}+\mathbb{E}\int_0^T\| v ^\epsilon\|_{ \dot{H}^\alpha }^2\textrm{d}t  \leq C,
\end{split}
\end{equation}
for some positive constant $C $  independent of $\epsilon$.
\end{lemma}

\begin{proof}
Applying the chain rule to $\frac{1}{2} \|n^\epsilon(t)\|_{L^2}^2 $, it follows from $\eqref{3.1}_1$ that
\begin{equation*}
\begin{split}
\|n^\epsilon (t)\|_{L^2}^2 + 2\int_0^t\left(\|\nabla n^\epsilon \|_{ L^2}^2  + \| n^\epsilon \|_{ L^3}^3\right)\textrm{d}r= \|n^\epsilon_0\|_{L^2}^2+2   \int_0^t\| n^\epsilon \|_{ L^2}^2 \textrm{d}r + 2 \int_0^t(n^{\epsilon}(\nabla c^{\epsilon} *\rho^{\epsilon}),\nabla n^\epsilon)_{L^2}  \textrm{d}r.
\end{split}
\end{equation*}
Integrating by parts and using  {the Gagliardo-Nirenberg (GN) inequality (cf. \cite[Lecture II]{nirenberg1959})}, we have
\begin{equation*}
\begin{split}
2 (n^{\epsilon}(\nabla c^{\epsilon} *\rho^{\epsilon}),\nabla n^\epsilon)_{L^2}  
&\leq  \|\Delta c^{\epsilon } *\rho^{\epsilon}\|_{L^2} \| n^\epsilon \|_{L^2} \|\nabla n^\epsilon \|_{L^2}\\
&\leq  \|\nabla n^\epsilon \|_{L^2}^2+C\left(\| \Delta \sqrt{c^\epsilon}\|_{ L^2}^2+\bigg\|\frac{\nabla \sqrt{c^\epsilon}}{\sqrt[4]{c^\epsilon}} \bigg\|_{ L^4}^4 \right) \| n^\epsilon \|_{L^2}^2.
\end{split}
\end{equation*}
By \eqref{4.23} and the Gronwall Lemma, we arrive at
\begin{equation*}
\begin{split}
  \|n^\epsilon (t)\|_{L^2}^2 +  \int_0^t\left(\|\nabla n^\epsilon \|_{ L^2}^2  + \| n^\epsilon \|_{ L^3}^3\right)\textrm{d}r\leq \|n^\epsilon_0\|_{L^2}^2  \exp\bigg\{ t+C \bigg(1+\sup_{r \in [0,t]}\|  u ^\epsilon(r)\|_{L^4}^4  \bigg) \bigg\},~~\mathbb{P}\textrm{-a.s.,}
\end{split}
\end{equation*}
which implies \eqref{4.28} by using the definition of $\Omega_{N}^\epsilon $.

For the $c^\epsilon$-component, we get from (4.1) and the Lemma \ref{lem4.4} that
\begin{equation*}
\begin{split}
\mathbb{E}\sup_{t\in [0,T]}\|\nabla c^\epsilon(t)\|_{L^2}^2
  \leq 2\| c^\epsilon \|_{L^\infty}\mathbb{E}\sup_{t\in [0,T]}\| \nabla  \sqrt{c^\epsilon} (t)\|_{L^2}^2\leq C\exp\{\exp\{\exp\{C T\} \}\}.
\end{split}
\end{equation*}
Since
$
\Delta c^\epsilon= 2 \sqrt{c^\epsilon}\Delta \sqrt{c^\epsilon}+2 |\nabla \sqrt{c^\epsilon}|^2,
$
we have
\begin{equation*}
\begin{split}
\mathbb{E}\int_0^T\|\Delta c^\epsilon\|_{ L^2}^2 \textrm{d}t&\leq C \mathbb{E}\int_0^T\left(\| \Delta \sqrt{c^\epsilon}\|_{ L^2}^2+\bigg\|\frac{\nabla \sqrt{c^\epsilon}}{\sqrt[4]{c^\epsilon}} \bigg\|_{ L^4}^4 \right)\textrm{d}t\leq C \exp\{\exp\{\exp\{CT\} \}\},
\end{split}
\end{equation*}
which together with \eqref{4.1} imply \eqref{4.29}.

Now we apply the It\^{o} formula  to  {$ \| v ^\epsilon(t)\|_{L^2}^2= \|\nabla\wedge u ^\epsilon(t)\|_{L^2}^2 $} and then integrate by parts over $\mathbb{R}^2$,  it follows that
\begin{equation*}
\begin{split}
\| v ^\epsilon(t)\|_{L^2}^2  +2\int_0^t\|(-\Delta)^\frac{\alpha}{2}  v ^{\epsilon}\|_{L^2} ^2\textrm{d}r\leq & \| v ^\epsilon_0\|_{L^2}^2 + C(\phi)\int_0^t\|\nabla n^\epsilon \|_{ L^2}^2 \textrm{d}r+  C \int_0^t (1+\|  v ^{\epsilon} \|_{L^2} ^2)\textrm{d}r \\
 & +2 \int_0^t( v ^\epsilon,\nabla\wedge\textbf{P} f(r, u ^{\epsilon}) \mathrm{d} W_r)_{L^2}.
\end{split}
\end{equation*}
Utilizing the It\^{o} chain rule to $ e^{-Ct}(1+\| v ^\epsilon(t)\|_{L^2}^2) $, we see that
\begin{equation}\label{5.5}
\begin{split}
&\sup_{t\in [0,T]}\left(1+\| v ^\epsilon(t)\|_{L^2}^2\right)+\int_0^Te^{ C(T-t)} \| v ^{\epsilon}\|_{\dot{H}^\alpha} ^2\textrm{d}t\leq e^{CT}\left( 1+\| v ^\epsilon_0\|_{L^2}^2\right)\\
&\quad + C(\phi)  \int_0^T e^{ C(T-r)} \|\nabla n^{\epsilon}\|_{L^2}^2\textrm{d}r +\sup_{t\in [0,T]}\bigg|\int_0^te^{ C(t-r)}( v ^\epsilon,\nabla\wedge\textbf{P} f(r, u ^{\epsilon}) \mathrm{d} W_r)_{L^2}\bigg|.
\end{split}
\end{equation}
Applying the BDG inequality to \eqref{5.5},  we get
\begin{equation*}
\begin{split}
&\mathbb{E}\sup_{t\in [0,T]}\left(1+\| v ^\epsilon (t)\|_{L^2}^2\right)+\mathbb{E} \int_0^T\| v ^{\epsilon}\|_{ \dot{H}^\alpha} ^2 \textrm{d}t\leq C\exp\{\exp\{\exp\{C T\} \}\}.
\end{split}
\end{equation*}
To estimate the term $\|  v ^\epsilon(t)\| _{ L^{ \frac{4}{3}}}$, we utilize the It\^{o} formula  to $ \varphi^{\frac{4}{3}}_\eta ( v ^\epsilon(t))$ to find
\begin{align}\label{4..30}
\mathbb{E}\sup_{r\in [0,t]}\|\varphi _\eta ( v ^\epsilon)\|_{L^{ \frac{4}{3}}}^{ \frac{4}{3}} &\leq\mathbb{E}\|\varphi _\eta ( v ^\epsilon_0)\|_{L^{ \frac{4}{3}}}^{ \frac{4}{3}}+\frac{4}{3}\mathbb{E}\int_0^t\int_{\mathbb{R}^2} \left|\varphi_\eta ^{-\frac{2}{3}}( v ^\epsilon)  v ^\epsilon  \textbf{P}\nabla \wedge [(n^{\epsilon}\nabla \phi)*\rho^{\epsilon}] \right| \textrm{d}x \textrm{d}r \nonumber\\
 &\quad+\frac{4}{3}\mathbb{E} \int_0^t\int_{\mathbb{R}^2}\left|\left(\varphi_\eta ^{-\frac{2}{3}}( v ^\epsilon) v ^\epsilon -| v ^\epsilon|^{-\frac{2}{3}} v ^\epsilon\right)    (-\Delta)^\alpha v ^{\epsilon}\right| \textrm{d}x \textrm{d}r \nonumber\\
 &\quad + \mathbb{E} \int_0^t\sum_{k\geq 1}\int_{\mathbb{R}^2}\left|\frac{2}{3}\varphi_\eta^{-\frac{2}{3}}( v ^\epsilon)- \frac{4}{9} \varphi_\eta^{-\frac{8}{3}}( v ^\epsilon) | v ^\epsilon|^2\right|\left(\textbf{P} \nabla \wedge f(r, u ^{\epsilon})e_k \right)^2\textrm{d}x\textrm{d}r \nonumber\\
 &\quad+ \frac{4}{3} \mathbb{E}\sup_{r\in [0,t]}\bigg|\int_0^r \sum_{k\geq 1}\int_{\mathbb{R}^2}\varphi_\eta ^{-\frac{2}{3}}( v ^\epsilon)  v ^\epsilon\textbf{P} \nabla \wedge f(\varsigma, u ^{\epsilon})e_k \textrm{d}x \textrm{d} W^k  _\varsigma\bigg| \nonumber\\
 &:=\mathbb{E}\|\varphi _\eta ( v ^\epsilon_0)\|_{L^{ \frac{4}{3}}}^4 + \mathcal {L}_1+ \mathcal {L}_2+ \mathcal {L}_3+ \mathcal {L}_4.
\end{align}
Let us estimate each term on both sides of \eqref{4..30}. First, it follows from the fact of $|x|\leq \varphi_\eta (x)=\sqrt{|x|^2+\eta}$ that
\begin{equation*}
\begin{split}
\mathbb{E}\sup_{r\in [0,t]}\| v ^\epsilon(r)\|_{L^{ \frac{4}{3}}}^{ \frac{4}{3}}\leq \mathbb{E}\sup_{r\in [0,t]}\|\varphi _\eta ( v ^\epsilon)\|_{L^{ \frac{4}{3}}}^{ \frac{4}{3}}.
\end{split}
\end{equation*}
Noting that $\varphi^{\frac{4}{3}}_\eta ( v ^\epsilon(0))\downarrow | v ^\epsilon(0)|$ as $\eta\downarrow 0$, the Monotone Convergence Theorem implies that
\begin{equation*}
\begin{split}
\lim_{\eta\rightarrow 0} \int_{\mathbb{R}^2}\varphi^{\frac{4}{3}}_\eta (v^\epsilon(0))\textrm{d}x= \|v^\epsilon(0)\|_{L^{\frac{4}{3}}}^{\frac{4}{3}}\leq \|\nabla \wedge u_0\|_{L^{\frac{4}{3}}}^{\frac{4}{3}}.
\end{split}
\end{equation*}
For $\mathcal {L}_1$, note that
$$
|\varphi_\eta ^{-\frac{2}{3}}(v^\epsilon) v ^\epsilon|\leq | v ^\epsilon|^{\frac{1}{3}} ,~~ \|\rho^{\epsilon}\|_{L^1}=1~~ \textrm{and} ~~\nabla \wedge [(n^{\epsilon}\nabla \phi)*\rho^{\epsilon}]  = [(\nabla \wedge n^{\epsilon})\nabla \phi]*\rho^{\epsilon}.
$$
For any $\delta>0$, we have
\begin{equation*}
\begin{split}
\mathcal {L}_1
&\leq C (\phi)\sup_{r\in [0,t]}\| v ^\epsilon(r)\|_{L^{\frac{4}{3}}}^{\frac{1}{3}} \int_0^t\|\sqrt{n^{\epsilon}}\|_{L^{4}}\|\nabla \sqrt{n^{\epsilon}}\|_{L^{2}} \textrm{d}r\\
&\leq \delta \sup_{r\in [0,t]}\| v ^\epsilon(r)\|_{L^{\frac{4}{3}}}^{\frac{4}{3}}  +  \int_0^t \|\nabla \sqrt{n^{\epsilon}}\|_{L^{2}}^2  \textrm{d}r + C(\delta, \phi)\int_0^t \| n^{\epsilon} \|_{L^{2}}^2 \textrm{d}r.
\end{split}
\end{equation*}
For $\mathcal {L}_2$, since the term $|\varphi_\eta ^{-2/3}( v ^\epsilon)v^\epsilon -| v ^\epsilon|^{-2/3} v ^\epsilon| |(-\Delta)^\alpha  v ^{\epsilon}|$ is a monotone nonnegative sequence with respect to $\eta$ over the set
$$
Z_{\neq}:= \left\{(r,x)\in [0,t]\times \mathbb{R}^2| ~  v ^{\epsilon}(r,x)\neq0\right\}.
$$
We get from the Monotone Convergence Theorem that
$$
\iint_{Z_{\neq}}\left| \varphi_\eta ^{-\frac{2}{3}}( v ^\epsilon) v ^\epsilon -| v ^\epsilon|^{-\frac{2}{3}} v ^\epsilon \right|    |(-\Delta)^\alpha  v ^{\epsilon}| \textrm{d}x \textrm{d}r \rightarrow 0,~~ \textrm{as}~\eta\rightarrow0.
$$
Moreover, note that
$$
\left|\left(\varphi_\eta ^{-\frac{2}{3}}( v ^\epsilon) v ^\epsilon -| v ^\epsilon|^{-\frac{2}{3}} v ^\epsilon\right)    (-\Delta)^\alpha  v ^{\epsilon}\right|\leq | v ^\epsilon|^{\frac{1}{3}} |(-\Delta)^\alpha  v ^{\epsilon}| \in L^1((0,t)\times\mathbb{R}^2),\quad \mathbb{P}\textrm{-a.s. }
$$
The Dominated Convergence Theorem implies that
$$
\iint_{Z_{=}}\left| \varphi_\eta ^{-\frac{2}{3}}( v ^\epsilon) v ^\epsilon -| v ^\epsilon|^{-\frac{2}{3}} v ^\epsilon \right|    |(-\Delta)^\alpha  v ^{\epsilon}| \textrm{d}x \textrm{d}r = 0 ,
$$
where $Z_{=}:=\{(r,x)\in [0,t]\times \mathbb{R}^2|~   v ^{\epsilon}(r,x)=0\}$. As a result, we obtain
\begin{equation*}
\begin{split}
\lim_{\eta \downarrow0} \mathcal {L}_2\leq \frac{4}{3}\lim_{\eta \downarrow0}\mathbb{E}\iint_{Z_{\neq}}\left|\varphi_\eta ^{-\frac{2}{3}}( v ^\epsilon) v ^\epsilon -| v ^\epsilon|^{-\frac{2}{3}} v ^\epsilon\right| \left|(-\Delta)^\alpha  v ^{\epsilon}\right| \textrm{d}x \textrm{d}r=0.
\end{split}
\end{equation*}
For $\mathcal {L}_3$, the assumption (H3) guarantees that
\begin{equation*}
\begin{split}
\mathcal {L}_3\leq C\mathbb{E} \int_0^t\sum_{k\geq 1}\int_{\mathbb{R}^2} \varphi_\eta^{-\frac{2}{3}}( v ^\epsilon) \left(\textbf{P} \nabla \wedge f_k(r, u ^{\epsilon})  \right)^2\textrm{d}x\textrm{d}r
\leq C\mathbb{E} \int_0^t \left(1+\| v ^\epsilon \|_{L^{\frac{4}{3}}}^{\frac{4}{3}}\right) \textrm{d}r.
\end{split}
\end{equation*}
For $\mathcal {L}_4$, it follows from the BDG inequality and the Minkowski inequality that, for any $\zeta>0$,
\begin{equation*}
\begin{split}
\mathcal {L}_4&\leq  C\mathbb{E}  \left[\int_0^t \sum_{k\geq 1}\left(\int_{\mathbb{R}^2}\varphi_\eta ^{-\frac{2}{3}}( v ^\epsilon) | v ^\epsilon|  |\nabla \wedge f(r, u ^{\epsilon})e_k(x)|\textrm{d}x\right)^2 \textrm{d}r \right]^{\frac{1}{2}}\\
&\leq  \zeta\mathbb{E}  \sup_{r \in [0,t]}\| v ^\epsilon (r)\|_{L^{\frac{4}{3}}}^{\frac{4}{3}} + C(\zeta) \mathbb{E} \int_0^t \left(1+\| v ^\epsilon \|_{L^{\frac{4}{3}}}^{\frac{4}{3}}\right) \textrm{d}r.
\end{split}
\end{equation*}
Plugging the above estimates  into \eqref{4.28} and choosing $\delta=\zeta=\frac{1}{4}$,  we get
\begin{equation*}
\begin{split}
&\mathbb{E}\sup_{t\in [0,T]}\| v ^\epsilon(t)\|_{L^{\frac{4}{3}}}^{\frac{4}{3}}\leq C \exp\{C T\}\left(1+  \mathbb{E} \int_0^T \|\nabla \sqrt{n^{\epsilon}}\|_{L^{2}}^2  \textrm{d}r +   \mathbb{E} \int_0^T \| v ^\epsilon \|_{L^{\frac{4}{3}}}^{\frac{4}{3}} \textrm{d}t\right),
\end{split}
\end{equation*}
which implies the desired estimate. The proof of Lemma \ref{lem4.6} is completed.
\end{proof}

\subsection{Pathwise solution for KS-SNS system}

Based on the uniform bounds derived in the previous subsection, one can now prove the existence of martingale solution to the KS-SNS system \eqref{KS-SNS}.

\begin{lemma}[Martingale weak solution]\label{lem4.7}
Suppose that the assumptions (H1)-(H3) hold, then the KS-SNS system \eqref{KS-SNS} possesses at least one global martingale weak solution $\left( (\widetilde{\Omega} , \widetilde{\mathcal {F}}, \widetilde{\mathbb{P}}),\widetilde{W},\widetilde{n},\widetilde{c},\widetilde{u}\right)$.
\end{lemma}

\begin{proof}
Let us first confirm the tightness of $(n^\epsilon,c^\epsilon,u^\epsilon)$ by virtue of the uniform estimates obtained in last subsection.

$\bullet$ Since $\mathcal {L}[W^\epsilon]$ is a single Radon measure on the Polish space $\chi_{W^\epsilon}:= C([0,T];U_0)$, it is tight.

$\bullet$ We denote by $\mathcal {L}[n^\epsilon]$ the law of $n^\epsilon$ on the phase space $\chi_{n^\epsilon}:= L^2(0,T;L^2_{\textrm{loc}}(\mathbb{R}^2))$. According to \eqref{4.28} in Lemma \ref{lem4.6} and the $n^\epsilon$-equation in \eqref{3.1}, one can easily verify that $n^\epsilon$ is uniformly bounded (with respect to $\epsilon$) in $W^{1,2}(0,T;H^{-1}(\mathbb{R}^2))$ over $\Omega^\epsilon_N$.

 {Let us denote by $B_m\subseteq\mathbb{R}^2$ the open ball with radius $m\in \mathbb{N}$ centered at $\textbf{0}$}, then for each $m>0$, there exists a constant $a_m>0$ such that the bound
$$
 \| n^\epsilon  \|_{L^2(0,T;H^{1}(B_m))}  + \| \partial_t n^\epsilon  \|_{L^2(0,T;H^{-1}(B_m))} \lesssim a_m,\quad \textrm{for all}~ \omega\in  \Omega  ^\epsilon _N,
$$
holds  uniformly in $\epsilon$. Moreover, by the Aubin-Lions Theorem  {(\cite[Theorem 2.1 in Chapter III]{temam2001navier}}), for any sequence of balls $\{B_m\}_{m\in\mathbb{N}}$, the space
$$
\mathcal {V}:= \left\{f\in \chi_{n^\epsilon};~ f\in L^2\left(0,T;H^1(B_m)\right) ,~f\in W^{1,2}\left(0,T;L^2(B_m)\right) \right\}
$$
is relatively compact in $L^2(0,T;L^2(B_m))$. By \eqref{4..25}, we have
\begin{equation*}
\begin{split}
 \mathcal {L}[n^\epsilon]\{\|n^\epsilon\|_{\mathcal {V}}\leq a_m\} \geq \mathbb{P}\{\Omega_N^\epsilon\}  \geq 1- \frac{C}{N}.
\end{split}
\end{equation*}
By choosing $N>1$ as large as we wish, one can prove the tightness of $\{\mathcal {L}[n^\epsilon]:\epsilon \in (0,1)\}$ on $L^2 (0,T;L^2_{\textrm{loc}}(\mathbb{R}^2))$. Similarly, one can also obtain the tightness on $(L^2(0,T;H^{1}(\mathbb{R}^2)),\textrm{weak})$  by virtue of the bound \eqref{4..25}, where $(G,\textrm{weak})$ denotes the Banach space $G$ equipped with the weak topology.

Hence, the family of probability measures $\{\mathcal {L}[n^\epsilon]:\epsilon \in (0,1)\}$ is tight on
$$
\chi_{n^\epsilon}:= L^2 \left(0,T;L^2_{\textrm{loc}}(\mathbb{R}^2)\right)\cap  \left(L^2(0,T;H^{1}(\mathbb{R}^2)),\textrm{weak}\right).
$$

$\bullet$ Due to the uniform bound \eqref{4.29} in Lemma \ref{4.7} and the $c^\epsilon$-equation in \eqref{3.1}, one can easily verify that $\{\partial_t c^\epsilon\}_{\epsilon \in (0,1)}$ is uniformly bounded in $L^2(\Omega; L^2(0,T;L^2(\mathbb{R}^2)))$, and the family of the measures  $\{\mathcal {L}[c^\epsilon]:\epsilon \in (0,1)\}$ is tight on
$$
\chi_{c^\epsilon}:= L^\infty \left(0,T;L^2_{\textrm{loc}}(\mathbb{R}^2)\right) \cap L^2\left(0,T;H^1_{\textrm{loc}}(\mathbb{R}^2)\right).
$$

$\bullet$ To prove the tightness of $\{ u ^\epsilon\}_{\epsilon \in (0,1)}$, we first show that there is a $C>0$ independent of $\epsilon$ such that for some $\gamma \in (0,1)$
\begin{equation}\label{4...30}
\begin{split}
\mathbb{E}\| u ^\epsilon\|_{ W^{\gamma,2}(0,T;H^{1-\alpha}) }^2 \leq C,
\end{split}
\end{equation}
where
$$
W^{\gamma,2}\left(0,T;H^{1-\alpha}(\mathbb{R}^2)\right):= \left\{f \in L^2\left(0,T;H^{1-\alpha}(\mathbb{R}^2)\right);~\int_0^T \int_0^T  \frac{\|f(t)-f(r)\|_{H^{1-\alpha}}^2}{|t-r|^{1+2 \gamma}}   \textrm{d}t\textrm{d}r<\infty\right\},
$$
and it is endowed with the norm
$$
\|f\|_{W^{\gamma,2}(0,T;H^{1-\alpha})}^2 = \int_0^T\|f(t)\|_{ H^{1-\alpha}}^2\textrm{d}t+ \int_0^T \int_0^T \frac{  \|f(t)-f(s)\|_{H^{1-\alpha}}^2}{|t-s|^{1+2 \gamma}} \textrm{d}t\textrm{d}r.
$$
Indeed, we deduce from Lemma \ref{lem4.2} that, for any $\gamma\in (0,1)$, the inclusions
\begin{equation*}
\begin{split}
&\int_0^t \textbf{P} ( u ^{\epsilon}\cdot \nabla)  u ^{\epsilon} \textrm{d}r, ~\int_0^t \textbf{P}(-\Delta)^\alpha  u ^{\epsilon} \textrm{d}r,~\int_0^t \textbf{P}(n^{\epsilon}\nabla \phi)*\rho^{\epsilon} \textrm{d}r\\
&\quad \in L^2\left(\Omega; W^{1,2}(0,T;H^{1-\alpha}(\mathbb{R}^2))\right)\subseteq L^2\left(\Omega; W^{\gamma,2}(0,T;H^{1-\alpha}(\mathbb{R}^2))\right)
\end{split}
\end{equation*}
are uniformly bounded in $\epsilon$. For the stochastic integration term in \eqref{3.1}$_3$, it follows from the Lemma \ref{lem4.6} and the same argument in Lemma 2.1 of \cite{flandoli1995martingale} that
\begin{equation*}
\begin{split}
\mathbb{E}\bigg\|\int_0^t\textbf{P} f(r, u ^{\epsilon}) \mathrm{d} W^\epsilon_r\bigg\|_{W^{\gamma,2}(0,T;H^{1-\alpha} )}^2 &\leq C \mathbb{E} \int_0^T \| f(r, u ^{\epsilon})\|_{L_2(U; H^{1-\alpha})}^2 \textrm{d}r  \\
&\leq C \mathbb{E} \left(1+\sup_{t \in [0,T]}\| v ^{\epsilon}(t) \|_{ L^{2} }^2\right) \leq C,
\end{split}
\end{equation*}
for some $\gamma \in (0,1)$. This proves \eqref{4...30}.

 {For each open ball $B_m \subset \mathbb{R}^2$ with radius $m\in \mathbb{N}$ centered at $\textbf{0}$}, we get from the Lemma \ref{lem4.6} and \eqref{4...30} that there exists $a_m>0$ such that
\begin{equation} \label{4.31}
\begin{split}
\mathbb{E}\left(\| u ^\epsilon\|_{L^2 (0,T;H^{1+\alpha}(B_m))}^2+ \| u ^\epsilon\|_{W^{\gamma,2}(0,T;H^{1-\alpha}(B_m)) }^2 \right)\leq a_m.
\end{split}
\end{equation}
Note that the following embedding
$$
L^2\left(0,T;H^{1+\alpha}(B_m)\right)\bigcap W^{\gamma,2}\left(0,T;H^{1-\alpha}(B_m)\right)\subset  L^2\left(0,T;H^1(B_m)\right)
$$
is compact, we deduce from the bound  \eqref{4.31} and the Chebyshev inequality that the family of measures $\mathcal {L}[u^\epsilon]$ is tight on $L^2(0,T;H^1_{\textrm{loc}}(\mathbb{R}^2))$.

Moreover, the momentum estimate \eqref{4.30} implies that the family of measures $\mathcal {L}[u^\epsilon]$ is also tight on the spaces $\left(L^\infty(0,T;\dot{W}^{1,\frac{4}{3}}(\mathbb{R}^2)),\textrm{weak}*\right)$ and $\left(L^2(0,T;H^{1+\alpha}(\mathbb{R}^2)),\textrm{weak}\right)$. Here $(H,\textrm{weak}*)$ stands for the Banach space $H$ equipped with the weak-star topology. Therefore, the family of measures $\mathcal {L}[u^\epsilon]$ is tight on
$$
\chi_{u^\epsilon}:= L^2\left(0,T;H^1_{\textrm{loc}}(\mathbb{R}^2)\right) \cap \left(L^\infty(0,T;\dot{W}^{1,\frac{4}{3}}(\mathbb{R}^2)),\textrm{weak}*\right)\cap \left(L^2(0,T;H^{1+\alpha}(\mathbb{R}^2)),\textrm{weak}\right).
$$

In conclusion, we have proved that the sequence $\{(W^\epsilon,n^\epsilon,c^\epsilon, u ^\epsilon)\}_{\epsilon> 0}$ is tight on the phase space
\begin{equation*}
\begin{split}
\textbf{X} := &C([0,T];U_0)\times L^2 \left(0,T;L^2_{\textrm{loc}}(\mathbb{R}^2)) \cap \left(L^2(0,T;H^{1 }(\mathbb{R}^2)\right),\textrm{weak}\right)\\
&\times L^2 \left(0,T;H_{\textrm{loc}}^{1 }(\mathbb{R}^2)\right)\cap \left(L^2(0,T;H^{2}(\mathbb{R}^2)),\textrm{weak}\right) \\
&\times L^2 \left(0,T;H_{\textrm{loc}}^{1 }(\mathbb{R}^2)\right)\cap \left(L^\infty(0,T;W^{1,\frac{4}{3}}(\mathbb{R}^2)),\textrm{weak}*\right)\cap \left(L^2(0,T;H^{1+\alpha}(\mathbb{R}^2)),\textrm{weak}\right).
\end{split}
\end{equation*}
In view of  the Prokhov Theorem (cf. \cite[Theorem 2.3]{da2014stochastic}), there exists a subsequence $\{\epsilon_j\}_{j\in\mathbb{N}}$ of $\{\epsilon\}_{\epsilon>0}$ and a probability measure $\pi$ defined on $\textbf{X}$ such that
$$
\pi^{\epsilon_j}=\mathcal {L}\left[W^{\epsilon_j},n^{\epsilon_j},c^{\epsilon_j}, u^{\epsilon_j}\right]\rightharpoonup\pi ~~~ \textrm{as}~~~j\rightarrow\infty.
$$
According to the Skorokhod Theorem (cf. \cite[Theorem 2.4]{da2014stochastic}), there exists a stochastic basis $( \widetilde{\Omega} , \widetilde{\mathcal {F}}, (\widetilde{\mathcal {F}}_t)_{t\in [0,T]}, \widetilde{\mathbb{P}})$, on which a sequence of $\textbf{X}$-valued random variables $ (\widetilde{W}^{\epsilon_j},\widetilde{n}^{\epsilon_j}, \widetilde{c}^{\epsilon_j},\widetilde{u}^{\epsilon_j})_{j\geq 1}$ and an element $(\widetilde{W},\widetilde{n},\widetilde{c},\widetilde{u})$ can be defined,  such that

\textbf{(a$_1$)} the joint laws $
\mathcal {L}\left[\widetilde{W}^{\epsilon_j},\widetilde{n}^{\epsilon_j},\widetilde{c} ^{\epsilon_j},\widetilde{u} ^{\epsilon_j}\right]$ and $\mathcal {L}\left[W^{\epsilon_j},n^{\epsilon_j},c ^{\epsilon_j},u ^{\epsilon_j}\right]$ coincide  on $\textbf{X}$;

\textbf{(a$_2$)}  the law $\mathcal {L}\left[\widetilde{W},\widetilde{n},\widetilde{c},\widetilde{u}\right]=\pi$ is a Radon measure on $\textbf{X}$, and we have $ \widetilde{\mathbb{P}}$-a.s.
\begin{subequations}\label{4.32}
\begin{align}
 \widetilde{W}^{\epsilon_j}\rightarrow\widetilde{W} \quad\textrm{in}\quad& C([0,T];U),\tag{4.32a}\\
 \widetilde{n}^{\epsilon_j}\rightarrow\widetilde{n} \quad\textrm{in}\quad& L^2 \left(0,T;L^2_{\textrm{loc}}(\mathbb{R}^2)\right) \cap \left(L^2(0,T;H^{1 }(\mathbb{R}^2)),\textrm{weak}\right) ,\tag{4.32b}\\
 \widetilde{c}^{\epsilon_j}\rightarrow\widetilde{c} \quad\textrm{in}\quad& L^2 \left(0,T;H_{\textrm{loc}}^{1 }(\mathbb{R}^2)\right)\cap \left(L^2(0,T;H^{2}(\mathbb{R}^2)),\textrm{weak}\right),\tag{4.32c}\\
 \widetilde{u}^{\epsilon_j}\rightarrow\widetilde{u} \quad\textrm{in}\quad& L^2 \left(0,T;H_{\textrm{loc}}^{1 }(\mathbb{R}^2)\right)\cap \left (L^\infty(0,T;\dot{W}^{1,\frac{4}{3}}(\mathbb{R}^2)),\textrm{weak}*\right)\nonumber\\
 &\cap \left (L^2(0,T;H^{1+\alpha}(\mathbb{R}^2)),\textrm{weak}\right);\tag{4.32d}
 \end{align}
 \end{subequations}

\textbf{(a$_3$)} the quadruple $\left(\widetilde{W}^{\epsilon_j},\widetilde{n}^{\epsilon_j},\widetilde{c}^{\epsilon_j},
\widetilde{u}^{\epsilon_j}\right)$ satisfies the \eqref{KS-SNS} in the sense of distribution, $\widetilde{\mathbb{P}}$-a.s.

\vspace{2mm}
Now we show that the quantity $(( \widetilde{\Omega} , \widetilde{\mathcal {F}}, (\widetilde{\mathcal {F}}_t)_{t\in [0,T]}, \widetilde{\mathbb{P}}),\widetilde{n}, \widetilde{c}, \widetilde{u})$ is a martingale solution to the system \eqref{KS-SNS} by taking the limits as $j\rightarrow\infty$. Here we just treat the fluid equation in \eqref{KS-SNS}, since the other two equations can be treated in a similar manner. This will be done by four steps $1)-5)$.


1) Since the process $\widetilde{W}^{\epsilon_j}$ has the same law as $W$, by applying the \cite[Lemma 2.1.35 and Corollary 2.1.36]{breit2018stochastically} and  \cite[Proof of Lemma 4.8]{Debussche2016}, one can conclude that $\widetilde{W}^{\epsilon_j}=\widetilde{W}$ is a $(\widetilde{\mathcal {F}}_t)_{t\in [0,T]}$-cylindrical Wiener process. Hence,  there exists a family of mutually independent real-valued $(\widetilde{\mathcal {F}}_t)_{t\in [0,T]}$-Wiener processes $(\widetilde{W}^k)_{k \geq 1}$ such that
$$
\widetilde{W}(t,x,\omega)= \sum_{k\geq 1} e_k(x) \widetilde{W}^k(t,\omega).
$$

2) As the triple $\left(n^{\epsilon_j}, c^{\epsilon_j}, u ^{\epsilon_j}\right)$ has the same law with $\left(\widetilde{n}^{\epsilon_j}, \widetilde{c}^{\epsilon_j},\widetilde{u}^{\epsilon_j}\right)$ (by \textbf{(a$_1$)}), if the expectation $\mathbb{E}$ is replaced by $\widetilde{\mathbb{E}}$, then the previous proof means that $\left(\widetilde{n}^{\epsilon_j}, \widetilde{c}^{\epsilon_j},\widetilde{u}^{\epsilon_j}\right)$ satisfies the same uniform estimates in Lemmas \ref{lem4.6}, i.e.,
\begin{equation*}
\begin{split}
\sup_{j\geq 1}\left(\widetilde{\mathbb{E}}\sup_{t\in[0,T]}\| \widetilde{u}^{\epsilon_j}(t)\|_{ H^1}^2+\widetilde{\mathbb{E}}\sup_{t\in[0,T]}\|\widetilde{u}^{\epsilon_j}(t)\|_{ \dot{W}^{1,\frac{4}{3}}}^{ \frac{4}{3}}+\widetilde{\mathbb{E}}\int_0^T\|\widetilde{u}^{\epsilon_j}(t)\|_{ H^{1+\alpha} }^2\textrm{d}t\right)  \leq C,
\end{split}
\end{equation*}
for some constant $C>0$  independent of $\epsilon_j$, which implies that there exists a subsequence of $(\widetilde{u}^{\epsilon_j})_{j\geq1}$, still denoted by itself from simplicity, convergent weak star in $L^2(\Omega;L^\infty(0,T; H^1(\mathbb{R}^2) ))\cap L^{4/3}(\Omega;L^\infty(0,T; \dot{W}^{1,4/3}(\mathbb{R}^2) ))$ and weakly in $ L^2(\Omega \times [0,T];H^{1+\alpha}(\mathbb{R}^2))$ to an element $\varrho$ as $j\rightarrow\infty$. Since $\widetilde{u}^{\epsilon_j}\rightarrow\widetilde{u}$ in the phase space $\chi_{u^\epsilon} $ by conclusion \textbf{(a$_2$)}, we infer that $\varrho=\widetilde{u}$, and so
\begin{equation*}
\begin{split}
\widetilde{u} \in L^\infty \left (0,T; H ^1(\mathbb{R}^2)\cap \dot{W}^{1,\frac{4}{3}}(\mathbb{R}^2)\right) \cap L^2 \left(0,T; H^{1+\alpha}(\mathbb{R}^2)\right),\quad \widetilde{\mathbb{P}}\textrm{-a.s.}
\end{split}
\end{equation*}

3) By using the convergence in (4.32d), we get that for any $\varphi \in C^\infty_0(\mathbb{R}^2; \mathbb{R}^2)$ with $\textrm{div} \varphi =0$
\begin{equation*}
\begin{split}
&\int_0^T( \widetilde{u}^{\epsilon_j}\otimes \widetilde{u}^{\epsilon_j},\nabla\varphi)_{L^2}\textrm{d}t\rightarrow\int_0^T( \widetilde{u}\otimes \widetilde{u},\nabla\varphi)_{L^2}\textrm{d}t;\quad \int_0^T(\widetilde{n}^{\epsilon_j}\nabla \phi,\varphi)_{L^2}\textrm{d}t\rightarrow\int_0^T(\widetilde{n}\nabla \phi,\varphi)_{L^2}\textrm{d}t,\\
& \int_0^T\left( (-\Delta)^\frac{\alpha}{2}\widetilde{u}^{\epsilon_j},(-\Delta)^\frac{\alpha}{2}\varphi\right)_{L^2}\textrm{d}t\rightarrow\int_0^T\left( (-\Delta)^\frac{\alpha}{2}\widetilde{u},(-\Delta)^\frac{\alpha}{2}\varphi\right)_{L^2}\textrm{d}t,\quad \textrm{as} \ j\rightarrow\infty,
\end{split}
\end{equation*}

4) Let us verify the convergence of the stochastic integral. It will be done by applying Bensoussan's method in \cite{bensoussan1995stochastic}; see also \cite{chen2019martingale}. Indeed, for each $  \delta>0$, introducing the functions $$\Theta(t):=f(t,\widetilde{u}(t)),\quad \Theta _\delta(t):=\varrho_\delta \star \Theta(t)= \frac{1}{\delta}\int_0^t \varrho  (\frac{t-r}{\delta} )\Theta(r) \textrm{d}r,$$ where $\varrho (\cdot)$ is a standard mollifier. Note that
\begin{equation}\label{b1}
\begin{split}
\widetilde{\mathbb{E}}\int_0^T \|\Theta _\delta(t)\|_{L_2(U;L^2)}^2\textrm{d}t\leq \widetilde{\mathbb{E}}\int_0^T \|\Theta(t)\|_{L^2}^2\textrm{d}t,
\end{split}
\end{equation}
and
\begin{equation}\label{b2}
\begin{split}
\Theta _\delta(t)\rightarrow \Theta(t) \quad \textrm{in}~ L^2(\widetilde{\Omega};L^2(0,T;L^2(\mathbb{R}^2)))~ \textrm{as}~\delta\rightarrow 0.
\end{split}
\end{equation}
Set
$
\Theta ^{\epsilon_j}(t):= f(t,\widetilde{u}^{\epsilon_j}(t)).
$
By \eqref{4.32}, the Lemma \ref{lem4.2} and the BDG inequality, we infer that the sequence $\int_0^t \Theta ^{\epsilon_j}(r)\textrm{d} \widetilde{W}_r^{\epsilon_j}$ is uniformly bounded in $ L^2(\Omega;L^2(\mathbb{R}^2))$. Therefore, there exists a subsequence of $\{\epsilon_j\}$ (still denoted by itself) and an element $\xi\in L^2(\Omega;L^2(\mathbb{R}^2))$ such that, for any $\psi\in L^2(\Omega;L^2(\mathbb{R}^2))$,
\begin{equation}\label{b7}
\begin{split}
\widetilde{\mathbb{E}}\langle  \int_0^t \Theta ^{\epsilon_j}(r)\textrm{d} \widetilde{W}_r^{\epsilon_j},\psi\rangle \rightarrow \widetilde{\mathbb{E}}\langle \xi,\psi\rangle ~ \textrm{as}~j\rightarrow\infty,
\end{split}
\end{equation}
where $\widetilde{\mathbb{E}} \langle\cdot,\cdot\rangle $ denotes the inner product in $L^2(\Omega;L^2(\mathbb{R}^2))$. To finish the proof, it remains to show that $\xi=\int_0^t f(r,\widetilde{u} (r))\textrm{d} \widetilde{W}_r$. Indeed, since $\Theta_\delta^{\epsilon_j}(t)= \varrho_\delta\star\Theta^{\epsilon_j}(t)$ is smooth in time, one can get by integrating by parts that
\begin{equation*}
\begin{split}
 \int_0^t \Theta_\delta^{\epsilon_j}(r)\textrm{d} \widetilde{W}_r^{\epsilon_j}=\Theta_\delta^{\epsilon_j}(t) \widetilde{W}_t^{\epsilon_j}-\int_0^t \widetilde{W}_r^{\epsilon_j}\frac{\partial}{\partial r}\Theta_\delta^{\epsilon_j}(r) \textrm{d}r.
\end{split}
\end{equation*}
By adopting the same argument as in \cite[Theorem 1, P.2046-2047]{chen2019martingale}, one can prove that for any $\psi\in L^2(\widetilde{\Omega};L^2(\mathbb{R}^2))$,
\begin{equation}\label{b4}
\begin{split}
\widetilde{\mathbb{E}}\langle\int_0^t \Theta_\delta^{\epsilon_j}(r)\textrm{d} \widetilde{W}_r^{\epsilon_j},\psi \rangle \rightarrow \widetilde{\mathbb{E}}\langle\int_0^t \Theta_\delta (r) \textrm{d} \widetilde{W}_r,\psi \rangle~  \textrm{as}~j\rightarrow\infty.
\end{split}
\end{equation}
Observing that for any $\psi\in L^2(\widetilde{\Omega};L^2(\mathbb{R}^2))$,
\begin{equation}\label{b5}
\begin{split}
&\widetilde{\mathbb{E}}\langle\int_0^t \Theta^{\epsilon_j}(r)\textrm{d} \widetilde{W}_r^{\epsilon_j},\psi  \rangle - \langle\int_0^t \Theta (r) \textrm{d} \widetilde{W}_r,\psi \rangle\\
&\quad =\widetilde{\mathbb{E}}\langle\int_0^t (\Theta^{\epsilon_j}(r)-\Theta^{\epsilon_j}_\delta(r))\textrm{d} \widetilde{W}_r^{\epsilon_j},\psi  \rangle+\widetilde{\mathbb{E}}\langle\int_0^t  \Theta^{\epsilon_j}_\delta(r) \textrm{d} \widetilde{W}_r^{\epsilon_j}-\int_0^t  \Theta _\delta(r) \textrm{d} \widetilde{W}_r ,\psi  \rangle \\
&\quad\quad+\widetilde{\mathbb{E}}\langle\int_0^t ( \Theta _\delta(r)-\Theta  (r) ) \textrm{d} \widetilde{W}_r ,\psi  \rangle\\
&\quad:= A_1+A_2+A_3.
\end{split}
\end{equation}
For $A_1$, we get by the Cauchy-Schwarz inequality, \eqref{4.32} and \eqref{b1} that
\begin{equation}\label{b6}
\begin{split}
 |A_1|  &\leq \left(\widetilde{\mathbb{E}}\|\psi\|_{L^2}^2\right)^{1/2} \left(\widetilde{\mathbb{E}}\left\| \int_0^t (\Theta^{\epsilon_j}(r)-\Theta^{\epsilon_j}_\delta(r))\textrm{d} \widetilde{W}_r^{\epsilon_j}\right\|_{L^2}^2\right)^{1/2}\\
 &\leq \left(\widetilde{\mathbb{E}}\|\psi\|_{L^2}^2\right)^{1/2} \left(\widetilde{\mathbb{E}} \int_0^t \left\|\Theta^{\epsilon_j}(r)-\Theta^{\epsilon_j}_\delta(r) \right\|_{L_2(U;L^2)}^2\textrm{d}r \right)^{1/2}\rightarrow0,
\end{split}
\end{equation}
as $j\rightarrow\infty$, $\delta\rightarrow0$. Similar as $A_1$, for the term $A_3$, it follows from the Cauchy-Schwarz inequality and \eqref{b2} that $|A_3|\rightarrow0$ as $\delta\rightarrow 0$, which together with \eqref{b7}, \eqref{b4}-\eqref{b6} imply that $\xi=\int_0^t f(r,\widetilde{u} (r))\textrm{d} \widetilde{W}_r$.

5) Since the quadruple $(\widetilde{W}^{\epsilon_j},\widetilde{n}^{\epsilon_j},\widetilde{c}^{\epsilon_j},
\widetilde{u}^{\epsilon_j})$ satisfies the system \eqref{KS-SNS} in the sense of distribution (by conclusion \textbf{(a$_3$)}), one can now take the limit as  $j\rightarrow\infty$ in the $\widetilde{u}^{\epsilon_j}$-equation in \eqref{KS-SNS} to find
\begin{equation}\label{4.34}
\begin{split}
(\widetilde{u},\varphi)_{L^2}  =&(\widetilde{u}_0,\varphi)_{L^2} +\int_0^t  ( \widetilde{u}\otimes \widetilde{u},\nabla\varphi)_{L^2} \textrm{d}r-\int_0^t \left( (-\Delta)^\frac{\alpha}{2}\widetilde{u},(-\Delta)^\frac{\alpha}{2}\varphi\right)_{L^2}\textrm{d}r \\
 &+ \int_0^t(\widetilde{n}\nabla \phi,\varphi)_{L^2}\textrm{d}r+ \sum_{k\geq1}\int_0^t (f(r,\widetilde{u})e_k ,\varphi)_{L^2} \mathrm{d} \widetilde{W}^k_r,
\end{split}
\end{equation}
$\widetilde{\mathbb{P}}$-a.s., for any $t\in [0,T]$ and any $\varphi \in C^\infty_0(\mathbb{R}^2; \mathbb{R}^2)$ with $\textrm{div} \varphi =0$.

Moreover, we have all of the spatio-temporal regularity of solutions to take the limit as $j\rightarrow\infty$ to obtain that $(\widetilde{n},\widetilde{c})$ satisfies the first two random PDEs of \eqref{KS-SNS} in the sense of distribution. More precisely, the following identities hold  $\widetilde{\mathbb{P}}$-a.s.:
\begin{equation*}
\begin{split}
(\widetilde{n}(t) ,\varphi_1)_{L^2}&=(\widetilde{n} _0,\varphi_1)_{L^2}+\int_0^t\left(\widetilde{u}  \widetilde{n} -\nabla \widetilde{n} + \widetilde{n} \nabla \widetilde{c}  ,\nabla\varphi_1\right)_{L^2}\textrm{d}t +\int_0^t\left( \widetilde{n} -   \widetilde{n} ^2 ,\varphi_1\right)_{L^2} \textrm{d}t,
\end{split}
\end{equation*}
 and
 \begin{equation*}
\begin{split}
(\widetilde{c}(t) ,\varphi_2)_{L^2}&=(\widetilde{c} _0,\varphi_2)_{L^2}+\int_0^t \left(\widetilde{u}  \widetilde{c} -\nabla \widetilde{c} , \nabla\varphi_2\right)_{L^2} \textrm{d}t -\int_0^t\left(\widetilde{n}  \widetilde{c} ,\varphi_2\right)_{L^2}\textrm{d}t,
\end{split}
\end{equation*}
for any $t \in [0,T]$ and for all $\varphi_1,\varphi_2\in C_0^\infty(\mathbb{R}^2;\mathbb{R})$.

Now we need to verify the regularity satisfied by the solution $\widetilde{u}$. First, by \eqref{4.30}, it follows from the compactness criteria  ({cf. \cite[Theorem 5]{simon1986compact}}) that
$
\widetilde{u} \in C \left([0,T]; H^{1}(\mathbb{R}^2)\right) $, $\widetilde{\mathbb{P}}$-a.s.
Second, we demonstrate that
$$
\widetilde{u} \in C \left([0,T]; \dot{W}^{1,\frac{4}{3}}(\mathbb{R}^2)\right) ~~ \textrm{or}~~
 \widetilde{v}=\nabla \wedge \widetilde{u} \in C \left([0,T]; L^{ \frac{4}{3}}(\mathbb{R}^2)\right), ~~\widetilde{\mathbb{P}}\textrm{-a.s.}
$$
Indeed, by applying the It\^{o} formula  to $\varphi_\eta^{ \frac{4}{3}}(\widetilde{u}(t))$ with $\varphi_\eta=\sqrt{|x|^2+\eta}$,  we infer that
\begin{align}
\widetilde{\mathbb{E}}\sup_{r\in [s,t]}\|\varphi _\eta (\widetilde{v} )\|_{L^{ \frac{4}{3}}}^{ \frac{4}{3}}
\leq& \widetilde{\mathbb{E}}\|\varphi _\eta (\widetilde{v}(s))\|_{L^{ \frac{4}{3}}}^{ \frac{4}{3}} +C\widetilde{\mathbb{E}}\int_s^t\left|(\varphi_\eta ^{-\frac{2}{3}}(\widetilde{v})\widetilde{v} -|\widetilde{v}|^{-\frac{2}{3}}\widetilde{v},    (-\Delta)^\alpha \widetilde{v})_{L^2}\right|\textrm{d}r \nonumber\\
 &+C\widetilde{\mathbb{E}}\int_s^t\left(\varphi_\eta ^{-\frac{2}{3}}(\widetilde{v}) \widetilde{v},  \textbf{P}\nabla \wedge (\widetilde{n}\nabla \phi)\right)_{L^2} \textrm{d}r\label{4..45}\\
&+ C\widetilde{\mathbb{E}} \int_s^t \left\| \varphi_\eta^{-\frac{2}{3}}(\widetilde{v})\textbf{P} \nabla \wedge f(r, u ^{\epsilon})\right\|_{L_2(U;L^2)}^2\textrm{d}x\textrm{d}r \nonumber\\
 &+ C \widetilde{\mathbb{E}}\sup_{r\in [s,t]}\bigg|\int_0^r (\varphi_\eta ^{-\frac{2}{3}}(\widetilde{v}) \widetilde{v},\textbf{P} \nabla \wedge f(\varsigma,\widetilde{u}))_{L^2}\textrm{d} W_\varsigma \bigg|.\nonumber
\end{align}
Similar to the arguments in the proof of Lemma \ref{lem4.5}, by using the uniform bounds in Lemma \ref{lem4.6}, one can deduce from \eqref{4..45} that
$$
\widetilde{\mathbb{E}} \|\widetilde{v}(t)\|_{L^{\frac{4}{3}}}^{\frac{4}{3}}
 \leq \widetilde{\mathbb{E}}\|\widetilde{v}(s)\|_{L^{ \frac{4}{3}}}^{\frac{4}{3}} +C\left(|t-s|+|t-s|^{\frac{1}{2}}\right),
$$
which implies
$$
\widetilde{\mathbb{E}}\|\widetilde{v}(s)\|_{L^{ \frac{4}{3}}}^{\frac{4}{3}}\geq \limsup _{t\downarrow s}\widetilde{\mathbb{E}} \|\widetilde{v}(t)\|_{L^{\frac{4}{3}}}^{\frac{4}{3}} \geq \liminf _{t\downarrow s}\widetilde{\mathbb{E}} \|\widetilde{v}(t)\|_{L^{\frac{4}{3}}}^{\frac{4}{3}}\geq \widetilde{\mathbb{E}} \|\widetilde{v}(s)\|_{L^{\frac{4}{3}}}^{\frac{4}{3}}.
$$
Then we get
$
\widetilde{\mathbb{E}} \|\widetilde{v}(s)\|_{L^{\frac{4}{3}}}^{\frac{4}{3}}= \lim  _{t\downarrow s}\widetilde{\mathbb{E}} \|\widetilde{v}(t)\|_{L^{\frac{4}{3}}}^{\frac{4}{3}}.
$
The proof for the case of $t\uparrow s$ is similar.  Therefore, we have proved that
$
\widetilde{\mathbb{E}} \|\widetilde{v}(s)\|_{L^{\frac{4}{3}}}^{\frac{4}{3}}= \lim  _{t\rightarrow s}\widetilde{\mathbb{E}} \|\widetilde{v}(t)\|_{L^{\frac{4}{3}}}^{\frac{4}{3}},
$
which together with the uniform bound \eqref{4.30} implies the desired result.  The proof of Lemma \ref{lem4.7} is completed.
\end{proof}

\begin{lemma}[Pathwise uniqueness]\label{lem4.8}
For any $T>0$, suppose that $(n_1,c_1,u_1)$ and $(n_2,c_2,u_2)$ are two martingale solutions to \eqref{KS-SNS} under the stochastic basis $(\widetilde{\Omega},\widetilde{\mathcal {F}},(\widetilde{\mathcal {F}}_t)_{t\in [0,T]},\widetilde{\mathbb{P}})$ with respect to the same initial data $(u_0,c_0,u_0)$. Then we have
$$
\widetilde{\mathbb{P}}\left\{(n_1,c_1,u_1)(t)=(n_2,c_2,u_2)(t),~\forall t\in [0,T]  \right\}=1.
$$
\end{lemma}

\begin{proof}
 For simplicity, we set
$$
(\bar{\bar{n}} ,\bar{\bar{c}},\bar{\bar{u}})=(n_1- {n_2},c_1- c_2,u_1-u_2)~~\textrm{and}~~\bar{\bar{v}}= v_1-v_2 ,~ v_i= \nabla \wedge u_i,~ i=1,2.
$$
For each $R>0$, define
$
\textbf{t} ^R= \textbf{t}_1^R\wedge \textbf{t}_2^R,
$
where
\begin{equation*}
\begin{split}
\textbf{t}_i^R:=& T \wedge\inf \left \{t>0;~ \sup_{s\in [0,t]}\|n_i\|_{L^2}^2\vee \int_0^t \|n_i\|_{L^2}^2 \textrm{d}r \vee\sup_{s\in [0,t]}\|c_i\|_{H^1}^2\vee \int_0^t \|c_i\|_{H^2}^2 \textrm{d}r\right.\\
&\quad \quad\quad  \left. \vee\sup_{s\in [0,t]}\|u_i\|_{L^2}^2\vee \int_0^t \|u_i\|_{L^2}^2 \textrm{d}r \vee \int_0^t \|v_i\|_{H^\alpha}^2 \textrm{d}r \geq R \right\} ,\quad i=1,2.
\end{split}
\end{equation*}
Then by \eqref{4.3}, \eqref{4.29} and \eqref{4.30},  we see that
$\textbf{t} ^R\rightarrow T $ as $R\rightarrow\infty$, $\widetilde{\mathbb{P}}$-a.s.
We shall prove the result by considering the two cases of $\alpha\in (\frac{1}{2},1]$ and $\alpha=\frac{1}{2}$, respectively.

\textbf{The case of $\alpha\in (\frac{1}{2},1]$.} Consider the functional
$$
\mathcal {E}(t)  := \|(\bar{\bar{n}},\bar{\bar{c}},\bar{\bar{u}},\nabla\bar{\bar{c}} ,\bar{\bar{v}} ) \|_{\textbf{L}^2}^2~~~ \textrm{and}~~~
  \mathcal {F}^\alpha (t)  :=   \|(\nabla\bar{\bar{n}},\nabla\bar{\bar{c}} ,(-\Delta)^\frac{\alpha}{2}\bar{\bar{u}},\Delta\bar{\bar{c}},(-\Delta)^\frac{\alpha}{2}\bar{\bar{v}} ) \|_{\textbf{L}^2}^2.
$$
Applying the chain rule to  {$  \|\bar{\bar{{n}}}(t)\|_{L^2}^2 $  and $ \|\bar{\bar{c}}(t)\|_{L^2}^2 $}, respectively, it is standard to derive that for any $\eta >0$
\begin{equation}\label{4.4.6}
\begin{split}
&\|\bar{\bar{{n}}}(t)\|_{L^2}^2+(2-\eta)\int_0^t \|\nabla\bar{\bar{{n}}} \|_{L^2}^2\textrm{d}r\leq C\int_0^t F_1(r)\mathcal {E}(r)  \textrm{d}r + \eta \int_0^t \|\Delta\bar{\bar{c}} \|_{L^2}^2\textrm{d}r,
\end{split}
\end{equation}
\begin{equation}\label{4.4.7}
\begin{split}
& \|\bar{\bar{c}}(t)\|_{L^2}^2 +(2-\eta)\int_0^t \|\nabla\bar{\bar{c}} \|_{L^2}^2\textrm{d}r\leq C\int_0^t  F_2(r) \mathcal {E}(r) \textrm{d}r,
\end{split}
\end{equation}
where
\begin{equation*}
\begin{split}
&F_1(t)=1+ \|n_1 \|_{L^2}^2\|\nabla n_1 \|_{L^2}^2+\|\nabla c_1 \|_{L^2}^2 \|\Delta c _1 \|_{L^2}^2+\|  {n_2}  \|_{L^2}^2\| \nabla {n_2}  \|_{L^2}^2 ,\\
&F_2(t)=1+ \|\nabla c_1 \|_{L^2} \|\Delta c _1\|_{L^2} +  \|c_2\|_{L^2}^2+\|n_1 \|_{L^2}^2.
\end{split}
\end{equation*}
To estimate the term  $\|\Delta \bar{\bar{c}} (t)\|_{L^2} $ on the R.H.S. of \eqref{4.4.6}, we apply the chain rule to $ \|\nabla\bar{\bar{c}}(t)\|_{L^2}^2 $ to obtain
\begin{equation}\label{4.4.8}
\begin{split}
 \|\nabla\bar{\bar{c}}(t)\|_{L^2}^2  +(2-\eta)\int_0^t \|\Delta \bar{\bar{c}} \|_{L^2}^2\textrm{d}r\leq C\int_0^tF_3(r) \mathcal {E}(r) \textrm{d}r,
\end{split}
\end{equation}
where
$$F_3(t)=\|\nabla c_1\|_{L^2}^2\|\Delta c_1\|_{L^2}^2+ \|\nabla u_2\|_{L^2}^ 2 +  \| c_2\|_{H^2}^2+ \|n_1\|_{L^2}^2\|\nabla n_1\|_{L^2}^2+1.
$$
By applying  the It\^{o} formula  to  {$ \| \bar{\bar{u}}(t)\|_{L^2}^2$}, we infer that
\begin{equation}\label{4.4.9}
\begin{split}
\|\bar{\bar{u}}(t)\|_{L^2}^2  +2\int_0^t \| \bar{\bar{u}} \|_{\dot{H}^\alpha}^2\textrm{d}r \leq C\int_0^t (\|\nabla u_1\|_{L^2}  +1)\mathcal {E}(r)\textrm{d}r + 2 \int_0^t \left(\bar{\bar{u}},(f(r,u)-f(r,\bar{u})) \right)_{L^2}\textrm{d} W _r.
\end{split}
\end{equation}
Taking the operator $\nabla \wedge$ to $\eqref{KS-SNS}_3$ and utilizing the It\^{o} formula  to $  \|\bar{\bar{v}}(t) \|_{L^2}^2 $, we get
\begin{equation}\label{4.5.0}
\begin{split}
\|\bar{\bar{v}}(t)\|_{L^2}^2  +(2-\eta)\int_0^t \| \bar{\bar{v}} \|_{\dot{H}^\alpha}^2\textrm{d}r\leq& C\int_0^t \left(\| v_1\|_{\dot{H}^{ \alpha}}^2+1\right)\mathcal {E}(r) \textrm{d}r+ \eta\int_0^t \|\nabla\bar{\bar{n}}\|_{L^2}^2\textrm{d}r\\
 & + 2 \int_0^t (\bar{\bar{v}},\nabla \wedge(f(r,u)-f(r,\bar{u})) )_{L^2}\textrm{d} W _r,
\end{split}
\end{equation}
where we used the estimate
$$
\|\bar{\bar{u}}\cdot \nabla v_1\|_{\dot{H}^{-\alpha}}\leq C \|\bar{\bar{u}}\|_{H^1} \| v_1\|_{\dot{H}^{ \alpha}} \leq C \|\bar{\bar{v}}\|_{L^2} \| v_1\|_{\dot{H}^{ \alpha}},
$$
by taking $p=r=2$ in \eqref{2.5} of Lemma \ref{lem1.6}. Putting the estimates \eqref{4.4.6}-\eqref{4.5.0} together and choosing $\eta>0$ small enough, we obtain
\begin{equation}\label{ddd}
\begin{split}
\mathcal {E}(t) + \int_0^t \mathcal {F}^\alpha(r)  \textrm{d}r &\leq C\int_0^t H(r)\mathcal {E}(r)  \textrm{d}r+  2 \sum_{k\geq 1}\int_0^t \mathcal {G}_k(r)\textrm{d} W^k  _r ,
\end{split}
\end{equation}
where
\begin{equation*}
\begin{split}
H (t)=& 1+  \|n_1 \|_{L^2}^2\|\nabla n_1 \|_{L^2}^2+\|\nabla c_1 \|_{L^2}^2 \|\Delta c _1 \|_{L^2}^2+\|  {n_2}  \|_{L^2}^2\| \nabla {n_2}  \|_{L^2}^2+ \|\nabla c_1 \|_{L^2} \|\Delta c _1\|_{L^2}\\
& +  \|c_2\|_{L^2}^2+\|n_1 \|_{L^2}^2 + \|\nabla u_2\|_{L^2}^ 2 +  \| c_2\|_{H^2}^2+ \| v_1\|_{\dot{H}^{ \alpha}}^2+\|\nabla u_1\|_{L^2} ,\\
 \mathcal {G}_k(t)=& (\bar{\bar{u}},(f(t,u_1)-f(t,u_2))e_k)_{L^2}+ (\bar{\bar{v}},\nabla \wedge(f(t,u_1)-f(t,u_2))e_k)_{L^2}.
\end{split}
\end{equation*}
By the Gronwall Lemma, it follows from \eqref{ddd}  and the definition of $\textbf{t}_i^R$ that
\begin{equation*}
\begin{split}
\mathcal {E}(t \wedge \textbf{t}_i^R)
&\leq   {Ce^{CR^2}}\sup_{r \in [0,t \wedge \textbf{t}_i^R]}\bigg|  \sum_{k\geq 1}\int_0^t \mathcal {G}_k(r)\textrm{d} W^k  _r \bigg|.
\end{split}
\end{equation*}
Applying the BDG inequality, we get from the last inequality that
\begin{align}\label{4.5.1}
\widetilde{\mathbb{E}}\sup _{r\in[0,t \wedge \textbf{t}_i^R]} \mathcal {E}(r)\leq&   {Ce^{CR^2}}\widetilde{\mathbb{E}} \Bigg (  \sum_{k\geq 1}\int_0^{t \wedge \textbf{t}_i^R}\Big(\|\bar{\bar{u}}\|_{L^2}^2 \|(f(r,u_1)-f(r,u_2))e_k\|_{L^2}^2\nonumber \\
& + \|\bar{\bar{v}}\|_{L^2}^2 \|\nabla \wedge (f(r,u_1)-f(r,u_2))e_k\|_{L^2}^2\Big)\textrm{d}r\Bigg)^{\frac{1}{2}}\\
\leq & \frac{1}{2} \widetilde{\mathbb{E}}\sup_{r\in[0,t \wedge \textbf{t}_i^R]}\left( \|\bar{\bar{u}}\|_{L^2}^2+   \|\bar{\bar{v}}\|_{L^2}^2\right) + {Ce^{CR^2}}\int_0^{t \wedge \textbf{t}_i^R  }\left (\|\bar{\bar{u}}(r)\|_{L^2}^2+ \|\bar{\bar{v}}(r)\|_{L^2}^2\right)    \textrm{d}r.\nonumber
\end{align}
Absorbing the first term on the R.H.S. of \eqref{4.5.1} by the term on the L.H.S.,  {it follows from the Gronwall Lemma  that
$$
\widetilde{\mathbb{E}}\sup _{r\in[0, t\wedge \textbf{t}_i^R]} \mathcal {E}(r)\equiv0,~~~\forall t\in [0,T].
$$}
By taking the limit as $R\rightarrow\infty$, we get $\mathcal {E}(t)\equiv0$ for all $t \in [0,T]$, $\widetilde{\mathbb{P}}$-a.s.

\textbf{The case of $\alpha=\frac{1}{2}$.} For simplicity, we set
$$
\widetilde{\mathcal {E}}(t)  := \|(\bar{\bar{n}},\bar{\bar{c}},\bar{\bar{u}},\nabla\bar{\bar{c}} ,(-\Delta)^{-\frac{1}{8}}\bar{\bar{v}} )  \|_{\textbf{L}^2}^2 ~~ \textrm{and} ~~ \widetilde{\mathcal {F}} (t)  := \|(\nabla\bar{\bar{n}},\nabla\bar{\bar{c}} ,(-\Delta)^\frac{1}{4}\bar{\bar{u}},\Delta\bar{\bar{c}},(-\Delta)^\frac{1}{8}\bar{\bar{v}} )  \|_{\textbf{L}^2}^2.
$$
First, we take the Littlewood-Paley operators $\dot{\triangle}_q$ to both sides of the vorticity equation and applying the It\^{o} formula  to  {$ \|\dot{\triangle}_q \bar{\bar{v}}(t)\|_{L^2}^2 $}.  Then, after multiplying both sides of the resulting equation by $2^{-\frac{1}{2}q }$ and summing up with respect to $q\geq -1$, we get
\begin{align}\label{4.5.4}
\| \bar{\bar{v}}(t)\|_{\dot{H}^{-\frac{1}{4}}}^2 +2\int_0^t\| \bar{\bar{v}} \|_{\dot{H}^{ \frac{1}{4}}}^2 \textrm{d}r
&\leq 2\int_0^t \| \bar{\bar{v}} \|_{\dot{H}^{ \frac{1}{4}}}   \| (\bar{\bar{u}}\cdot\nabla)\cdot v_1 \|_{\dot{H}^{-\frac{3}{4}}} \textrm{d}r +2\int_0^t\| \bar{\bar{v}} \|_{\dot{H}^{- \frac{1}{4}}}\| \nabla\wedge (\bar{\bar{n}}\nabla \phi)\|_{\dot{H}^{-\frac{1}{4}}} \textrm{d}r\nonumber\\
 &+ 2\int_0^t\| \bar{\bar{v}} \|_{\dot{H}^{ \frac{1}{4}}} \left\|\{2^{-\frac{3}{4}q }\|[\dot{\triangle}_q,u_1 \cdot \nabla ]\bar{\bar{v}}  \|_{L^2 }\}_{q \in\mathbb{Z}}\right\|_{ \ell^2 } \textrm{d}r\\
 &+ \int_0^t\| \nabla \wedge (f(r,u_1)-f(r,u_2))\|_{L_2(U;\dot{H}^{-\frac{1}{4}})}^2 \textrm{d}r\nonumber\\
 &+ 2\sum_{k\geq 1}\sum_{q\in \mathbb{Z}}\int_0^t2^{-\frac{1}{2}q }(\dot{\triangle}_q\bar{\bar{v}},\dot{\triangle}_q \nabla \wedge (f(r,u_1)-f(r,u_2))e_k)_{L^2 }\textrm{d} W^k  _r \nonumber\\
&:= J_1+J_2+J_3+J_4+J_5 .\nonumber
\end{align}
For $J_1$, we get by taking $p=r=2$ in \eqref{2.6} of Lemma \ref{lem1.6} that
\begin{equation*}
\begin{split}
 J_1 \leq \frac{1}{2} \int_0^t \| \bar{\bar{v}} \|_{\dot{H}^{ \frac{1}{4}}} ^2 \textrm{d}r +   C \int_0^t   \| \bar{\bar{v}} \|_{\dot{H}^{-\frac{1}{4}}} ^2\| v_1 \|_{\dot{H}^{ \frac{1}{2}}} ^2 \textrm{d}r.
\end{split}
\end{equation*}
For $J_3$, it follows from the commutator estimate   {(cf. \cite[Lemma 2.97]{bahouri2011fourier})} that
\begin{equation*}
\begin{split}
 J_3  &\leq\frac{1}{2} \int_0^t \| \bar{\bar{v}} \|_{\dot{H}^{ \frac{1}{4}}} ^2 \textrm{d}r +C\sum_{q \in\mathbb{Z}}\int_0^t 2^{-\frac{3}{2}q }\|[\dot{\triangle}_q,u_1 \cdot \nabla ]\bar{\bar{v}}  \|_{L^2 }^2 \textrm{d}r\\
& \leq \frac{1}{2} \int_0^t\| \bar{\bar{v}} \|_{\dot{H}^{ \frac{1}{4}}} ^2 \textrm{d}r +C \int_0^t \|u_1  \|_{H^{\frac{3}{2}}}^2\|\bar{\bar{v}} \|_{H^{-\frac{1}{4}}}^2 \textrm{d}r.
\end{split}
\end{equation*}
For $J_2$ and $J_4$, we get from the Lemma \ref{lem1.6} that
\begin{equation*}
\begin{split}
 J_2 +J_4   &\leq C\|\nabla\phi\|_{L^\infty}\int_0^t\left(\| \bar{\bar{v}} \|_{\dot{H}^{- \frac{1}{4}}}^2+\| \bar{\bar{n}} \|_{L^2}^2+\eta\|\nabla \bar{\bar{n}} \|_{L^2}^2\right ) \textrm{d}r+C\int_0^t\| \bar{\bar{v}} \|_{\dot{H}^{-\frac{1}{4}}}^2 \textrm{d}r.
\end{split}
\end{equation*}
Plugging the estimates for $J_1\sim J_4$ into \eqref{4.5.4}, we infer that
\begin{equation}
\begin{split}
\| \bar{\bar{v}}(t)\|_{\dot{H}^{-\frac{1}{4}}}^2 + \int_0^t\| \bar{\bar{v}} \|_{\dot{H}^{ \frac{1}{4}}}^2 \textrm{d}r\leq C \int_0^t  \left[\| \bar{\bar{v}} \|_{\dot{H}^{-\frac{1}{4}}} ^2\left(\|u_1 \|_{H^{\frac{3}{2}}}^2+\| v_1 \|_{\dot{H}^{ \frac{1}{2}}} ^2+1\right)+  \eta\|\nabla \bar{\bar{n}} \|_{L^2}^2 \right]\textrm{d}r+ J_5 (t).
\end{split}
\end{equation}
Thanks to the GN inequality in the form of
$
\| \bar{\bar{u}} \|_{L^4}\leq C\| \bar{\bar{u}} \|_{L^2}^{\frac{1}{3}}\| \bar{\bar{v}} \|_{\dot{H}^{-\frac{1}{4}}}^{\frac{2}{3}},
$
we have
\begin{equation}\label{4.56}
\begin{split}
 & \| \bar{\bar{n}}(t)\|_{L^2}^2 + \int_0^t\| \nabla \bar{\bar{n}} \|_{L^2}^2 \textrm{d}r\leq \int_0^t G_2 (r) \mathcal {E}(r) \textrm{d}r + \eta \int_0^t \| \nabla \bar{\bar{n}} \|_{L^2}^2 \textrm{d}r+ \eta \int_0^t \|\Delta \bar{\bar{c}} \|_{L^2}^2 \textrm{d}r,
\end{split}
\end{equation}
and
\begin{equation}\label{4.57}
\begin{split}
 & \| \nabla \bar{\bar{c}}(t)\|_{L^2}^2 + \int_0^t\| \Delta \bar{\bar{c}}  \|_{L^2}^2 \textrm{d}r\leq \int_0^t G_3 (r) \mathcal {E}(r) \textrm{d}r + \eta \int_0^t \| \Delta \bar{\bar{c}} \|_{L^2}^2 \textrm{d}r,
\end{split}
\end{equation}
where
\begin{equation*}
\begin{split}
G_2 (r)&= \|n_1\|_{L^2} ^{\frac{3}{2}}\|\nabla n_1\|_{L^2} ^{\frac{3}{2}}+ \|n_2\|_{L^2} ^2\|\nabla n_2\|_{L^2} ^2+\|\Delta c_1\|_{L^2} ^2+1,\\
G_3 (r)&=  \|\nabla c_1\|_{L^2} ^{\frac{3}{2}}\|\Delta c_1\|_{L^2} ^{\frac{3}{2}} +\|\nabla u_2\|_{L^2} ^2+\|c_2\|_{L^\infty} ^2+\|n_1\|_{L^2} ^2\|\nabla n_1\|_{L^2} ^2 +1.
\end{split}
\end{equation*}
By applying  the It\^{o} formula  to  {$ \| \bar{\bar{u}}(t)\|_{L^2}^2$}, there holds
\begin{equation} \label{4.58}
\begin{split}
 &\|\bar{\bar{u}}(t)\|_{L^2}^2  +2\int_0^t \| \bar{\bar{u}} \|_{\dot{H}^ \frac{1}{2}}^2\textrm{d}r \\
 &\quad\leq C\int_0^t \left(\|\nabla u_1\|_{L^2} ^{\frac{3 }{2}} +1\right)\mathcal {E}(r)\textrm{d}r + 2 \int_0^t (\bar{\bar{u}},(f(r,u)-f(r,\bar{u}))\textrm{d} W  _r)_{L^2} .
\end{split}
\end{equation}
Putting the estimates \eqref{4.56}-\eqref{4.58} together, we obtain
\begin{equation}\label{4.59}
\begin{split}
\mathcal {E}(t) + \int_0^t \mathcal {F}^ \frac{1}{2}(r)  \textrm{d}r &\leq C\int_0^t G(r)\mathcal {E}(r)  \textrm{d}r+  2 \sum_{k\geq 1}\int_0^t \textbf{T}_k(r)\textrm{d} W^k  _r ,
\end{split}
\end{equation}
where
\begin{equation*}
\begin{split}
G(r)=& \|n_1\|_{L^2} ^{\frac{3}{2}}\|\nabla n_1\|_{L^2} ^{\frac{3}{2}}+ \|n_2\|_{L^2} ^2\|\nabla n_2\|_{L^2} ^2+\|\Delta c_1\|_{L^2} ^2 + \|\nabla c_1\|_{L^2} ^{\frac{3}{2}}\|\Delta c_1\|_{L^2} ^{\frac{3}{2}}+\|\nabla u_2\|_{L^2} ^2 \\
&+\|c_2\|_{L^\infty} ^2+\|n_1\|_{L^2} ^2\|\nabla n_1\|_{L^2} ^2 +1,\\
 \textbf{T}_k(r)=&\sum_{q\in \mathbb{Z}}2^{-\frac{1}{2}q }(\dot{\triangle}_q\bar{\bar{v}},\dot{\triangle}_q \nabla \wedge (f(r,u_1)-f(r,u_2))e_k)_{L^2 }+(\bar{\bar{u}},(f(r,u)-f(r,\bar{u})) e_k)_{L^2}.
\end{split}
\end{equation*}
By using the Young inequality, we infer that
\begin{equation}\label{4.60}
\begin{split}
 \sum_{k\geq 1}|\textbf{T}_k |^2
 &\leq C \sum_{k\geq 1}\| \bar{\bar{v}} \|_{\dot{H}^{-\frac{1}{4}}}^2 \|\nabla \wedge (f(r,u_1)-f(r,u_2))e_k\|_{\dot{H}^{-\frac{1}{4}}}^2+  C\|\bar{\bar{u}} \|_{L^2}^4\\
 & \leq C \left(\| \bar{\bar{v}} \|_{\dot{H}^{-\frac{1}{4}}}^4+   \|\bar{\bar{u}} \|_{L^2}^4\right).
\end{split}
\end{equation}
Applying the Gronwall lemma to \eqref{4.59}, we get from the definition of $\textbf{t}_i^R$ and \eqref{4.60} that
\begin{equation*}
\begin{split}
\widetilde{\mathbb{E}}\sup _{r\in[0,t \wedge \textbf{t}_i^R]}\widetilde{\mathcal {E}}(r) + \int_0^{t\wedge \textbf{t}_i^R} \widetilde{\mathcal {F}} (r)  \textrm{d}r &\leq C \widetilde{\mathbb{E}}\sup _{r\in[0,t \wedge \textbf{t}_i^R]}\left (\exp\{C\int_0^{r} G(\tau)  \textrm{d} \tau \} \left |\sum_{k\geq 1}\int_0^r \textbf{T}_k(\varsigma)\textrm{d} W^k_\varsigma \right|\right)\\
&\leq   {Ce^{C(R^2+1)}}\widetilde{\mathbb{E}}\sup _{r\in[0,t \wedge \textbf{t}_i^R]}\left |\sum_{k\geq 1}\int_0^{r} \textbf{T}_k(\varsigma)\textrm{d} W^k  _\varsigma \right |\\
& \leq \frac{1}{2}\widetilde{\mathbb{E}}\sup _{r\in[0,t \wedge \textbf{t}_i^R]}\widetilde{\mathcal {E}}(r) + {Ce^{C(R^2+1)}} \int_0^{t\wedge \textbf{t}_i^R} \widetilde{\mathcal {E}}(r) \textrm{d}r,
\end{split}
\end{equation*}
which implies that for any $T>0$
$
\widetilde{\mathbb{E}}\sup _{r\in[0,t\wedge \textbf{t}_i^R]}\widetilde{\mathcal {E}}(r)=0$, $\forall t\in [0,T]$ and $R> 0$.
By taking $R\rightarrow\infty$, we get $\widetilde{\mathcal {E}}(t)=0$ for all $t\in [0,T]$, $\widetilde{\mathbb{P}}$-a.s. The proof of Lemma  \ref{lem4.8} is now completed.
\end{proof}

\begin{proof}[\emph{\textbf{Proof of Theorem \ref{main}}}]

By utilizing Lemma \ref{lem4.7} and Lemma \ref{lem4.8}, we can establish the existence and uniqueness of a global pathwise solution for system \eqref{KS-SNS} by applying the classical Watanable-Yamada Theorem in \cite{watanabe1971uniqueness}, which is based on an elementary characterization of convergence in probability (cf. \cite[Lemma 1.1]{gyongy1980stochastic}). This characterization has been previously employed to study other SPDEs, such as the stochastic Euler equation \cite[P. 112-113]{glatt2014local} and the stochastic dispersive equation \cite[Proposition 2.11]{zhang2020local}. Since the argument is very similar to the aforementioned works, we shall omit the details here. Thus, the proof of the main result is now completed.
\end{proof}

\section{Appendix}
Let $s>2$,  and $\textbf{B}(\cdot)$, $\textbf{F}^\epsilon(\cdot)$ be defined in \eqref{3.2}. Then for any $\textbf{u}=(u,v,h)$, $\textbf{u}_i=(u_i,v_i,h_i) \in \cap_{s>0}\textbf{H}^s(\mathbb{R}^2)$, $i=1,2$, the following basic properties hold:
\begin{equation} \label{A.1}
|(\textbf{B}(\textbf{u}),\textbf{u})_{\textbf{H}^s}| \leq C \|\nabla\textbf{u}\|_{L^\infty}\|\textbf{u}\|_{\textbf{H}^s}^2 ,\tag{A.1}
\end{equation}
\begin{equation} \label{A.2}
| (\textbf{B}(\textbf{u}_1)-\textbf{B}(\textbf{u}_2),\textbf{u}_1-\textbf{u}_2)_{\textbf{H}^s} |  \leq C (\|\textbf{u}_1\|_{\textbf{H}^{s+1}}+\|\textbf{u}_2\|_{\textbf{H}^{s+1}}) \|\textbf{u}_1-\textbf{u}_2\|_{\textbf{H}^s}^2, \tag{A.2}
\end{equation}
\begin{equation} \label{A.3}
|(\textbf{F}^\epsilon(\textbf{u} ),\textbf{u} )_{\textbf{H}^s}|\leq C( \epsilon,\phi,\|c_0\|_{L^\infty})  \|\textbf{u}\|_{W^{1,\infty}}\|\textbf{u}\|_{\textbf{H}^s}^2,\tag{A.3}
\end{equation}
\begin{equation} \label{A.4}
|(\textbf{F}^\epsilon(\textbf{u}_1)-\textbf{F}^\epsilon(\textbf{u}_2)
,\textbf{u}_1-\textbf{u}_2)_{\textbf{H}^s}| \leq   C(\epsilon)(\|\textbf{u}_1\|_{\textbf{H}^s}+\|\textbf{u}_2\|_{\textbf{H}^s}) \|\textbf{u}_1-\textbf{u}_2\|_{\textbf{H}^s}^2.\tag{A.4}
\end{equation}

\begin{proof}
To deal with \eqref{A.1}, note that
$
\textbf{P} \Lambda ^s=\Lambda ^s\textbf{P}$ and $(\textbf{P}u,v)_{L^2}=(u,\textbf{P}v)_{L^2}$,
where $\Lambda^s= (1-\Delta)^{s/2}$ denotes the Bessel potentials, we have
\begin{equation*}
\begin{split}
(\textbf{B}(\textbf{u}),\textbf{u})_{\textbf{H}^s}=&([\Lambda^s, u\cdot \nabla] n ,\Lambda^sn)_{L^2}+(u\cdot \nabla \Lambda^s  n ,\Lambda^sn)_{L^2}+([\Lambda^s, u\cdot \nabla] c ,\Lambda^sc)_{L^2}\\
&+(u\cdot \nabla \Lambda^s  c ,\Lambda^sc)_{L^2} +([\Lambda^s, u\cdot \nabla] u ,\Lambda^su)_{L^2}+(u\cdot \nabla \Lambda^s  u ,\Lambda^su)_{L^2}.
\end{split}
\end{equation*}
By using the divergence-free condition, we have
$$
(u\cdot \nabla \Lambda^s  n ,\Lambda^sn)_{L^2}=(u\cdot \nabla \Lambda^s  c ,\Lambda^sc)_{L^2}=(u\cdot \nabla \Lambda^s  u,\Lambda^su)_{L^2}=0.
$$
In virtue of the commutator estimate (cf. \cite[Lemma 2.97]{bahouri2011fourier}), we have
\begin{equation*}
\begin{split}
([\Lambda^s, u\cdot \nabla] n ,\Lambda^sn)_{L^2}&\leq  C(\|\Lambda^su\|_{L^2}\|\nabla n\|_{L^\infty}+\|\nabla u\|_{L^\infty}\|\Lambda^{s-1}\nabla u\|_{L^2})\|\Lambda^sn\|_{L^2}\\
&\leq C(\|\nabla u\|_{L^\infty}+\|\nabla n\|_{L^\infty})\|u\|_{H^s}\|n\|_{H^s}.
\end{split}
\end{equation*}
Similarly, we infer that
\begin{equation*}
\begin{split}
&([\Lambda^s, u\cdot \nabla] c ,\Lambda^sc)_{L^2}\leq   C(\|\nabla u\|_{L^\infty}+\|\nabla c\|_{L^\infty})\|u\|_{H^s}\|c\|_{H^s},\\
&([\Lambda^s, u\cdot \nabla] u ,\Lambda^su)_{L^2}\leq   C \|\nabla u\|_{L^\infty} \|u\|_{H^s}^2.
\end{split}
\end{equation*}
Putting the above estimates together leads to \eqref{A.1}.

To prove \eqref{A.2}, we set $\textbf{u}^{1,2}=(n^{1,2},c^{1,2},u^{1,2}):=\textbf{u}_1-\textbf{u}_2$. Note that
\begin{equation}\label{5.1}
\begin{split}
&(\textbf{B}(\textbf{u}_1)-\textbf{B}(\textbf{u}_2),\textbf{u}_1-\textbf{u}_2)_{\textbf{H}^s}
\\&\quad =
([\Lambda^s ,u^{1,2}\cdot \nabla] n_1, \Lambda^sn^{1,2})_{L^2}
+(u^{1,2}\cdot \nabla \Lambda^sn_1, \Lambda^sn^{1,2})_{L^2}
+ ([\Lambda^s,u_2\cdot \nabla] n^{1,2}, \Lambda^sn^{1,2})_{L^2} \\
&\quad\quad+ (u_2\cdot \nabla\Lambda^s  n^{1,2} , \Lambda^sn^{1,2})_{L^2}
+([\Lambda^s ,u^{1,2}\cdot \nabla] c_1, \Lambda^sc^{1,2})_{L^2}
+(u^{1,2}\cdot \nabla \Lambda^sc_1, \Lambda^sc^{1,2})_{L^2}\\
&\quad\quad
+ ([\Lambda^s,u_2\cdot \nabla] c^{1,2}, \Lambda^sc^{1,2})_{L^2}
+ (u_2\cdot \nabla\Lambda^s  c^{1,2} , \Lambda^sc^{1,2})_{L^2}
+([\Lambda^s ,u^{1,2}\cdot \nabla] u_1, \Lambda^su^{1,2})_{L^2}\\
&\quad\quad
+(u^{1,2}\cdot \nabla \Lambda^su_1, \Lambda^su^{1,2})_{L^2}
+ ([\Lambda^s,u_2\cdot \nabla] u^{1,2}, \Lambda^su^{1,2})_{L^2}
+ (u_2\cdot \nabla\Lambda^s  u^{1,2} , \Lambda^su^{1,2})_{L^2}\\
&\quad = I_1+\cdots +I_{12}.
\end{split}
\end{equation}
Since $\textrm{div} u_1=\textrm{div} u_2=\textrm{div} u^{1,2}=0$, there holds
$$
I_4=I_8=I_{12}=0.
$$
By using the commutator estimate and the embedding $H^s(\mathbb{R}^2)\hookrightarrow W^{1,\infty}(\mathbb{R}^2)$ with $s>2$, we get
\begin{equation*}
\begin{split}
|I_1|&\leq \|[\Lambda^s ,u^{1,2}\cdot \nabla] n_1\|_{L^2}\| \Lambda^sn^{1,2}\|_{L^2}\\
&\leq C(\|\Lambda^su^{1,2}\|_{L^2} \|\nabla n_1\|_{L^\infty}+ \|\nabla u^{1,2}\|_{L^\infty} \|\Lambda^{s-1}\nabla n_1\|_{L^2})\|n^{1,2}\|_{H^s}\\
&\leq C \|n_1\|_{H^s}\|u^{1,2}\|_{H^s}\|n^{1,2}\|_{H^s},\\
|I_2|&\leq  \| u^{1,2}\|_{L^\infty}\| \nabla \Lambda^sn_1\|_{L^2}  \| \Lambda^sn^{1,2}\|_{L^2} \leq C\| n_1\|_{H^{s+1}} \|n^{1,2}\|_{H^s}\| u^{1,2}\|_{H^s},
\end{split}
\end{equation*}
and
\begin{equation*}
\begin{split}
|I_3|&\leq   \|[\Lambda^s,u_2\cdot \nabla] n^{1,2}\|_{L^2}\| \Lambda^sn^{1,2}\|_{L^2}\\
&\leq C(\|\Lambda^su_2\|_{L^2} \|\nabla n^{1,2}\|_{L^\infty}+ \|\nabla u_2\|_{L^\infty} \|\Lambda^{s-1}\nabla n^{1,2}\|_{L^2})\| n^{1,2}\|_{H^s}\\
&\leq C \| u_2\|_{H^s} \| n^{1,2}\|_{H^s}^2.
\end{split}
\end{equation*}
Similar to the estimates for $I_1\sim I_3$, one can deduce that
\begin{equation*}
\begin{split}
|I_5|,|I_6|&\leq C \|c_1\|_{H^{s+1}}\|u^{1,2}\|_{H^s}\|c^{1,2}\|_{H^s},\\
|I_7|&\leq  C \| u_2\|_{H^s} \| c^{1,2}\|_{H^s}^2,\\
|I_9|,|I_{10}|&\leq  C  \|u_1\|_{H^{s+1}}\|u^{1,2}\|_{H^s}^2,\\
|I_{11}|&\leq  C \| u_2\|_{H^s} \| u^{1,2}\|_{H^s}^2.
\end{split}
\end{equation*}
Putting the above estimates for terms $I_i$, $i=1,...,12$, into \eqref{5.1} leads to \eqref{A.2}.

Now let us deal with the estimates with respect to $\textbf{F}^\epsilon(\cdot)$. First note that
\begin{equation*}
\begin{split}
 (\textbf{F}^\epsilon(\textbf{u} ),\textbf{u} )_{\textbf{H}^s}=&-(\Lambda^s\textrm{div}\left(n  ( \nabla c *\rho^{\epsilon}) \right) ,\Lambda^s n)_{L^2}+(\Lambda^s(   n -  n^2),\Lambda^s n)_{L^2}\\
 &-(\Lambda^s(c(n *\rho^{\epsilon})),\Lambda^s c )_{L^2}+(\Lambda^s(\textbf{P}(n \nabla \phi)*\rho^{\epsilon}),\Lambda^s u)_{L^2}.
\end{split}
\end{equation*}
The terms on the R.H.S. of last inequality can be estimated as
\begin{equation*}
\begin{split}
|(\Lambda^s\textrm{div}\left(n (\nabla c *\rho^{\epsilon}) \right) ,\Lambda^s n)_{L^2}|&\leq \|\Lambda^s\textrm{div}\left(n (\nabla c *\rho^{\epsilon}) \right)\|_{L^2}\|\Lambda^s n\|_{L^2}\\
& \leq\frac{C}{\epsilon^2} \|n (\nabla c *\rho^{\epsilon}) \|_{H^{s-1}} \|n\|_{H^s}\\
& \leq C(\epsilon) (\|n \|_{L^\infty}\| \nabla c\|_{H^{s-1}}+\|n \|_{H^{s-1}}\| \nabla c\|_{L^\infty}) \|n\|_{H^s}\\
& \leq C(\epsilon) (\|n \|_{L^\infty} + \| \nabla c\|_{L^\infty})  (\|n\|_{H^s}^2+\|c\|_{H^s}^2),
\end{split}
\end{equation*}
\begin{equation*}
\begin{split}
|(\Lambda^s(   n -  n^2),\Lambda^s n)_{L^2}| \leq C (\|n\|_{H^s}+\|n^2\|_{H^s})  \|n\|_{H^s} \leq C (1+\|n \|_{L^\infty})  \|n\|_{H^s}^2,
\end{split}
\end{equation*}
\begin{equation*}
\begin{split}
|(\Lambda^s(\textbf{P}(n \nabla \phi)*\rho^{\epsilon}),\Lambda^s u)_{L^2}| &\leq C\| (n \nabla \phi)*\Lambda^s\rho^{\epsilon} \|_{L^2} \|u\|_{H^s}\\
&\leq C\| n \nabla \phi \|_{L^2} \|\Lambda^s\rho^{\epsilon} \|_{L^1}\|u\|_{H^s} \leq C(\epsilon,\phi)\| n \|_{L^2} \|u\|_{H^s},
\end{split}
\end{equation*}
and
\begin{equation*}
\begin{split}
|(\Lambda^s(c(n *\rho^{\epsilon})),\Lambda^s c )_{L^2}| &\leq C(\|c\|_{L^\infty} \|n *\rho^{\epsilon}\|_{H^s}+\|c\|_{H^s} \|n *\rho^{\epsilon}\|_{L^\infty}) \|c\|_{H^s}\\
&\leq C(\|c\|_{L^\infty} \|n \|_{H^s}+\|c\|_{H^s} \|n \|_{L^\infty}) \|c\|_{H^s}\\
&\leq C(\|c\|_{L^\infty} + \|n \|_{L^\infty}) (\|c\|_{H^s}^2+\|n \|_{H^s}^2).
\end{split}
\end{equation*}
Putting the above estimates together leads to \eqref{A.3}.

To prove \eqref{A.4}, we observe that
\begin{equation}\label{5.2}
\begin{split}
&|(\textbf{F}^\epsilon(\textbf{u}_1)-\textbf{F}^\epsilon(\textbf{u}_2)
,\textbf{u}_1-\textbf{u}_2)_{\textbf{H}^s}|\\
 &\leq  |(\Lambda^s \textrm{div} [n^{1,2} (\nabla  c_1 *\rho^{\epsilon})],\Lambda^s n^{1,2})_{L^2}| +|(\Lambda^s\textrm{div} [ n _2(\nabla c^{1,2}  *\rho^{\epsilon} )],\Lambda^s n^{1,2})_{L^2}|\\
&\quad + \|n^{1,2}\|_{H^s}^2+   |\Lambda^s ((n_1+n_2)n^{1,2}) ,\Lambda^s n^{1,2})_{L^2}|+ |(\Lambda^s [ c_2 (n^{1,2}*\rho^{\epsilon})],\Lambda^s c^{1,2})_{L^2}|\\
&\quad+ |(\Lambda^s [c^{1,2}(n_2*\rho^{\epsilon})] ,\Lambda^s c^{1,2})_{L^2}| + |(\Lambda^s (\textbf{P}(n ^{1,2} \nabla \phi)*\rho^{\epsilon} ) ,\Lambda^s u^{1,2})_{L^2}|\\
& := J_1+\cdots + J_7.
\end{split}
\end{equation}
For $J_1$, we have
\begin{equation*}
\begin{split}
J_1 &\leq \frac{C}{\epsilon}\|\nabla  c_1 *\rho^{\epsilon}\|_{H^{s}} \| n^{1,2}\|_{H^{s }}^2 \leq \frac{C}{\epsilon^2} \| c_1 \|_{H^{s}} \| n^{1,2}\|_{H^{s }}^2.
\end{split}
\end{equation*}
For $J_2$, we use the convolution inequality to derive
\begin{equation*}
\begin{split}
J_2&  \leq  \frac{C}{\epsilon} \| n _2\|_{H^{s}}\|\nabla c^{1,2}  *\Lambda^s\rho^{\epsilon} \|_{L^2}  \|n^{1,2}\|_{H^s} \\
&\leq \frac{C}{\epsilon } \| n _2\|_{H^{s }} \|\nabla c^{1,2}\|_{L^2} \|\Lambda^s\rho^{\epsilon}\|_{L^1} \|n^{1,2}\|_{H^s} \leq \frac{C}{\epsilon^{s+1}}\| n _2\|_{H^{s }}  \|c^{1,2}\|_{H^s} \|n^{1,2}\|_{H^s}.
\end{split}
\end{equation*}
For $J_4$ and $J_7$, we have
$$
J_4   \leq  C (\| n_1\|_{H^s}+\| n_2 \|_{H^s}) \|n^{1,2}\|_{H^s}^2,
$$
$$
J_7  \leq C(\epsilon,\phi)\|n^{1,2}\|_{H^s}\|u^{1,2}\|_{H^s}.
$$
For $J_5$ and $J_6$, we have
$$
J_5   \leq   C\|c_2 (n^{1,2}*\rho^{\epsilon})\|_{H^s}\|c^{1,2}\|_{H^s} \leq C\|c_2 \|_{H^s}\| n^{1,2} \|_{H^s}\|c^{1,2}\|_{H^s},
$$
$$
J_6   \leq   C\|c^{1,2}(n_2*\rho^{\epsilon})\|_{H^s}\|c^{1,2}\|_{H^s} \leq  C\|n_2 \|_{H^s}\| c^{1,2} \|_{H^s}^2.
$$
Plugging the estimates for $J_1\sim J_7$ into \eqref{5.2} leads to \eqref{A.4}.
\end{proof}

\section*{Conflict of interest statement}

The authors declared that they have no conflicts of interest to this work.

\section*{Data availability}

No data was used for the research described in the article.

\section*{Acknowledgements}

The authors thank the anonymous referees for their constructive comments and suggestions which improved the quality of this article significantly. This work was partially supported by the National Natural Science Foundation of China (No. 12231008).

\bibliographystyle{plain}%
\bibliography{SCNS}

\end{document}